\documentclass[11pt]{article}
\usepackage{amssymb,amsmath,amsfonts,amsthm,mathtools,color}

\usepackage{mathrsfs}
\usepackage{dsfont}
 \usepackage{bm} 

\usepackage{comment} 
\usepackage[title,titletoc,header]{appendix}
\usepackage{graphicx,subfig}
\usepackage{float} 

\usepackage{empheq} 

\usepackage{paralist}

\usepackage{tikz}
\usetikzlibrary{arrows,positioning,shapes.geometric}

\usepackage[linesnumbered, ruled,vlined]{algorithm2e}
\SetKwInput{KwInput}{Input}                

\graphicspath{ {./figures/} }

\usepackage{indentfirst}
\usepackage{multicol}
\usepackage{booktabs}
\usepackage{url}
\usepackage[outdir=./]{epstopdf} 

\usepackage[shortlabels]{enumitem}

\usepackage{hyperref}
\hypersetup{
    colorlinks=true, 
    linktoc=all,     
    linkcolor=blue,  
}
\numberwithin{equation}{section}

\newtheorem{Theorem}{Theorem}[section]
\newtheorem{Lemma}[Theorem]{Lemma}
\newtheorem{Proposition}[Theorem]{Proposition}
\newtheorem{Corollary}[Theorem]{Corollary}
\newtheorem{Assumption}{H.\!\!}

\theoremstyle{definition}
\newtheorem{Definition}{Definition}[section]

\newtheorem{Example}{Example}[section]

\theoremstyle{remark}
\newtheorem{Remark}{Remark}[section]

 \def\p{\partial} \def\nb{\nonumber}
\def\to{\rightarrow}
 \def\ol{\overline}    
\def\Om{\Omega}   
 
\newcommand{\q}{\quad}   \newcommand{\qq}{\qquad}

  \def\fa{\forall}

\def\eps{\varepsilon}

\def\ms{\medskip}

\def \la{\langle} \def\ra{\rangle}

\def\cF{\mathcal{F}}
\def\cG{\mathcal{G}}
\def\cH{\mathcal{H}}

\def\cL{\mathcal{L}}

\def\cS{\mathcal{S}}

\def\d{{\mathrm{d}}}

\def\sE{{\mathbb{E}}}

\def\sN{{\mathbb{N}}}
\def\sP{\mathbb{P}}

\def\sR{{\mathbb R}}
\def\sS{{\mathbb{S}}}

\newcommand{\tr}{\textnormal{tr}}

\DeclareMathOperator*{\argmin}{arg\,min}

\DeclareMathOperator*{\spn}{span}

\newcommand{\lc}
{\mathrel{\raise2pt\hbox{${\mathop<\limits_{\raise1pt\hbox
{\mbox{$\sim$}}}}$}}}

\newcommand{\gc}
{\mathrel{\raise2pt\hbox{${\mathop>\limits_{\raise1pt\hbox{\mbox{$\sim$}}}}$}}}

\newcommand{\ec}
{\mathrel{\raise2pt\hbox{${\mathop=\limits_{\raise1pt\hbox{\mbox{$\sim$}}}}$}}}

\def\bb{\begin{equation}} \def\ee{\end{equation}}
\def\bbn{\begin{equation*}} \def\een{\end{equation*}}

\def\beqn{\begin{eqnarray}}  \def\eqn{\end{eqnarray}}

\def\beqnx{\begin{eqnarray*}} \def\eqnx{\end{eqnarray*}}

\def\bn{\begin{enumerate}} \def\en{\end{enumerate}}

\def\bd{\begin{description}} \def\ed{\end{description}}



\definecolor{DarkGreen}{rgb}{0.2,0.6,0.2}

\def\pink#1{\textcolor{\pink}{#1}}

\definecolor{brilliantrose}{rgb}{1.0, 0.33, 0.64}

\begin{document}

\title{
Towards An Analytical Framework for {Dynamic} Potential Games
  }

\author{
\and Xin Guo\thanks{Department of Industrial Engineering and Operations Research, University of California, Berkeley, USA ({\tt xinguo@berkeley.edu}).}
\and
Yufei Zhang\thanks{Department of Mathematics, Imperial College London,  London,  UK 
({\tt yufei.zhang@imperial.ac.uk}).}
}

\date{}
\maketitle

\noindent\textbf{Abstract.} 
Potential game is  an emerging notion and framework for studying $N$-player games, especially with
heterogeneous players.
In this paper,  we  build an analytical framework for dynamic potential games. We prove  that  a game is a dynamic potential game  if and only if each player's value function  can be decomposed as  a potential function and a residual term which is
  solely dependent on other players' policies. 
  {This decomposition is consistent with the result in the static setting} and  enables us to  identify and analyze an  important  and new class of  dynamic potential games called the distributed game. 
  Moreover, 
  we prove  that   a  game is a  dynamic potential game
if  the value function   has a symmetric Jacobian.
{This 
 generalizes the differential characterization for static potential games
by replacing the  classical derivative with a new notation 
 of  functional derivative with respect to Markov policies.}
For a general class of continuous-time stochastic   games,
we explicitly characterize 
their potential functions. 
{In particular, 
we show that the potential function  of  linear-quadratic  games can be studied through   a system of linear ODEs. Furthermore,
under a rank  condition on  control coefficients, we prove a linear-quadratic game is a Markov potential game if and only if all players have identical cost functions. 
}


\ms



\medskip
 
 \section{Introduction}

General-sum $N$-player games are of paramount importance in various fields, spanning economics, engineering, biology, and ecology  (see e.g., \cite{fudenberg1991game, bacsar1998dynamic}). These games model strategic interactions among multiple players, where each player aims to optimize her individual objective based on the observed system state and actions of other players. 
$N$-player games are notoriously difficult to analyze.

One well-established paradigm  for analyzing 
$N$-player games  is the mean field  framework \cite{caines2006large,lasry2007mean}. 
Its ingenious yet simple aggregation ideas enables efficiently approximating  the Nash equilibrium of  games with a large number of players.  However, such an approximation typically requires that
  players  are  (approximately) homogeneous 
who interact weakly and symmetrically.
As a result, the  mean field approach is not appropriate for general games with heterogeneous players and/or with general forms of interactions.

 An alternative and emerging
 framework for studying  $N$-player   games,
 especially  with heterogeneous players, is the potential game introduced  by \cite{monderer1996potential}. 
 A game  is a {(static)} potential game  if there exists a potential function such that whenever one player unilaterally deviates from their action, the change of the potential function equals to the change of that  player's value function. 
Once the potential function is identified, the daunting task of finding  Nash equilibrium of the game is simplified  to an optimization problem of finding the global optimum of the potential function.  

This introduction of potential functions for  static  games has since been adopted naturally for dynamic games.  
{The key distinctions between dynamic potential games and static ones lie in the   dependence of the potential function on the set  of admissible polices, and 
the interconnection and interaction among all players  through the state dynamics. 
 This makes constructing a potential function or verifying its existence in dynamic games significantly more challenging.  Indeed, \cite{leonardos2022global} shows that a dynamic game where the cost at each state forms a static potential game may not   be a dynamic potential game. 
  }
  
 One of the main culprits for the slow progress on dynamic potential games is the lack of appropriate analytical tools and limited understanding of the game structure for a dynamic potential game. 
The vast majority of existing literature on dynamic potential games focuses on algorithms for computing Nash equilibria of dynamic potential games 
\cite{marden2009joint,  marden2012state,   macua2018learning, mao2021decentralized, song2021can, zhang2021gradient,
ding2022independent, fox2022independent, 
leonardos2022global,
maheshwari2022independent,
 narasimha2022multi, zhangglobal}. 
These algorithmic studies are, however, built {\it without} verifying {\it a priori} whether the game under consideration is actually a dynamic potential game.
 In  discrete-time games, separability of  value functions has been identified in \cite{leonardos2022global}. With this separability property,  one  class of dynamic    games known as  team  Markov   games has been analyzed and remains the primary example for potential games studies \cite{littman2001value, wang2002reinforcement}.


 \paragraph{Our work.} 

{This paper focuses on   closed-loop dynamic potential games, which have a wide range of applications including economic  modelling and analysis of reinforcement learning algorithms. 
In closed-loop dynamic   games, 
 players optimize their objectives over functions of the time and the system state,  known as policies.  
This is in contrast to open-loop dynamic potential games  \cite{fonseca2018potential, fonseca2020stochastic}, where the strategies are functions of  the time and  system noise.
}


\begin{itemize}
\item  We prove  that    a  game is a  dynamic potential game  if and only if each player's value function  can be decomposed in two components:  a potential function and a residual term 
  solely dependent on other players' policies (Theorem  \ref{thm:separation_value}).
{It generalizes similar results for static games \cite{monderer1996potential,ui2000shapley} and discrete Markov games \cite{leonardos2022global} to general dynamic games with arbitrary state dynamics and policy classes.}

\item   We  apply this decomposition   characterization to
identify a new and  important   class of dynamic potential games called distributed games
(Section \ref{sec:distributed_game}).
In these  games,  
 each player's state and polices are decoupled,
allowing for the explicit construction of dynamic potential functions based on the games' cost functions  (Theorems \ref{thm:distributed_game} 
and 
\ref{thm:distributed_game_differentiable}).
{This characterization applies to distributed games with measurable coefficients and heterogeneous players. It  also underscores the critical distinction between static and dynamic potential games, even in  this decoupled setting. Specifically, constructing dynamic potential functions requires considering static potential games with enlarged action sets that include players’ state variables as additional actions (Remark \ref{rmk:distributed}).
 
}


  To the best of our knowledge, 
this is the first systematic characterization of potential functions for continuous-time   distributed games with heterogeneous players.
{It enables the analysis of   closed-loop Nash equilibria  in these (previously intractable) large-scale games using stochastic control techniques, as detailed in  Remark   \ref{rmk:distributed}. 
}

\item 
 We  introduce a notion  of  linear derivative
to characterize the sensitivity   
of the value function with respect to   unilateral  deviation of policies  (Definition \ref{def:linear_deri}). 
Leveraging this concept, 
we prove  that   an arbitrary   dynamic game is a  dynamic potential game
if  the value functions   have a symmetric Jacobian formed by their second-order derivatives 
(Theorem \ref{thm:symmetry_value_sufficient}). Moreover, 
we  construct  the potential function using   first-order derivatives of value functions. 

 {

The notation of linear derivatives is crucial for extending the differential characterization, originally established for static potential games \cite{monderer1996potential, ui2000shapley, arefizadeh2023characterization}, to dynamic games with continuous state and action spaces and infinite-dimensional policy classes.
It   addresses the challenge of defining derivatives over policy classes lacking a natural metric structure; see Remark \ref{rmk:differential_characterization} for details.

}

\item 
For continuous-time stochastic   games, 
where the state dynamics involves controlled drift and diffusion coefficients, 
we 
establish the   linear differentiability  of value functions and characterize the dynamic potential function via two approaches:  one probabilistic approach where the potential function is represented via 
sensitivity processes of the state and control processes with respect to policies,
and another where  the potential function is expressed via  a system of linear  partial differential equations (PDEs).  
{The probabilistic approach applies to games with twice differentiable policies and model coefficients (see Remark \ref{rmk:pde_C2_policy}), 
while the PDE approach accommodates more general   forms of policies and coefficients, such as    H\"{o}lder continuous 
state coefficients and policies (see Remark \ref{rmk:pde_irregular_policy}). 
}

\item We show that the probabilistic and PDE       characterizations  are consistent for   linear-quadratic  (LQ) games, and reduce to   a system of ordinary differential equations (Theorem  \ref{thm:lq_characterisation}). 
 { In particular, we prove that 
   under a  rank  condition on control coefficients, a linear-quadratic game is a Markov potential game if and only if all players have identical cost functions 
  (Proposition \ref{prop:lq_nondegenreate}).

 To the best of our knowledge, this is  
the first necessary and sufficient condition for multi-dimensional LQ Markov potential games. 
It implies that  a strong coupling between the players' actions    requires 
 the homogeneity of cost functions for a potential game. 
It also suggests that  the strength of interactions among players through dynamics and policies is crucial in determining whether a dynamic game is a potential game. These are intriguing insights about the structure of dynamic potential games (Remark \ref{rmk:lq_rank_condition}). 
 }
\end{itemize}

\paragraph{Our approach.} One of the key challenges to develop the general analytical framework for the dynamic potential game is to develop an appropriate notion of the functional derivative
 {with respect to polices}
for possibly continuous  time/state spaces and general policy classes. 
For instance, in the discrete time setting of  \cite{ leonardos2022global, hosseinirad2023general} with finite-dimensional parameterizations  for policies,   Fr\'{e}chet  derivatives are adopted to  show that the game is a potential game if the second-order derivatives of the value functions are symmetric. However,  neither Fr\'{e}chet derivative nor parametrization is  appropriate for the general framework: the set of admissible policies becomes infinite-dimensional and may not be a normed   space; for games with stochastic policies,
 the  admissible polices   take values in the space of  probability measures 
and hence do  not form a vector space. 
 To overcome the above difficulties, we define the derivative of 
value function along a convex combination of two policies.   
Compared with the Fr\'{e}chet  derivative,
such a   
  directional derivative    is   easier  to compute and
 well-defined  under   weaker conditions.
It    also eliminates the need to select a norm over the   policy class, providing more flexibility in the analysis 
(see Remark \ref{rmk:differential_characterization} for more details).

In contrast to the mean field approach and its graphon extensions, which either  require homogeneity in players' state dynamics and symmetric interaction among players  and policy classes or at least some specific \emph{asymptotic} structures of the cost functions (see e.g., \cite{gao2020linear, aurell2022stochastic, lacker2022label, bayraktar2023propagation}),  the potential game framework offers a new perspective on constructing (closed-loop) Nash equilibrium for finite-player games that is non-asymptotic in nature. It also accommodates heterogeneity among players to allow for different action sets and state dynamics, different dependencies on individual behavior, and different interaction strengths between pairs of players (see Theorem \ref{thm:distributed_game_differentiable}
and  Proposition \ref{prop:distribted_qudratic_MF}).
{
We emphasize that our proposed framework characterizes  closed-loop dynamic potential games with a finite number of heterogeneous players. This contrasts with potential mean field games, which involve infinitely many homogeneous players and are associated with McKean-Vlasov control problems (see, for example, \cite{aurell2018mean, cecchin2022weak, gangbo2022global, carmona2023synchronization} and their references for more details).

}

For instance, in Section \ref{sec:distributed_game}, we construct potential functions for  closed-loop $N$-player distributed games, arising from      multi-vehicle coordination \cite{aghajani2015formation},   energy market \cite{paccagnan2016aggregative}, 
crowd motion modelling \cite{aurell2018mean,  carmona2023synchronization}, 
 and  interacting particle models in  biology \cite{nourian2011mean} and physics \cite{carmona2023synchronization}. 
The results  
 accommodate players with heterogeneous characteristics and interactions, facilitating the construction of distributed closed-loop Nash equilibria in large-scale games using stochastic control techniques.

\paragraph{Notation.}
For each $n\in \sN$, we denote by $\sS^n$ the space of $n\times n$ symmetric matrices. 
For each $T>0$,  $n\in \sN$, $\alpha\in (0,1]$,  probability space
$(\Omega, \cF,   \sP)$    and Euclidean space
$(E, |\cdot|)$, we      introduce the following spaces:
\begin{itemize}
\item 
  $\cS^\infty([t,T];E)$, $t\in [0,T]$,
is the space of  measurable processes $X:\Om\times [t,T]\to E $
  such that 
 $\sE[\sup_{s\in [t,T]} |X_s|^q]<\infty$  for all $q\ge 1$;
 \item 
 $\cH^\infty([t,T];E)$,  $t\in [0,T]$, 
is the space of  measurable processes $X:\Om\times [t,T]\to \sR^n $
  such that 
 $\sE[\int_t^T |X_s|^q \d s]<\infty$  for all $q\ge 1$;
\item 
$C^{i+\alpha/2, j+\alpha}([0,T]\times \sR^{n}; E)$, $i,j\in \sN\cup\{0\}$, is the space of    functions 
$u:[0,T]\times \sR^{n}\to E$  
such that 
$
\|u\|_{{i+\alpha/2,j+\alpha }} \coloneqq \sum_{\ell= 0}^{i} \|\p^{\ell}_t u\|_{0 } 
+ \sum_{\ell= 0}^{j}  \|\p^\ell_x u\|_{0} 
+ [\p^{i}_t u]_{\alpha/2 } 
+[\p^{j}_x u]_{\alpha } 
<\infty
$, where
$\|u\|_{0}\coloneqq \sup_{(t,x) \in [0,T]\times \sR^n}
|u(t,x)|$
and 
$
[u]_{\alpha} \coloneqq  \sup_{(t,x),(t,x')\in [0,T]\times \sR^n}
\frac{|u(t,x)-u(t',x')|}{(|t-t'|^{1/2}+|x-x'|)^\alpha } 
$. 
Similar definitions extend to
$C^{i+\alpha/2, j+\alpha,k+\alpha}([0,T]\times \sR^{n}\times \sR^{m}; E)$ 
for 
 functions with two   space variables. 
\end{itemize}
 Throughout the paper, unless otherwise stated, proofs of  theorems 
   are deferred to Section  \ref{sec:main_proof}.

\section{(Dynamic) Potential Game and Its Nash Equilibrium}
To introduce the  mathematical framework for dynamic potential games, let us start by some basic notions for the game and associated policies. 
 Consider a     game 
$\cG=(I_N, \cS, (A_i)_{i\in I_N}, \pi^{(N)}, (V_i)_{i\in I_N})$, 
where 
 $I_N= \{1,\ldots, N\}$, $N\in \mathbb{N}$, is a finite set of players;  
 $\mathcal{S}$ is a  topological space  representing the state space of the underlying dynamics;
  $A_i$ is a Borel subset of a  topological vector space    representing   the  action set of  player $i$;
  $\pi^{(N)}= \prod_{i\in I_N} \pi_i$ 
is  the set of     joint policy profiles  of 
   all   players, where 
    $\pi_i$ is a set of 
    Borel measurable functions
    $\phi_i:\cS\to A_i$
       representing the admissible    policies 
  of player $i$;
   and $V_i:   \cS \times \pi^{(N)} \to \sR$ 
  is the value function  of player $i$, 
  where $V_i^{s}(\phi)\coloneqq V_i(s,\phi)$ is   player $i$'s expect cost  
if the state dynamics    starts at     the   state $s\in \cS$  
  and all players  take the policy profile $\phi \in \pi^{(N)}$. 
For each  $i\in I_N$,   we  denote  by $\pi^{(N)}_{-i}= \prod_{j\in I_N\setminus \{i\}}  \pi_j$  
 the sets of   policy  profiles    of all players except player $i$.
The elements of   $\pi^{(N)}$ and $\pi^{(N)}_{-i}$ 
are denoted as 
 ${\phi} =(\phi_i)_{i\in I_N}$, 
 and $\phi_{-i}=(\phi_j)_{j\in I_N\setminus \{i\}}$, respectively. 
 
This game includes   discrete-time
   and continuous-time 
     dynamic games;  and its value function $V_i^{s}(\phi)$
 may depend  on the underlying state system, whose time and space variables  are referred collectively as the state variable.  Note that one may simply view the static game as a stateless dynamic game. 

The focus of this paper is on a class of dynamic games called dynamic potential games, defined through two related yet distinct notions of Nash equilibrium:  closed-loop and    Markovian types. 

\begin{Definition}[Nash equilibrium]
\label{def:NE}
A policy profile  $\phi \in  \pi^{(N)}$ is  a  closed-loop Nash equilibrium for 
 $\cG$ with initial state $s_0\in \cS$
if
$$
V^{s_0}_i((\phi_i,\phi_{-i}))\le V^{s_0}_i((\phi'_i,\phi_{-i})) \quad \fa  i\in I_N, \phi'_{i}\in \pi_{i}.
$$
A policy profile  $\phi \in  \pi^{(N)}$ is   a  Markov   Nash equilibrium for  $\cG$  
if 
$$
V^{s}_i((\phi_i,\phi_{-i}))\le V^{s}_i((\phi'_i,\phi_{-i}))
\quad  \fa s\in \cS, i\in I_N, \phi'_{i}\in \pi_{i}.
$$
\end{Definition}

Definition \ref{def:NE}  is consistent with the concept of closed-loop and    Markov equilibrium introduced in stochastic differential games, as described in \cite{carmona2018probabilistic}. Note that 
a closed-loop equilibrium $\phi$  can depend on the initial state $s_0$  of the   system dynamics,
while 
Markov equilibrium requires that the policy profile $\phi$  be a closed-loop Nash equilibrium for {\it all} initial states.


\begin{Definition}[Dynamic potential game]
\label{def:MPG}
A   dynamic   game 
$\cG$ 
 with initial state $s_0\in \cS$
     is   a closed-loop   potential game (CLPG),
     if there exists 
$\Phi^{s_0}:   \pi^{(N)}\to \sR$, called   a potential function,    such that for all   $i\in I_N$, 
   $\phi_i,\phi_i'\in  \pi_i$ and    $\phi_{-i}\in  \pi^{(N)}_{-i}$,
   \bb
   \label{eq:MPG}
   \Phi^{s_0}((\phi'_i,\phi_{-i})) -\Phi^{s_0}((\phi_i,\phi_{-i})) = V^{s_0}_i((\phi'_i,\phi_{-i})) -V^{s_0 }_i((\phi_i,\phi_{-i})).
   \ee
 A game   $\cG$ is a Markov potential game (MPG), 
     if there exists  a potential function
$\Phi: \cS\times   \pi^{(N)}\to \sR$ such that for all $s_0\in \cS$, 
$\Phi(s_0,\cdot)$ is a potential function of $\cG$  
 with initial state $s_0\in \cS$.
\end{Definition}

Intuitively, 
a  game $\cG$ is a dynamic potential game  if there exists a potential function such that whenever one player unilaterally deviates from their policy, the change of the potential function equals to the change of that  player's value function. 
%
%
%
Definition \ref{def:MPG} enables finding 
a  Nash equilibrium of a   potential game  
by optimizing its potential function,
as shown in the following proposition.
\begin{Proposition}
\label{prop:ne_mpg}

If there exists $s_0\in \cS$
such that 
   $\cG$ with initial state $s_0 $ is a CLPG with potential function $\Phi^{s_0}$,
then 
any   $\phi^*\in \argmin_{\phi\in  \pi^{(N)} }\Phi^{s_0}(\phi)$ 
 is a closed-loop Nash equilibrium of $\cG$ with initial state $s_0 $.   
   Similarly, if $\cG$ is an MPG with potential function $\Phi$
and  $\phi^*\in \pi^{(N)}$ satisfies  
$\phi^*\in \argmin_{\phi\in  \pi^{(N)} }\Phi(s, \phi)$ for all $s\in \cS$,
then  $\phi^*$
 is an Markov  Nash equilibrium of $\cG$.
\end{Proposition}

  Proposition \ref{prop:ne_mpg} shows the 
 importance of an explicit criterion  for   the existence or the construction of a potential function for a given potential game. The most known example of MPG is the team  Markov  game, where the potential function 
corresponds to  the players' value function. 

\begin{Example}[Team Markov Game]\label{prop:team_MG}
This is  
a  game where all players have a common interest. That is, in this game $\cG=( I_N, \cS, (A_i)_{i\in I_N}, \pi^{(N)}, (V_i)_{i\in I_N})$,  
  $V_i=V_j$ for all $i,j \in I_N$; see \cite{littman2001value, wang2002reinforcement}.  
In this setting,
   $\cG$ is an MPG with a potential function  $\Phi\equiv V_1$. 
  
\end{Example}


Moreover, a game is a potential game if and only if the value function of each player can be decomposed into two terms:
 one that is common for all players, namely the potential function, and the other  that may be different for each player but depends only on the polices of other players. 


 \begin{Theorem}
 \label{thm:separation_value}
A game  $\cG=(I_N, \cS, (A_i)_{i\in I_N}, \pi^{(N)}, (V_i)_{i\in I_N})$
 with initial state $s_0\in \cS$ 
is a CLPG with 
potential $\Phi^{s_0}$ if and only if  
for all $i\in I_N$, there exists   $U^{s_0}_i:  \pi^{(N)}_{-i}\to \sR$ such that 
\begin{equation}
\label{eq:separation_value}
V^{s_0}_i((\phi_i, \phi_{-i}))=  \Phi^{s_0} ((\phi_i, \phi_{-i})) + U^{s_0}_i(\phi_{-i}),
\quad \fa \phi_i\in \pi_i, \phi_{-i}\in \pi^{(N)}_{-i}.
\end{equation}
Furthermore, a game  $\cG$ is an MPG if and only if \eqref{eq:separation_value}
holds for all $s_0\in S$ and $i\in I_N$.
\end{Theorem}

{
\begin{Remark}[Dynamic vs static potential games]
The definition of dynamic potential game and the subsequent Proposition \ref{prop:ne_mpg} are consistent with their respective counterparts for  static potential games.
 Theorem \ref{thm:separation_value}   generalizes results from \cite{ui2000shapley} for static   games and \cite{leonardos2022global} for discrete Markov games  with stationary policies to general dynamic games with arbitrary state dynamics and policy classes.

 The key distinctions between dynamic potential games and static ones lie in the   dependence of the potential function on the set  of admissible policy profiles, and the nonlinear relationship between value functions and policies through the state dynamics.
 This makes identifying a potential function in dynamic games, if one exists, significantly more challenging. The broader the class of admissible policy profiles $\pi^{(N)}$, the stricter the requirements on system dynamics and payoff functions to guarantee the existence of a potential function.
Deriving  and decomposing value functions for  dynamic games, as in Theorem \ref{thm:separation_value}, is more challenging  than for static ones,  primarily   due to the interconnection and interaction among all players that arise through the state dynamics. 

In fact, a dynamic game in which the cost at each state forms a static potential game may not necessarily be an MPG \cite{leonardos2022global}. As a result, there are few known examples of MPGs beyond the team Markov game. In the next section, we will identify and analyze an important (new) class of dynamic  potential games called  distributed games. 
\end{Remark}
  }

\section{Distributed Game  as Markov Potential Game}
\label{sec:distributed_game}
 In a distributed   game, each    player's control depends {\it only} on the state of herself, although her value function may depend on the joint state and policy profiles    of all players. More precisely,


 \begin{Definition}
     [Distributed  game
 $\cG_{\rm dist}$]
 \label{example:distributed_game}
 Let $T>0$, $N\in \sN$,   $(n_i)_{i=1}^N,(k_i)_{i=1}^N\subset  \sN$. 
The dynamic game 
  $\cG_{\rm dist}=(I_N, \cS, (A_i)_{i\in I_N}, {\pi}^{(N)}_{\rm  dist}, (V_i)_{i\in I_N})$
   with $I_N=\{1,\ldots, N\}$, $\cS=[0,T]\times\prod_{i\in I_N} \sR^{n_i}$ and $ A_i$ is a Borel set in $\sR^{k_i}$.
 Here the set  ${\pi}^{(N)}_{\rm  dist}=\prod_{i\in I_N}{\pi}^{\rm dist}_i$ contains    distributed policy profiles,
 where 
  ${\pi}^{\rm dist}_i$
contains measurable functions $\phi_i:[0,T]\times \sR^{n_i}\to A_i$ such that 
$\sup_{(t,x_i)\in [0,T]\times \sR^{n_i}}\frac{|\phi_i(t,x_i)|}{ 1+|x_i|}<\infty$,
and for all $(t, x_i )\in [0,T]\times  \sR^{n_i}$,
player $i$'s  state dynamics 
\bb
\label{eq:state_policy_distributed}
\d  X^i_s =
b_i(s,X^i_s,\phi_i(s,X^i_s))
\d s +
\sigma_i(s,X^i_s,\phi_i(s,X^i_s))
\d W_s, \quad s\in [t,T]; \quad X^i_t=x_i,
\ee
admits a unique square-integrable weak solution $X^{t,x_i,\phi_i}$ 
 on a probability space $(\Omega,\cF,\sP)$.  
Here $(b_i,\sigma_i):[0,T]\times \sR^{n_i}\times \sR^{k_i}\to \sR^{n_i}\times \sR^{n_i\times n_w}$
are given measurable functions,
and  $W$ is an $n_w$-dimensional Brownian motion on $(\Omega,\cF,\sP)$. 
Let $n_x=\sum_{i=1}^N n_i$ and 
   $n_a=\sum_{i=1}^N k_i$, 
   and 
for each $t\in [0,T]$, $ x=(x_i)_{i\in I_N}\in  \prod_{i\in I_N} \sR^{n_i}$ and $\phi = (\phi_i)_{i\in I_N}\in {\pi}^{(N)}_{\rm  dist}$, 
let 
$X^{t,x,\phi}= (X^{t,x_i,\phi_i})_{i\in I_N}$ be 
the joint   state process,
and 
let 
$\alpha^{t,x,\phi}= (\phi_i(\cdot, X^{t,x_i,\phi_i}))_{i\in I_N}$ be 
the joint  control process.
Then the player $i$'s value function 
 $V_i: [0,T]\times \sR^{n_x}\times {\pi}^{(N)}_{\rm  dist}\to \sR$ is given  by  
\bb 
\label{eq:cost_dist_i}
  V^{t,x}_i(\phi) = \sE^{\mathbb{P}} \left[\int_t^T f_i(s,X^{t,x,\phi}_s,\alpha^{t,x,\phi}_s)\d s + g_i(X^{t,x,\phi}_T)\right],
\ee
where  $f_i :  [0,T]\times \sR^{n_x}\times \sR^{n_a}\to  \sR $
is a   measurable function such that 
$\sup_{(t,x,a)\in [0,T]\times \sR^{n_x}\times \sR^{n_a}}\frac{|f_i(t,x,a)|}{1+|x|^2+|a|^2}<\infty$
and $g_i :  \sR^{n_x}  \to\sR $
is a    measurable function
such that 
$\sup_{x\in   \sR^{n_x}}\frac{ |g_i(x)|}{1+|x|^2}<\infty$.
\end{Definition}

 A distributed  game $\cG_{\rm dist}$  
is an MPG
  if its cost functions admit a potential structure, as stated below.
 \begin{Theorem}
 \label{thm:distributed_game}
      Take the    game   $ \cG_{\rm dist}$     in Definition \ref{example:distributed_game}.
Suppose that  there exist measurable functions $F:[0,T]\times\sR^{n_x}\times \sR^{n_a}\to \sR$,
$G: \sR^{n_x}\to \sR$, 
$U_{f_i}:[0,T]\times \prod_{j\not=i}\sR^{n_j}\times  \prod_{j\not=i}\sR^{k_j}\to \sR$ 
and $U_{g_i}:  \prod_{j\not=i}\sR^{n_j} \to \sR$, $i\in I_N$, such that  
for all $i\in I_N$ and $(t,x, a)\in [0,T]\times  \sR^{n_x}\times \sR^{n_a}$,
\begin{equation}
\label{eq:distributed_cost_decomp}
f_i(t,x,a)= F(t,x,a) +U_{f_i}(t,x_{-i},a_{-i}),
\quad 
g_i(x)= G(x) +U_{g_i}(x_{-i})
\quad 
\end{equation}
with 
$x_{-i}=(x_j)_{j\not =i}$ and $a_{-i}=(a_j)_{j\not =i}$. 
Then the distributed game $ \cG_{\rm dist}$ is an MPG,  with 
a potential function 
$\Phi:[0,T]\times  \sR^{n_x}\times  {\pi}^{(N)}_{\rm  dist}\to \sR$ defined by
\begin{align}
\label{eq:distributed_potential}
  \Phi^{t,x}(\phi) = \sE^{\mathbb{P}} \left[\int_t^T F(s,X^{t,x,\phi}_s,\alpha^{t,x,\phi}_s)\d s + G(X^{t,x,\phi}_T)\right],
\end{align}
with $X^{t,x,\phi}$ and $\alpha^{t,x,\phi}$  defined  in \eqref{eq:cost_dist_i}.
 \end{Theorem}
 
 {

 \begin{Remark}
 \label{rmk:distributed}

A few remarks regarding Theorem  \ref{thm:distributed_game} are in place.

First,  to the best of our knowledge, it is the first  characterization of potential functions for general dynamic distributed games; and   it applies to general distributed games with
    measurable policies and model coefficients defined by \eqref{eq:state_policy_distributed} and \eqref{eq:cost_dist_i}. 
    For instance, it  enables characterizing   potential functions  for  dynamic games in crowd motion modeling, as shown in Example \ref{example:crowd_distributed} below.

Secondly,       it highlights analytically the crucial distinction between static and dynamic potential games. Indeed,
to construct the functions $F$ and $G$ for a dynamic potential games, it is essential to construct associated static potential games by enlarging action sets with players' state variables. In particular,  it is sufficient to construct  
 a  static potential game with  a  potential function $F(t,\cdot)$ according to  \eqref{eq:distributed_cost_decomp} with 
  $f_i(t,\cdot)$ being player $i$'s value function  
 and  $\sR^{n_i}\times \sR^{k_i}$   player $i$'s   action set,
 as well as another static potential game with a potential function $G$ for which  $g_i$ is player $i$'s value function and $\sR^{n_i}$ is player $i$'s action set. 
 
Moreover, it allows for computing   Nash equilibria of a 
 distributed game with heterogeneous players using stochastic control techniques.
 As a special case of Theorem \ref{prop:ne_mpg},
 finding an equilibrium of the game $ \cG_{\rm dist}$ can be translated to a distributed control problem, whose minimizer can be 
constructed using a  stochastic maximum
principle and a Hamilton-Jacobi-Bellman equation;
see   \cite{jackson2023approximately} for the special case where the state dynamics \eqref{eq:state_policy_distributed} has coefficients $b_i(t,x_i,a_i)=a_i$ and $\sigma_i(t,x_i,a_i)\equiv \sigma$.

Finally, it enables employing existing numerical algorithms for high-dimensional control problems (see \cite{germain2021neural} for a comprehensive survey)  to 
 minimize the potential function and compute equilibria   for the game $\cG_{\rm dist}$. 
 The convergence of these algorithms can be ensured by adapting techniques for general control problems to the potential function $\Phi$ (see e.g., \cite{giegrich2024convergence, sethi2024entropy}).
Such equilibrium computations for distributed games via the potential function will be an interesting  future research.

\end{Remark}
 
\begin{Example}[Crowd motion]
\label{example:crowd_distributed}
    Take the    game   $ \cG_{\rm dist}$     in Definition \ref{example:distributed_game}.
    Suppose that    
for all $i,j \in I_N$,
$n_i=n$ and  
there exists $c_i:[0,T]\times \sR^{n}\times \sR^{k_i}\to \sR$,
$h_{ij}:[0,T]\times \sR^n\to \sR$
 and $\ell_i,  r_{ij}:\sR^n\to \sR$
such that for all $(t,x,a)\in [0,T]\times \sR^{n_x}\times \sR^{n_a}$,
\bb
\label{eq:cost_crowd}
f_i(t,x,a) =c_i(t,x_i,a_i)+\sum_{j\in I_N\setminus \{i\}}  h_{ij}(t, x_i-x_j), \quad 
g_i(x) =\ell_i(x_i)+\sum_{j\in I_N\setminus \{i\}}  r_{ij}(x_i-x_j), 
\ee
and $h_{ij}(t, x) = h_{ji}(t, -x)$ 
and $r_{ij}( x) = r_{ji}( -x)$ for all $i\not =j$. 
Then
 $ \cG_{\rm dist}$ is an MPG and a potential function $\Phi$
is given by  
\eqref{eq:distributed_potential}, with 
$F$ and $G$ defined as 
 \begin{align*}
 F(t,x,a)\coloneqq \sum_{i=1}^N c_i(t,x_i,a_i)+\frac{1}{2}\sum_{i,j\in I_N, i\not =j} h_{ij}(t, x_i-x_j),
 \quad G(x) \coloneqq \sum_{i=1}^N  \ell_i( x_i )+\frac{1}{2}\sum_{i,j\in I_N, i\not =j}  r_{ij}(  x_i-x_j).
 \end{align*}
 Note that in lieu of Remark \ref{rmk:distributed}, here
 $ 
U_{f_i} (t,x_{-i},a_{-i}) \coloneqq -  \sum_{j\in I_N\setminus \{i\}} c_j(t,x_j,a_j)
- \frac{1}{2} \sum_{k,\ell\in I_N\setminus \{i\}, k\not =\ell} h_{k\ell}(t, x_k-x_\ell)  
$
and 
$
U_{g_i} (x_{-i}) \coloneqq - \sum_{j\in I_N\setminus \{i\}} \ell_j( x_j)
- \frac{1}{2} \sum_{k,\ell\in I_N\setminus \{i\}, k\not =\ell} r_{k\ell}(  x_k-x_\ell) 
 $ in  \eqref{eq:distributed_cost_decomp}, both of which are independent of player $i$'s state and action.

 This crowd motion game has been studied in 
\cite{aurell2018mean} 
for pedestrian motion modelling,
in \cite{carmona2018probabilistic} for animal flocking, 
and \cite{carmona2023synchronization} for synchronization in 
neuroscience.
We emphasize that the form of  game here generalizes those in 
\cite{aurell2018mean,carmona2018probabilistic,carmona2023synchronization}: 
condition \eqref{eq:cost_crowd}  removes  the homogeneous and the symmetric interaction conditions imposed in their analysis. 
In particular, it allows each player to have different   individual preferences specified by   
$c_i$ and $\ell_i$,  as well as
   to be differently influenced by other players' positions as indicated by the kernels  
   $h_{ij}$ and $r_{ij}$.

\end{Example}

Now,  Theorem \ref{thm:distributed_game} can be extended with explicit  constructions of $F$ and $G$ and hence  explicit   expression of the dynamic potential function $\Phi$,   when the running and terminal costs are twice differentiable. 

 }

 \begin{Theorem}
 \label{thm:distributed_game_differentiable}
      Take the    game   $ \cG_{\rm dist}$     in Definition \ref{example:distributed_game}.
Suppose that 
for all $i\in I_N$ and $t\in [0,T]$,
 $\sR^{n_x} \times \sR^{n_a}\ni  (x,a)\mapsto ( f_i(t,x,a),g_i(x))\in \sR^2$ is twice continuously differentiable, 
 for all $i,j\in I_N$,
\bb
\label{eq:symmetric_hessian_distributed}
 \begin{pmatrix}
\p_{x_ix_j}f_i  & \p_{a_ix_j}f_i  
\\
\p_{x_ia_j}f_i   & \p_{a_ia_j}f_i 
\end{pmatrix}   
=
\begin{pmatrix}
\p_{x_ix_j}f_j   & \p_{a_ix_j }f_j 
\\
 \p_{x_ia_j} f_j     & \p_{a_ia_j}f_j 
\end{pmatrix},
\quad 
\p_{x_ix_j}g_i =\p_{x_ix_j}g_j,
\ee
and  for all $i\in I_N$,
$\sup_{(t,x,a)\in [0,T]\times \sR^{n_x}\times \sR^{n_a}} \frac{|(\p_{(x_i,a_i)}f_i)(t,x,a)|+|(\p_{x_i}g_i)(x)|}{1+|x|+|a|}<\infty
 $.
 Then 
      $\cG_{\rm  dist}$ 
is an MPG and 
a potential
function $\Phi$ is given by \eqref{eq:distributed_potential},
 with $F$ and $G$ defined as 
\begin{align}
\label{eq:F_G_distributed}
\begin{split}
F(t,x,a)& \coloneqq \sum_{i=1}^N \int_0^1
\left(
 \begin{pmatrix} 
\p_{x_i} f_i \\ \p_{a_i} f_i
\end{pmatrix} \big(t,    \hat{x}+r (x-\hat{x}),\hat{a} +r ( a- \hat{a}) \big)
\right)^\top  
  \begin{pmatrix} 
x_i-\hat{x}_{i}  \\   a_i-\hat{a}_i
\end{pmatrix}
 \d r,
\\
G(x) & \coloneqq \sum_{i=1}^N \int_0^1
  (\p_{x_i} g_i) ( \hat{x} +r (x-\hat{x}) )^\top 
(x_i -\hat{x}_i)  \d r 
\end{split}
\end{align}
for some fixed   $(\hat{x},\hat{a}) \in  \sR^{n_x}\times \sR^{n_a}$. 
 \end{Theorem}

{

  Theorem \ref{thm:distributed_game_differentiable}
 is reminiscent of
the differential characterizations of static potential games \cite{monderer1996potential, ui2000shapley}. However, 
 as highlighted in   Remark \ref{rmk:distributed}, 
    static potential games and the dynamic ones are distinct. For instance, replacing \eqref{eq:symmetric_hessian_distributed}  by the condition 
$(\partial_{a_ia_j}f_i)(t,x,\cdot)=(\partial_{a_ia_j}f_j)(t,x,\cdot) $ for all $(t,x) $   does not guarantee that 
$ \cG_{\rm dist}$  is an MPG. 

}

In addition, Theorem \ref{thm:distributed_game_differentiable} can be applied  to identify a special class of dynamic games as an MFG  with a quadratic form of potential function, 
 if the costs of the game  are  quadratic and  depend both on 
other players through their 
average behavior.

 \begin{Example}[MPG with quadratic form of potential functions]
\label{prop:distribted_qudratic_MF}
       Take the    game   $ \cG_{\rm dist}$     in Definition \ref{example:distributed_game}.
Suppose that     $n_i=n$ and $k_i=k$
for all $i\in I_N$,
and 
there exists
$\bar{Q}, Q_i: [0,T]\to   \sS^{n}$, $\bar{R}, R_i:  [0,T]\to   \sS^{k}$, $\bar{G}, G_i \in \sS^{n}$, $i\in I_N$, 
$\gamma, \kappa: [0,T]\to \sR$
 and   $\eta \in \sR$
such that  for all $i\in I_N$
and $(t,x,a)\in [0,T]\times \sR^{Nn}\times \sR^{Nk}$,
\begin{align}
\label{eq:distributed_quadratic_MF}
\begin{split}
 f_i(t,x,a)
 &= x_i^\top Q_i(t) x_i  
+ \left(x_i-\gamma(t) \ol{x}^{(N-1)}_{ -i}
\right)^\top \bar{Q}(t)   \left( x_i-\gamma(t) \ol{x}^{(N-1)}_{ -i}
\right)  
\\
&\q 
+ a_i^\top R_i(t) a_i +  \left(a_i-\kappa (t)\ol{a}^{(N-1)}_{ -i}
\right)^\top \bar{R}(t)   \left(a_i-\kappa (t)  \ol{a}^{(N-1)}_{ -i}
\right),
\\
g_i(x)
&=  x_i^\top G_i x_i  + 
\left(x_i - \eta \ol{x}^{(N-1)}_{ -i}
\right)^\top \bar{G}    \left( x_i- \eta \ol{x}^{(N-1)}_{ -i}
\right),  
\end{split}
\end{align}
where $x_i\in \sR^n$ and $a_i\in \sR^k$  are    player $i$'s state and action, respectively, 
and 
$\ol{x}^{(N-1)}_{ -i}\in \sR^n$
and 
$\ol{a}^{(N-1)}_{ -i}\in \sR^k$
are the average state and action of other players defined by:
$$
\ol{x}^{(N-1)}_{ -i} = \frac{1}{N-1}
\sum_{j=1,j\not=i}^N  x_j,
\quad 
 \ol{a}^{(N-1)}_{ -i}= \frac{1}{N-1}
\sum_{j=1,j\not=i}^N  a_j.
$$
Then    $\cG_{\rm  dist}$ 
is an MPG and 
a potential
function $\Phi$ is defined by \eqref{eq:distributed_potential}
 with 
 \begin{align}
\label{eq:F_G_distributed_MF}
F(t,x,a)\coloneqq  x^\top \tilde{Q}(t) x+a^\top \tilde{R}(t) a,
\quad 
G(x)\coloneqq  x^\top \tilde{G} x, 
\end{align}
where  
 \begin{align*}
\tilde{Q}&=\begin{pmatrix}
  Q_{1}+\bar{Q} & -\frac{\gamma\bar{Q}}{N-1} & \dots & -\frac{\gamma\bar{Q}}{N-1}
\\
-\frac{\gamma\bar{Q}}{N-1} &   Q_{2}+\bar{Q}   &\ddots & \vdots
\\
\vdots &\ddots  &\ddots &  -\frac{\gamma\bar{Q}}{N-1}
\\
-\frac{\gamma\bar{Q}}{N-1} & \dots  & -\frac{\gamma\bar{Q}}{N-1} & Q_{N}+\bar{Q}  
 \end{pmatrix},
 \q
\tilde{R}=\begin{pmatrix}
  R_{1}+\bar{R} & -\frac{\kappa\bar{R}}{N-1} & \dots & -\frac{\kappa\bar{R}}{N-1}
\\
-\frac{\kappa\bar{R}}{N-1} &   R_{2}+\bar{R}   &\ddots & \vdots
\\
\vdots &\ddots  &\ddots &  -\frac{\gamma\bar{R}}{N-1}
\\
-\frac{\kappa\bar{R}}{N-1} & \dots  & -\frac{\kappa\bar{R}}{N-1} & R_{N}+\bar{R}  
 \end{pmatrix},
 \\
 \tilde{G}&=\begin{pmatrix}
  G_{1}+\bar{G} & -\frac{\eta \bar{  G}}{N-1}  & \dots & -\frac{\eta \bar{  G}}{N-1}  
\\
-\frac{\eta \bar{  G}}{N-1}     &   G_{2}+\bar{G}   &\ddots & \vdots
\\
\vdots &\ddots  &\ddots & -\frac{\eta \bar{  G}}{N-1} 
\\
- \frac{\eta \bar{  G} }{N-1}   & \dots  & -\frac{\eta \bar{  G}}{N-1}  & G_{N}+\bar{G}  
 \end{pmatrix}.
\end{align*}

We reiterate that this is a game with heterogeneity among   players, who can  have different action sets, different state dynamics  (i.e., 
different $b_i$ and $\sigma_i$),
and   
different 
  dependence on their individual behavior  (i.e.,  different $Q_i, R_i$ and $G_i$).
This is 
in contrast to the classical  $N$-player mean field games in  \cite{bensoussan2016linear,carmona2018probabilistic}, which requires that all players have  homogeneous state dynamics and cost functions.  In particular, even when $N\to \infty$, the matrices $\tilde Q$, $\tilde R$ and $\tilde G$ become diagnolized indicating independence of actions among players, who nevertheless remain heterogeneous. 

 \end{Example}



  \section{Characterization of Differentiable    Potential Game}
\label{sec:characterisation_differentiable_MPG}

   In this section, 
   we propose,  
instead of Theorem \ref{thm:separation_value}, to construct   potential functions   
 based on    derivatives of  value functions, assuming their
regularities with respect to policies.

Let us start by
a  precise definition of the     differentiability  of a (scalar-valued)  function  
with respect to   unilateral  deviations of   policies. 
Recall that 
for each $i\in I_N$,  
$\pi_i$ is a subset of all   measurable functions from $\cS$ to $A_i$.
We denote by 
$\spn{(\pi_i)}$ 
 the vector space of all linear combinations of policies  in $\pi_i$:
 $$
\spn{(\pi_i)}=\left\{
 \sum_{\ell=1}^m \alpha_\ell  {\phi}^{(\ell)} \bigg\vert
m\in \sN,
(\alpha_\ell)_{\ell=1}^m \subset \sR,
(\phi^{(\ell)})_{\ell=1}^m\subset \pi_i
\right\}.
 $$

\begin{Definition} 
\label{def:linear_deri}
Let 
$\pi^{(N)}=\prod_{i\in I_N}\pi_i$
be a convex set 
and   $f: \pi^{(N)}\to \sR$.
For each $i\in I_N$,
we say $f$ has  a linear  derivative 
with respect to $\pi_i$, if there exists 
$\frac{\delta f}{\delta \phi_i}:\pi^{(N)} \times  \spn(\pi_i) \to \sR$,
such that for all $\phi=(\phi_i,\phi_{-i})\in \pi^{(N)}$,
$ \frac{\delta f}{\delta\phi_i} (\phi;\cdot) $
is linear and 
\bb
\label{eq:first_der_def}
\lim_{\eps\downarrow   0 }\frac{  f\big((\phi_i+\eps(\phi'_i-\phi_i),\phi_{-i})\big) -f (\phi) }{  \eps}
= \frac{\delta f}{\delta \phi_i} (\phi ; \phi'_i-\phi_i),
\quad \fa \phi'_i\in \pi_i.
\ee
For each $i,j\in I_N$, 
we say  $f $ 
has    second-order linear derivatives   
 with respect to $\pi_i\times \pi_j$, 
 if     
 (i)  for all $ k\in \{i,j\}$,   $f$ has a  linear   derivative $\frac{\delta f}{\delta \phi_k}  $  with respect to $\pi_k$,
and  (ii) for all $(k,\ell)\in \{(i,j),(j,i) \}$,  
 there exists  
 $\frac{\delta^2 f}{\delta \phi_k\delta\phi_\ell} 
:\pi^{(N)} \times  \spn(\pi_k)\times \spn(\pi_\ell) \to \sR$ 
such that 
for all $\phi \in \pi^{(N)}$,
$\frac{\delta^2 f}{\delta \phi_k\delta\phi_\ell}(\phi,\cdot,\cdot) $
is bilinear
and 
  for all   $\phi'_k\in \spn(\pi_k) $,  $\frac{\delta^2 f}{\delta \phi_k\delta\phi_\ell}(\cdot; 
\phi'_k,\cdot) $
is a linear   derivative of $\frac{\delta f}{\delta \phi_k}(\cdot; \phi'_k)$
 with respect to $\pi_\ell$. 
 We refer  to $\frac{\delta^2 f}{\delta \phi_i\delta\phi_j}$ and 
 $\frac{\delta^2 f}{\delta \phi_j\delta\phi_i}$ 
as second-order linear derivatives of $f$ 
  with respect to $\pi_i\times \pi_j$.
\end{Definition}

\begin{Remark}
Definition \ref{def:linear_deri} allows   $(\pi_i)_{i\in I_N}$ to be    generic  convex sets  of measurable functions, which may not be vector spaces. 
This is important  for     games with stochastic/mixed policies, whose policy class $\pi_i$ consists of  functions mapping the system state to a probability measure over the action set. 
{
In   general, 
$f:\pi^{(N)}\to \sR$ 
  may have multiple  linear derivatives
  since the condition \eqref{eq:first_der_def} may not uniquely determine
the values of $ \frac{\delta f}{\delta \phi_i} (\phi ; \cdot)$ on 
$\spn{(\pi_i)}$, especially  when  $\pi_i$ is not a vector space.}
This non-uniqueness of linear derivatives will not affect our subsequent analysis, 
as   
 our  results hold for  any choices of $   \frac{\delta f}{\delta \phi_i}$ 
satisfying the properties   outlined in Definition \ref{def:linear_deri}.
For a specific game with sufficiently regular coefficients, there exist
natural choices of linear derivatives for value functions,
as  discussed in Sections \ref{sec:sde_probabilistic} and \ref{sec:sde_analytic}.  
 The same remark also applies to the second-order derivatives of $f$.
\end{Remark}

 We then present two lemmas regarding the linear derivative of   $f:\pi^{(N)}\to \sR$,
 which 
are crucial  for
 Theorem \ref{thm:symmetry_value_sufficient}.
Lemma \ref{lemma:derivative_line} shows that 
if $f$ has a linear derivative in $\pi_i$, then $f$ is differentiable along any  line segment within $\pi_i$.
Additionally, Lemma  \ref{lemma:multi-dimension_derivative} 
shows that  if $f$ is differentiable concerning all unilateral deviations, 
then  
  $f$ 
  is differentiable with respect to perturbations of all players'  policies,
and the derivative can be computed via  a chain rule. 
The proofs  
 are given in 
  Appendix 
 \ref{sec:differentiability_lemma}.

 \begin{Lemma}
\label{lemma:derivative_line}
Suppose  
$\pi^{(N)}$ is    convex,
$i\in I_N$, 
and   $f: \pi^{(N)}\to \sR$ 
has  a linear  derivative $ \frac{\delta f}{\delta \phi_i} $ with respect to $\pi_i$.
Let $\phi\in  \pi^{(N)}$, $\phi'_i\in \pi_i$,
and for each $\eps\in [0,1]$, let $\phi^\eps =( \phi_i+\eps(\phi'_i-\phi_i),\phi_{-i})$.
Then 
the map 
$[0,1]\ni \eps\mapsto f(\phi^\eps)\in \sR$
is differentiable and 
$\frac{\d }{\d \eps }f(\phi^\eps)=  \frac{\delta f}{\delta \phi_i} (\phi^\eps ; \phi'_i-\phi_i)
 $
 for all $\eps\in [0,1]$.  
\end{Lemma}

 \begin{Lemma}
 \label{lemma:multi-dimension_derivative}
 Suppose  
$\pi^{(N)} $ is   convex  and 
for all $i\in I_N$,
 $f: \pi^{(N)}\to \sR$
 has  a linear  derivative $ \frac{\delta f}{\delta \phi_i} $ with respect to $\pi_i$ 
 such that  
 for all  
  $z,\phi\in \pi^{(N)}$ and     $\phi'_i\in \pi_i $ ,  
$
[0,1]^N\ni {\eps}\mapsto  \frac{\delta f}{\delta \phi_i}(z+\eps\cdot (\phi-z) ;\phi'_i)
$
is continuous at $0$,  
where  
 $ z+\eps\cdot (\phi-z)\coloneqq (z_i+\eps_i({\phi}_{i}-z_i))_{i\in I_N}$.
 Then for all  $z,\phi\in \pi^{(N)}$,  
the map 
  $[0,1]\ni r\mapsto  f(z+r( \phi -z))\in \sR$ 
is differentiable  
and  
$
 \frac{\d}{\d r}   f(z+r( \phi -z)) 
 =\sum_{j=1}^N \frac{\delta  f}{ \delta\phi_j}
 (z+r( \phi-z);   \phi_j-z_j)$.  
 \end{Lemma}


We now     show that  if the value functions of a dynamics  game are sufficiently regular and have a symmetric Jacobian,  then this    game
  is a potential game, and its   dynamic potential function  can be constructed via the   linear derivative  of its value function.

 \begin{Theorem}
 \label{thm:symmetry_value_sufficient}
 Let 
 $\cG=(I_N, \cS, (A_i)_{i\in I_N}, \pi^{(N)}, (V_i)_{i\in I_N})$
 be a  game whose set of policy profiles  $\pi^{(N)}$
  is convex. 
 Suppose that for some $s_0\in \cS$ and for all  $i,j\in I_N$, 
the value function  $  V^{s_0}_i$ has   second-order linear derivatives with respect to 
$\pi_i\times \pi_j$ such that 
 for all 
 $ z=(z_j)_{j\in I_N},
 \phi=(\phi_j)_{j\in I_N}\in \pi^{(N)}$, 
  $\phi'_i,\tilde{\phi}'_i\in \pi_i$  
and $ {\phi}''_j \in \pi_j$, 
\begin{enumerate}[(1)]
 \item
 \label{item:integrable_p-Vi}
 (Boundedness.)
 $ \sup_{r,\eps\in [0,1]}\left|  \frac{\delta^2 V_i^{s_0}}{\delta\phi_i\delta\phi_j}\big(z+r( \phi^\eps-z); {\phi}'_i,{\phi}''_j\big)  \right|<\infty$,
where
 $\phi^\eps\coloneqq (\phi_i+\eps(\tilde{\phi}'_{i}-\phi_i),\phi_{-i})$; 
\item
\label{item:joint_continuity}
(Continuity.)
$[0,1]^N\ni {\eps}\mapsto  \frac{\delta^2 V_i^{s_0}}{\delta \phi_i\phi_j}(z+\eps\cdot (\phi-z) ;\phi'_i, \phi''_j)$
is continuous at $0$,  
where  
 $ z+\eps\cdot (\phi-z)\coloneqq (z_i+\eps_i({\phi}_{i}-z_i))_{i\in I_N}$;

  \item
  \label{item:symmetry-Vi} 
(Symmetric Jacobian.)
$\frac{\delta^2 V_i^{s_0}}{\delta\phi_i\delta\phi_j}
 (  \phi; \phi'_i, \phi''_j)=\frac{\delta^2 V_j^{s_0}}{\delta\phi_j\delta\phi_i}(  \phi ;\phi''_j, \phi'_i).
$
\end{enumerate}
Then     $\cG$ with initial state $s_0$  is a CLPG
and for any $z\in \pi^{(N)}$,
$\Phi^{s_0}:\pi^{(N)}\to\sR$ defined by
 \bb
 \label{eq:potential_variation}
 \Phi^{s_0}(\phi)=\int_0^1\sum_{j=1}^N  \frac{\delta V_j^{s_0}}{\delta\phi_j} (z+r(\phi-z);\phi_j-z_j) \d r
\ee
is  a potential function. 
If the above conditions hold for all $s_0\in \cS$,
then $\cG$ is an MPG
with a potential function $  (s,\phi)\mapsto  \Phi^{s}(\phi)$ defined as  in \eqref{eq:potential_variation}. 
\end{Theorem}


{
\begin{Remark}
\label{rmk:differential_characterization}

Theorem \ref{thm:symmetry_value_sufficient} generalizes the differential characterization for static potential games in \cite{monderer1996potential} to dynamic games. 
Indeed, 
 according to \cite[Theorem 4.5]{monderer1996potential}, a  static game
 is a static potential game 
if $ {\partial^2_{a_ia_j} V_i} = {\partial^2_{a_ja_i} V_j} $ for all $i,j\in I_N$, and for any $z\in \prod_{i=1}^N A_i $, a potential function is given by
$\Phi(a) =\int_0^1 \sum_{i=1}^N ({\partial_{a_i} V_i}) (z+r( a- z))(a_i- z_i)\d r,
$
with ${\partial_{a_i} V_i}$ and ${\partial^2_{a_ia_j} V_i} $ being the classical derivatives of $V_i$
when the action sets are either intervals \cite{monderer1996potential, ui2000shapley} or finite-dimensional spaces \cite{leonardos2022global,arefizadeh2023characterization}. In comparison, 
Theorem  \ref{thm:symmetry_value_sufficient} extends the characterization by replacing the classical derivative by the   linear derivative with respect to policies, given by $\frac{\delta V_i}{\delta a_i}(a; a'_i)\coloneqq ({\partial_{a_i} V_i})  (a) a'_i $  and  $\frac{\delta^2 V_i}{\delta a_i\delta a_j}( a; a'_i, a''_j)\coloneqq ({\partial^2_{a_ia_j} V_i})  ( a) a'_i a''_j $.


In fact,  the main technical challenge in establishing an analogous differential characterization for the dynamic game is defining an appropriate notion of the derivative of the value function with respect to policies, when the set of admissible policies $\pi^{(N)}$ may not form a vector space and lacks natural norms to define differentiability. Consequently, one cannot straightforwardly replace the classical derivatives in \cite{monderer1996potential} with  
Fr\'{e}chet derivatives
to characterize dynamic potential games. For instance, in continuous-time problems with continuous state spaces, determining the differentiability of value functions in Markov policies is a well-known challenge (see \cite{pradhan2024continuity, pradhan2024near} for some results on continuity). 

Our approach is to propose  the linear derivatives of value functions, defined along a line segment connecting two policies.  
The advantage of this linear derivative for dynamic games is to avoid   introducing a topology on $\pi^{(N)}$ and accommodate any convex policy classes. 

For continuous-time stochastic games with state dynamics governed by controlled diffusions, linear derivatives exist and can be easily computed under mild regularity assumptions on the policy classes and model coefficients. For such games, we will provide detailed  
analysis for    differentiable  policies and coefficients in Section \ref{sec:sde_probabilistic}  and for nondifferentiable policies and coefficients  in Section \ref{sec:sde_analytic}.
  
\end{Remark}}

\subsection{Probabilistic  Characterization of  Continuous-time   Potential Game}
\label{sec:sde_probabilistic}

Consider   the   game 
 $\cG_{\rm prob}=(I_N, \cS, (A_i)_{i\in I_N}, \pi^{(N)}, (V_i)_{i\in I_N})$,
   where   $I_N=\{1,\ldots, N\}$,
  $\cS= [0,T]\times \sR^{n_x}$
with   $T>0$ and  $n_x\in \sN$;
 $A_i\subset \sR^{k_i}$ with  $k_i\in \sN$  is player $i$'s action set; 
  $\pi^{(N)}= \prod_{i\in I_N} \pi_i$
  is the set of   admissible policy profiles, 
  where   player $i$'s   policy class    $\pi_i$ is     defined by:
   \begin{equation} 
 \label{eq:policy_probabilistic}
\pi_i \coloneqq 
\left\{ 
\phi: [0,T]\times \sR^{n_x}\to A_i
\,\middle\vert\, 
\begin{aligned}
&\textnormal{$\phi$ is   measurable and for all $t\in [0,T]$,  $ x\mapsto \phi(t,x)$ is}
\\
&\textnormal{twice continuously differentiable and }
\\
&
\textnormal{$\|\phi(\cdot,0)\|_0 +\| \p_x \phi\|_0+\| \p_{xx} \phi\|_0<\infty $}
%
\end{aligned}
\right\}.
\end{equation}
Let $n_a=\sum_{i\in I_N}k_i$ and $A^{(N)}=\prod_{i\in I_N}A_i\subset \sR^{n_a}$ 
be the set of   action profiles of all players. 
Let $(\Omega, \cF,   \sP)$
be a  probability space supporting    an $n_w$-dimensional    Brownian motion $W$. 
 For each  $i\in I_N$,   player $i$'s value function
 $ V_i : [0,T]\times \sR^{n_x}\times  \pi^{(N)}\to \sR$
is given by: 
for all    $(t,x)\in [0,T]\times \sR^{n_x}$ and $\phi\in \pi^{(N)}$, 
\bb\label{eq:cost_policy}
V_i^{t,x}(\phi)\coloneqq   \sE \left[\int_t^T f_i(s,X^{t,x,\phi}_s, \phi(s,X^{t,x,\phi}_s))\d s + g_i(X^{t,x,\phi}_T)\right],
\ee
where $f_i : [0,T]\times \sR^{n_x}\times   \sR^{n_a}\to \sR$
and $g_i :  \sR^{n_x} \to \sR$ 
are given   cost functions,
$X^{t,x,\phi}$ is the state process governed by    the  following     dynamics:
\bb
\label{eq:state_policy}
\d  X_s =B(t,X_s, \phi (s,X_s)) \d s +\Sigma(s,X_s, \phi (s,X_s))\d W_s, \quad s\in [t,T]; \quad X_t=x,
\ee
and  $( B,\Sigma): [0,T]\times \sR^{n_x}\times  \sR^{n_a}\to \sR^{n_x}\times  \sR^{n_x\times n_w}  $ 
are given coefficients.
Note that in \eqref{eq:cost_policy},  we write 
$V_i^{t,x}(\phi)\coloneqq V_i(t,x, \phi)$ ,
as we will analyze the     derivatives of value functions for each fixed 
initial condition $(t,x)$ (see Theorem \ref{thm:differential_MPG_proba}).

{ 
\begin{Remark}
\label{rmk:pde_C2_policy}
It has been shown in \cite{pradhan2024near} that under suitable assumptions, smooth policies are nearly optimal, thus justifying the restriction of the optimization problem to the space of smooth policies. Consequently, one can  
 characterize value function's linear derivatives via suitable sensitivity processes derived from a chain rule. 
Moreover, to facilitate algorithm design, most existing works restrict admissible policies to some parameterized families of smooth functions  (see \cite{leonardos2022global} and references therein). 
Finally, 
 this differentiability restriction   in \eqref{eq:policy_probabilistic}  can be removed when characterizing the linear derivatives via a PDE approach,
 as will be shown  in Section \ref{sec:sde_analytic}.

\end{Remark}
}

Throughout this section,
the  following regularity assumptions are imposed  on the action sets $(A_i)_{i\in I_N}$ and  coefficients  $(B,\Sigma,(f_i)_{i\in I_N},(g_i)_{i\in I_N})$.

\begin{Assumption}
\label{assum:regularity_N}
$A_i\subset \sR^{k_i}$, $i\in I_N$, is   nonempty and convex.
The functions 
$(B, \Sigma):  [0,T]\times \sR^{n_x}\times \sR^{n_a}\to \sR^{n_x}\times \sR^{n_x\times n_w} $, 
$f_i:  [0,T]\times \sR^{n_x}\times \sR^{n_a}\to  \sR $ and $g_i :  \sR^{n_x}  \to\sR $, $i\in I_N$,
are measurable and 
satisfy the following properties:
for all $i\in I_N$ and $t\in [0,T]$, 
\begin{enumerate}[(1)]
\item
$(x,a)\mapsto \big(B(t,x,a),\Sigma(t,x,a),f_i(t,x,a),g_i(x) \big)$ 
 are twice continuously differentiable.
 \item 
$(x,a)\mapsto (B(t,x,a),\Sigma(t,x,a))$ is of linear growth 
and their first and second   derivatives are   bounded (uniformly in $t$). 
 \item 
$(x,a)\mapsto (f_i(t, x,a),g_i(x))$  and their first and second derivatives   are 
of polynomial growth  (uniformly in $t$).
\end{enumerate}

\end{Assumption} 

Under Condition (H.\ref{assum:regularity_N}), 
 for all $(t,x)\in [0,T]\times \sR^{n_x}$ and $\phi\in \spn( \pi^{(N)})$, 
\eqref{eq:state_policy} admits a unique strong solution 
$X^{t,x,\phi}\in \cS^\infty([t,T];\sR^{n_x})$ and the value function 
$V^{t,x}_i(\phi)$ in \eqref{eq:cost_policy} is well-defined 
(see e.g., \cite[Theorem 3.3.1]{zhang2017backward}).
We then introduce several stochastic processes, which will be used to describe  
the    conditions under which 
  $\cG_{\rm prob}$ is a potential game.

 Fix      $(t,x)\in [0,T]\times \sR^{n_x}$ and $\phi\in \pi^{(N)}$.
For each $i\in I_N$ and $\phi'_i\in \spn(\pi_i)$,
let  $\frac{\delta X^{t,x}}{\delta \phi_i}(\phi;\phi'_i)\in  \cS^\infty([t,T];\sR^{n_x})$ be the solution to   the following     equation: 
 \begin{align} 
\label{eq:X_sensitivity_first}
\begin{split}
\d Y_s & = (\p_x B^\phi) (s, X^{t,x, \phi}_s  )  Y_s   \d s 
 +  \big((\p_x \Sigma^\phi) (s, X^{t,x, \phi}_s  )  Y_s  \big)
\d W_s
\\
&\quad 
+ (\p_{a_i} B^\phi)[\phi'_i](s, X^{t,x, \phi}_s  )  \d s
+ (\p_{a_i} \Sigma^\phi)[\phi'_i](s, X^{t,x, \phi}_s  )  \d W_s
\quad \fa  s\in [t,T];
\quad     Y_t   =0, 
 \end{split}
 \end{align}
 where   $B^\phi(t,x)\coloneqq B  (t, x, \phi(t,x))$,
$(\p_{a_i} B^\phi)[\phi'_i](t,x)\coloneqq(\p_{a_i}B) (t, x,  \phi(t,x) )  \phi'_i(t,x)$,
and 
$\Sigma^\phi(t,x)$ and $(\p_{a_i} \Sigma^\phi)[\phi'_i](t,x)$ are defined similarly. 
Here the differentiations  are taken componentwise, 
i.e., $ ((\p_x B^\phi) (s, X^{t,x, \phi}_s  )  Y_s)_{\ell}
=\sum_{j=1}^{n_x}\p_{x_j} B^\phi_\ell (s, X^{t,x, \phi}_s  )  (Y_s)_{j}$
for all $\ell=1,\ldots, n_x$. 
Equation \eqref{eq:X_sensitivity_first}
is the sensitivity equation of \eqref{eq:state_policy}
with respect to $\phi_i$ (see Lemma
 \ref{lemma:1st_derivative_state}).

In addition,  
for each $\phi'_i\in \spn(\pi_i)$ and $\phi''_j\in \spn(\pi_j)$, 
 let    $\frac{\delta^2 X^{t,x}}{\delta \phi_i\delta \phi_j}(\phi;\phi'_i,\phi''_j)
\in  \cS^\infty([t,T];\sR^{n_x}) $ be the solution to   the following   equation: 
 \begin{align} 
\label{eq:X_sensitivity_second}
\begin{split}
\d Z_s & = (\p_x B^\phi) (s, X^{t,x, \phi}_s  )  Z_s  \d s 
 +  \big((\p_x \Sigma^\phi) (s, X^{t,x, \phi}_s  )  Z_s  \big)
\d W_s
\\
&\quad 
+ F_B\left(s,X^{t,x, \phi}_s, \frac{\delta X^{t,x}}{\delta \phi_i}(\phi;\phi'_i), \frac{\delta X^{t,x}}{\delta \phi_j}(\phi;\phi''_j), \phi'_i, \phi''_j\right)\d s
\\
&\q 
+F_\Sigma\left(s,X^{t,x, \phi}_s, \frac{\delta X^{t,x}}{\delta \phi_i}(\phi;\phi'_i), \frac{\delta X^{t,x}}{\delta \phi_j}(\phi;\phi''_j), \phi'_i, \phi''_j\right) \d W_s
\q \fa s\in [t,T];
\q Z_t=0, 
 \end{split}
 \end{align}
 where 
 \begin{align}
 \label{eq:F_B}
 \begin{split}
&   F_B(\cdot, y_1,y_2,  \phi'_i, \phi''_j)
\\
&   =y_2^\top  (\p_{xx} B^\phi) (\cdot ) y_1
+\phi''_j(\cdot)^\top (\p_{a_ia_j}B)(\cdot,\phi(\cdot))\phi'_i(\cdot)
\\
&\q +
 \left[\phi''_j(\cdot)^\top (\partial_{xa_j}B)(\cdot,\phi(\cdot))
+\phi''_j(\cdot)^\top  (\partial_{aa_j}B)(\cdot,\phi(\cdot))\p_x \phi(\cdot)
\right]y_1
+(\p_{a_j}B)(\cdot,\phi(\cdot))\p_x \phi''_j(\cdot)y_1
\\
& \q +
y_2^\top  \left[(\partial_{a_ix}B)(\cdot,\phi(\cdot)) \phi'_i(\cdot) 
+\p_x \phi(\cdot) ^\top (\partial_{a_ia}B)(\cdot,\phi(\cdot))\phi'_i(\cdot)\right]+(\p_{a_i}B)(\cdot,\phi(\cdot)\p_x \phi'_i(\cdot)
y_2,
\end{split}
 \end{align}
and  $F_\Sigma(\cdot, y_1,y_2,  \phi'_i, \phi''_j)$ is defined similarly. 
 Equation \eqref{eq:X_sensitivity_second}
is the sensitivity equation of \eqref{eq:state_policy}
with respect to $\phi_i$ and $\phi_j$ (see Lemma
 \ref{lemma:2nd_derivative_state}).

Finally, 
let   $\alpha^{t,x,\phi}_s=\phi(s,X^{t,x,\phi}_s)$ for all $s\in [t,T]$,
 and  for each $i,j\in I_N$, define 
 $
 \frac{\delta \alpha^{t,x}}{\delta \phi_i}(\phi;\cdot):\spn(\pi_i)\to \cH^\infty([t,T];\sR^{n_a})$
by
 \begin{align}
 \label{eq:control_1st}
  \frac{\delta \alpha^{t,x}}{\delta \phi_i}(\phi;\phi'_i)_s
  & =(\p_x \phi)(s,X^{t,x,\phi}_s)  \frac{\delta X^{t,x}}{\delta \phi_i}(\phi;\phi'_i)_s+E_i\phi'_i(s, X^{t,x,\phi}_s),
 \end{align}
 and define 
 $
 \frac{\delta^2 \alpha^{t,x}}{\delta \phi_i\delta\phi_j}(\phi;\cdot,\cdot):\spn(\pi_i)\times\spn(\pi_j) \to \cH^\infty([t,T];\sR^{n_a})
 $
 by
  \begin{align}
   \label{eq:control_2nd}
  \begin{split}
   & \frac{\delta^2 \alpha^{t,x}}{\delta \phi_i\delta\phi_j}(\phi;\phi'_i,\phi''_j)_s 
   \\
    & =\left(\frac{\delta X^{t,x}}{\delta \phi_j}(\phi;\phi''_j)_s\right)^\top (\p_{xx} \phi)(s,X^{t,x,\phi}_s)  \frac{\delta X^{t,x}}{\delta \phi_i}(\phi;\phi'_i)_s+(\p_x\phi)(s, X^{t,x,\phi}_s)  \frac{\delta^2 X^{t,x}}{\delta \phi_i\delta\phi_j}(\phi;\phi'_i,\phi''_j)_s 
    \\
     &\q +E_j (\p_x \phi''_j)(s, X^{t,x,\phi}_s) \frac{\delta X^{t,x}}{\delta \phi_i}(\phi;\phi'_i)_s
     +E_i (\p_x \phi'_i)(s, X^{t,x,\phi}_s) \frac{\delta X^{t,x}}{\delta \phi_j}(\phi;\phi''_j)_s,
     \end{split}
\end{align}  
 where $E_\ell\in \sR^{n_a\times k_\ell}$,   $\ell\in \{i,j\}$,   is a block row matrix whose $\ell$-th row is a $k_\ell$-by-$k_\ell$ identity matrix (player $\ell$'s action)
 and other rows are zero (other players' actions). 

The following theorem
 characterizes the linear derivatives of value functions in \eqref{eq:cost_policy} using the   sensitivity processes of  the states and controls, 
and further constructs potential functions for  the   game $\cG_{\rm prob}$  
defined by \eqref{eq:policy_probabilistic}
and \eqref{eq:cost_policy}.

\begin{Theorem}
\label{thm:differential_MPG_proba}
 Assume (H.\ref{assum:regularity_N}).
 For each $i\in I_N$, 
 given $\pi_i$ and  $V_i$ in \eqref{eq:policy_probabilistic}
and \eqref{eq:cost_policy}, respectively.
  \begin{enumerate}[(1)]
  \item
  \label{item:prob_first}
For each  $(t,x)\in [0,T]\times\sR^{n_x}$ and 
   $i,j\in I_N$, 
define 
$\frac{\delta V^{t,x}_i}{\delta \phi_j}:\pi^{(N)} \times  \spn(\pi_k) \to \sR$
such that for all  
$\phi\in \pi^{(N)}$
and $\phi'_j\in \spn(\pi_j)$, 
  \begin{align}
  \label{eq:value_derivative_1st}
  \begin{split}
 \frac{\delta V^{t,x}_i}{\delta \phi_j} (\phi;\phi'_j)
 &   = \sE \bigg[\int_t^T
 \begin{pmatrix}
 \p_x f_i & \p_a f_i
 \end{pmatrix}
 (s,X^{t,x,\phi}_s, \alpha^{t,x,\phi}_s)
 \begin{pmatrix}
   \frac{\delta X^{t,x}}{\delta \phi_j} (\phi;\phi'_j)_s
 \\
 \frac{\delta \alpha^{t,x}}{\delta \phi_j} (\phi;\phi'_j)_s
 \end{pmatrix}
 \d s  \bigg]
 \\
&\q + \sE \bigg[(\p_xg_i)(X^{t,x,\phi}_T)  \frac{\delta X^{t,x}}{\delta \phi_j} (\phi;\phi'_j)_T\bigg].
\end{split}
\end{align}
Then $\frac{\delta V^{t,x}_i}{\delta \phi_j} $
is a linear derivative of $\phi\mapsto V^{t,x}_i(\phi)$ 
with respect to $\pi_j$.

\item
\label{item:prob_second}
For each  $(t,x)\in [0,T]\times\sR^{n_x}$ and 
   $i, k,\ell\in I_N$,
define 
$\frac{\delta^2 V^{t,x}_i}{\delta \phi_k\delta \phi_\ell}:\pi^{(N)} \times  \spn(\pi_k)\times \spn(\pi_\ell) \to \sR$
such that  for all 
$\phi\in \pi^{(N)}$, $\phi'_k\in \spn(\pi_k)$
and $\phi''_\ell\in \spn(\pi_\ell)$, 
  \begin{align}
  \label{eq:value_derivative_2nd}
& \frac{\delta^2 V^{t,x}_i}{\delta \phi_k \delta \phi_\ell} (\phi;\phi'_k, \phi''_\ell)
\nb
\\
 &\q   = \sE \left[\int_t^T
    \begin{pmatrix}
  \frac{\delta X^{t,x}}{\delta \phi_\ell} (\phi;\phi''_\ell)_s 
  \nb
  \\
    \frac{\delta \alpha^{t,x}}{\delta \phi_\ell} (\phi;\phi''_\ell)_s 
\end{pmatrix}^\top
 \begin{pmatrix}
 \p_{xx} f_i   & \p_{ax} f_i
\\
\p_{x  a}f_i  & \p_{a  a} f_i
\end{pmatrix}(s,X^{t,x,\phi}_s, \alpha^{t,x,\phi}_s)
    \begin{pmatrix}
  \frac{\delta X^{t,x}}{\delta \phi_k} (\phi;\phi'_k)_s 
  \\
    \frac{\delta \alpha^{t,x}}{\delta \phi_k} (\phi;\phi'_k)_s 
\end{pmatrix} \d s\right]
\nb \\
&\qq  +  \sE \bigg[\int_t^T
 \begin{pmatrix}
 \p_x f_i & \p_a f_i
 \end{pmatrix}
 (s,X^{t,x,\phi}_s, \alpha^{t,x,\phi}_s)
 \begin{pmatrix}
   \frac{\delta^2 X^{t,x}}{\delta \phi_k\phi_\ell} (\phi;\phi'_k,\phi''_\ell)_s 
\nb \\
 \frac{\delta^2 \alpha^{t,x}}{\delta \phi_k\phi_\ell} (\phi;\phi'_k,\phi''_\ell)_s 
 \end{pmatrix}
 \d s  \bigg]
 \nb \\
 &\qq
+\sE\left[\left(\frac{\delta X^{t,x}}{\delta \phi_\ell} (\phi;\phi''_\ell)_T\right)^\top (\p_{xx}g_i)(X^{t,x,\phi}_T)  \frac{\delta X^{t,x}}{\delta \phi_k} (\phi;\phi'_k)_T\right]
\nb \\
 &\qq + 
 \sE\left[ (\p_xg_i)(X^{t,x,\phi}_T)  \frac{\delta^2 X^{t,x}}{\delta \phi_k\phi_\ell} (\phi;\phi'_k,\phi''_\ell)_T\right].
\end{align}  
Then  
$\frac{\delta^2 V^{t,x}_i}{\delta \phi_i\delta \phi_j}$ 
and $\frac{\delta^2 V^{t,x}_i}{\delta \phi_j\delta \phi_i}$
are  second-order linear derivatives of $\phi\mapsto V^{t,x}_i(\phi)$  
with respect to $\pi_i\times \pi_j$.

\item 
\label{item:prob_potential}
 If there exists $(t,x)\in [0,T]\times\sR^{n_x}$ such that for all   
$ \phi\in \pi^{(N)}$, 
$i,j\in I_N$,
  $\phi'_i\in \pi_i$  
and $ {\phi}''_j \in \pi_j$, 
\bb
\label{eq:prob_symmetry}
\frac{\delta^2 V^{t,x}_i}{\delta \phi_i \delta \phi_j} (\phi;\phi'_i, \phi''_j)
=\frac{\delta^2 V^{t,x}_j}{\delta \phi_j \delta \phi_i} (\phi;\phi''_j, \phi'_i)
\ee
with $\frac{\delta^2 V^{t,x}_i}{\delta \phi_i \delta \phi_j}$ 
and $\frac{\delta^2 V^{t,x}_j}{\delta \phi_j \delta \phi_i} $ 
defined in Item \ref{item:prob_second},
then 
 $\cG_{\rm prob}$ 
 with initial condition $(t,x)$
 is a CLPG 
 and 
a potential function is given by
 \eqref{eq:potential_variation} with  
$(  \frac{\delta V^{t,x}_i}{\delta \phi_i})_{i\in I_N}$
defined in Item \ref{item:prob_first}.

Moreover, if \eqref{eq:prob_symmetry} holds for all $(t,x)\in [0,T]\times\sR^{n_x}$,
then  $\cG_{\rm prob}$ is an MPG.

  \end{enumerate}

\end{Theorem}


 Theorem 
\ref{thm:differential_MPG_proba} generalises 
Theorem \ref{thm:distributed_game}
by accommodating interconnected states among players. 
Indeed, for the distributed game $ \cG_{\rm dist}$     in Definition \ref{example:distributed_game},
player $i$'s state and control processes depend solely on her own policy $\phi_i$, resulting in zero derivatives with respect to player $j$'s policy $\phi_j$ for all $j \neq i$. 
In this distributed game,    \eqref{eq:value_derivative_2nd}
can be simplified and \eqref{eq:prob_symmetry} follows from   condition \eqref{eq:symmetric_hessian_distributed}.

  { 
To prove 
 Theorem 
\ref{thm:differential_MPG_proba}, one first   shows that 
\eqref{eq:X_sensitivity_first}
and \eqref{eq:X_sensitivity_second}
characterize the derivatives of  $X^{t,x,\phi}$
with respect to policies, 
using the pathwise differentiability 
of solutions to SDEs  \cite[Theorem 4, p.~105]{krylov2008controlled}.
Then, by the regularity assumptions and 
 the chain rule
\cite[Theorem 9, p.~97]{krylov2008controlled}, 
one can show that 
 \eqref{eq:control_1st} and  \eqref{eq:control_2nd} 
characterize
the derivatives of   $\alpha^{t,x,\phi}$ with respect to policies,
and further establish 
\eqref{eq:value_derivative_1st}
and \eqref{eq:value_derivative_2nd}. 
The detailed proof is given in Appendix \ref{sec:proof_sde_probabilistic}.

}
 
\subsection{PDE  Characterization of  Continuous-time  Potential Games}
  \label{sec:sde_analytic}
  This section studies the continuous-time   games 
    in Section \ref{sec:sde_probabilistic} via a PDE approach.
For general games with 
{H\"{o}lder continuous policies} and 
nondegenerate diffusion coefficients, potential functions and the linear derivatives of value functions are characterized using a system of linear  PDEs.

  Consider  a      game 
 $\cG_{\rm analyt}=(I_N, \cS, (A_i)_{i\in I_N}, \pi^{(N)}, (V_i)_{i\in I_N})$,
with   $I_N=\{1,\ldots, N\}$,
  $\cS= [0,T]\times \sR^{n_x}$, 
   $A_i\subset \sR^{k_i}$ 
   and $n_a=\sum_{i\in I_N}k_i$
   as in  Section \ref{sec:sde_probabilistic}. 
The set  
  $\pi^{(N)}=\prod_{i\in I_N}\pi_i$ consists of all     policy profiles,
  where   
   player $i$'s   policy class $\pi_i$ is  defined by (cf.~\eqref{eq:policy_probabilistic}):
\bb
\label{eq:policy_pde}
\pi_i=\{\phi:[0,T]\times\sR^{n_x}\to A_i\mid \phi\in C^{\gamma/2,\gamma}([0,T]\times \sR^{n_x};\sR^{k_i})\}
\quad \textnormal{for some $\gamma\in (0,1]$.}
\ee
The value functions $(V_i)_{i\in I_N}$ are defined  in a manner analogous to \eqref{eq:cost_policy}. However, we impose a different set of regularity conditions compared to (H.\ref{assum:regularity_N}) to facilitate the subsequent PDE analysis.

\begin{Assumption}
\label{assum:Holder_regularity}
$A_i\subset \sR^{k_i}$, $i\in I_N$, is   nonempty and convex.
There exists $\beta\in (0,1]$ and $\kappa>0$ such that 
\begin{enumerate}[(1)]
\item
\label{item:holder}
 $(B, \Sigma)\in C^{ {\beta}/{2},\beta,\beta}( [0,T]\times \sR^{n_x}\times \sR^{n_a} ; \sR^{n_x}\times \sR^{n_x\times n_x}) $ and 
for all $i\in I_N$, $f_i\in C^{ {\beta}/{2},\beta,2+\beta}( [0,T]\times \sR^{n_x}\times \sR^{n_a};  \sR )$ and $g_i \in C^{2+\beta}( \sR^{n_x} ;\sR) $;
\item
\label{item:non-degeneracy}
  For all $t\in [0,T]$ and $(x,a)\in \sR^{n_x}\times \sR^{n_a}$, 
$\Sigma(t,x,a)$ is symmetric and satisfies 
$\xi^\top \Sigma(t,x,a)\xi \ge \kappa |\xi|^2$ for all $\xi\in \sR^{n_x}$.
\end{enumerate}
\end{Assumption} 

{ 

\begin{Remark}
\label{rmk:pde_irregular_policy}
Condition (H.\ref{assum:Holder_regularity}) and  \eqref{eq:policy_pde}
allow for    H\"{o}lder continuous 
state coefficients and policies,
relaxing the twice differentiability condition required    
 in Section \ref{sec:sde_probabilistic}   (cf.~(H.\ref{assum:regularity_N})
and \eqref{eq:policy_probabilistic}). The  H\"{o}lder continuity of  policies,
along with the uniform ellipticity of the diffusion coefficient,
ensures the uniqueness in law of the state process and facilitates  characterizing the value functions' linear derivatives  by    solutions of   associated PDEs in     H\"{o}lder spaces. If the diffusion coefficient of \eqref{eq:state_policy} is uncontrolled, one can in fact consider measurable policies and establish similar characterizations by  analyzing PDEs in suitable Sobolev spaces (see e.g., \cite{ito2021neural, sethi2024entropy, pradhan2024near}).
 
 \end{Remark}
  }

Under (H.\ref{assum:Holder_regularity}), 
 by \cite[Theorem 2]{mishura2020existence},  
  for each $(t,x)\in [0,T]\times \sR^{n_x}$ and $\phi\in \pi^{(N)}$, 
 the     state   dynamics \eqref{eq:state_policy}
 admits a unique weak solution $X^{t,x,\phi}$ on a probability space 
 $(\Omega, \cF, \bar{\sP})$ and the value functions 
  \bb\label{eq:cost_pde}
  V^\phi_i(t,x) = \sE^{\bar{\sP}} \left[\int_t^T f_i(s,X^{t,x,\phi}_s, \phi(s,X^{t,x,\phi}_s))\d s + g_i(X^{t,x,\phi}_T)\right],
  \q i\in I_N
\ee
  are well-defined. 
In contrast to \eqref{eq:cost_policy}, in \eqref{eq:cost_pde}, we use the notation $V^\phi_i(t,x)\coloneqq V_i(t,x, \phi)$ to analyze the value functions for all initial conditions $(t, x)$ simultaneously.
By standard   regularity results of linear PDEs 
  (see e.g., \cite[Theorem 5.1, p.~320]{ladyzhenskaia1988linear})
  and It\^{o}'s formula, 
  $(t,x)\mapsto V^{\phi}_i(t,x)$ is the unique classical solution to 
the following PDE:
 for all $(t,x)\in [0,T]\times \sR^{n_x}$, 
 \begin{align}
 \label{eq:pde_value}
\begin{split}
&\p_t V(t,x) + \cL^\phi V(t,x)+f_i(t,x,\phi(t,x))=0;   \quad  
  V(T,x)= g_i(x),
\end{split}
\end{align}
 where $\cL^\phi$ is the generator of \eqref{eq:state_policy} defined by
$$
   \cL^\phi u(t,x) = \frac{1}{2}\operatorname{tr} \big(\Sigma\Sigma^\top(t,x,\phi(t,x))(\p_{xx} u)(t,x)\big)+ B(t,x,\phi(t,x))^\top (\p_x u)(t,x).
   $$ 
   
 We now introduce several    equations associated with 
\eqref{eq:pde_value}, 
whose solutions characterize the linear derivatives of $(V_i)_{i\in I_N}$. 
Fix $i\in I_N$. 
For each $k\in I_N$, 
consider  
$\frac{\delta V_i}{\delta \phi_k}: [0,T]\times \sR^{n_x}\times  \pi^{(N)}\times \spn(\pi_k)
\to \sR$:
\bb
\label{eq:value_1st_pde}
  (t,x,\phi, \phi'_k)\mapsto 
  \frac{\delta V_i}{\delta \phi_k}(t,x,\phi, \phi'_k)
\coloneqq \frac{\delta V^\phi_i}{\delta \phi_k}(t,x; \phi'_k), 
\ee
where 
  $(t,x)\mapsto   \frac{\delta V^\phi_i}{\delta \phi_k}(t,x; \phi'_k)$  solves  the following equation: for all $(t,x)\in [0,T]\times \sR^{n_x}$, 
\bb
\label{eq:pde_value_1st}
 \left\{
 \begin{aligned}
 &\p_t U (t,x) +\cL^{\phi} U (t,x) + (\p_{a_k} H_i)\big(t,x,(\p_x V^\phi_i)(t,x),(\p_{xx} V^\phi_i)(t,x),\phi(t,x)\big)^\top \phi'_k(t,x)=0,
 \\
& U(T,x)= 0,
 \end{aligned}\right.
 \ee
 with the function  
$H_i:[0,T]\times \sR^{n_x}\times \sR^{n_x}\times   \sR^{n_x\times n_x}\times \sR^{n_a}\to \sR$  
defined as follows:
\bb
\label{eq:Hamiltonian_H}
H_i(t,x,y,z,a)= \frac{1}{2}\operatorname{tr}  \big(  (\Sigma\Sigma^\top ) (t,x,a ) z  \big)
+ B(t,x, a)^\top y+ f_i(t,x, a).
\ee
 Equation \eqref{eq:pde_value_1st} 
 describes  the sensitivity   of  \eqref{eq:pde_value}
with respect to $\phi_k$.
It  is   derived 
by considering  the value function $V^{\phi^\eps}$
with 
 $\phi^\eps=(\phi_k+\eps(\phi'_k-\phi_k),\phi_{-k})$ for   $\eps\in [0,1)$, 
and then 
  differentiating   the corresponding PDE  \eqref{eq:pde_value}  with respect to $\eps$.

 In addition, for each $k, \ell\in I_N$, 
consider the map 
$\frac{\delta^2 V^\phi_i}{\delta \phi_k\phi_\ell}: 
[0,T]\times \sR^{n_x}\times  \pi^{(N)}\times \spn(\pi_k)\times  \spn(\pi_\ell)\to \sR$:
\bb
\label{eq:value_2nd_pde}
  (t,x,\phi, \phi'_k,\phi''_\ell)\mapsto 
  \frac{\delta^2 V_i}{\delta \phi_k\phi_\ell}(t,x, \phi, \phi'_k,\phi''_\ell)
  \coloneqq 
\frac{\delta^2 V^\phi_i}{\delta \phi_k\phi_\ell}(t,x; \phi'_k,\phi''_\ell), 
\ee
 where  
 $ (t,x)\mapsto 
\frac{\delta^2 V^\phi_i}{\delta \phi_k\phi_\ell}(t,x; \phi'_k,\phi''_\ell)$ 
solves 
the following equation:  for all $(t,x)\in [0,T]\times \sR^{n_x}$, 
    \bb
\label{eq:pde_value_2nd}
 \left\{
 \begin{aligned}
 &\p_t  W (t,x) +\cL^{\phi} W (t,x) 
\\
& \qq
+
 (\p_{a_\ell} L)\left(t,x, \left(\p_x \frac{\delta V^\phi_i}{\delta \phi_k}\right)(t,x; \phi'_k) ,
  \left(\p_{xx} \frac{\delta V^\phi_i}{\delta \phi_k}\right)(t,x; \phi'_k) ,\phi(t,x)\right) ^\top \phi''_\ell(t,x)
\\
&  \qq
 +  (\p_{a_k} L)\left(t,x, \left(\p_x \frac{\delta V^\phi_i}{\delta \phi_\ell}\right)(t,x; \phi''_\ell) ,
\left(\p_{xx} \frac{\delta V^\phi_i}{\delta \phi_\ell}\right)(t,x; \phi''_\ell),
\phi(t,x)\right) ^\top
 \phi'_k(t,x)
 \\
 & \qq
 +   \phi''_\ell(t,x)^\top (\p_{a_ka_\ell} H_i)\big(t,x,(\p_x V^{\phi}_i)(t,x),(\p_{xx} V^{\phi}_i)(t,x),\phi(t,x)\big)  \phi'_k(t,x)
 =0,
 \\
& W(T,x)=0,
 \end{aligned}\right.
 \ee
 with the function  
$L:[0,T]\times \sR^{n_x}\times \sR^{n_x}\times   \sR^{n_x\times n_x}\times \sR^{n_a}\to \sR$  defined by
 \bb
\label{eq:L}
L(t,x,y,z,a)= \frac{1}{2}\operatorname{tr}  \big(  (\Sigma\Sigma^\top ) (t,x,a ) z  \big)
+ B(t,x, a)^\top y.
\ee
 Equation \eqref{eq:pde_value_2nd}
 describes  the sensitivity   of  \eqref{eq:pde_value}
with respect to $\phi_k$ and $\phi_\ell$.

Now, via the 
Schauder's estimate  \cite[Theorem 5.1, p.~320]{ladyzhenskaia1988linear} for 
 \eqref{eq:pde_value}, \eqref{eq:pde_value_1st} 
and   \eqref{eq:pde_value_2nd}   in suitable H\"{o}lder spaces
one can show  that 
the maps 
 \eqref{eq:value_1st_pde}
 and 
 \eqref{eq:value_2nd_pde}
 are well-defined, 
 and are  the linear derivatives of value function  $V_i$.

  \begin{Theorem}
\label{thm:differential_MPG_pde}
   Assume  (H.\ref{assum:Holder_regularity}).
    For each $i\in I_N$, 
 let $\pi_i$ and  $V_i$ be defined by \eqref{eq:policy_pde}
and \eqref{eq:cost_pde}, respectively.

\begin{enumerate}[(1)]
\item 
  \label{item:pde_first}
For all  $i,k\in I_N$, the map $\frac{\delta V_i}{\delta \phi_k}$ in  \eqref{eq:value_1st_pde} is well-defined, 
and for all $(t,x)\in [0,T]\times \sR^{n_x}$,
  $(\phi,\phi'_k)\mapsto \frac{\delta V^{\phi}_i}{\delta \phi_k}(t,x; \phi'_k) $
is a linear derivative of $\phi\mapsto V^{\phi}_i(t,x)$ 
with respect to $\pi_k$.
 \item 
  \label{item:pde_second}
For all $i,k,\ell\in I_N$, the map $\frac{\delta^2 V_i}{\delta \phi_k\delta \phi_\ell}$ in  \eqref{eq:value_2nd_pde} is well-defined, 
and for all $(t,x)\in [0,T]\times \sR^{n_x}$,
  $(\phi,\phi'_k,\phi''_\ell)\mapsto \frac{\delta^2 V^\phi_i}{\delta \phi_k\delta \phi_\ell}(t,x; \phi'_k, \phi''_\ell) $
  and 
  $(\phi,\phi''_\ell,\phi'_k)\mapsto \frac{\delta^2 V^\phi_i}{\delta \phi_\ell \delta \phi_k}(t,x; \phi''_\ell, \phi'_k) $  
are  second-order linear derivatives of $\phi\mapsto V^{\phi}_i(t,x)$ 
with respect to $\pi_k\times \pi_\ell$.

\item 
\label{item:pde_potential}
  If 
  there exists $(t,x)\in [0,T]\times\sR^{n_x}$
  such that 
  for all   
$ \phi\in \pi^{(N)}$, 
$i,j\in I_N$,
  $\phi'_i\in \pi_i$  
and $ {\phi}''_j \in \pi_j$, 
\bb
\label{eq:pde_symmetry}
 \frac{\delta^2 V^{\phi}_i}{\delta \phi_i \delta \phi_j} (t,x;\phi'_i, \phi''_j)
=\frac{\delta^2 V^{\phi}_j}{\delta \phi_j \delta \phi_i} (t,x;\phi''_j, \phi'_i),
\ee
with $ \frac{\delta^2 V^{\phi}_i}{\delta \phi_i \delta \phi_j} $ 
and $\frac{\delta^2 V^{\phi}_j}{\delta \phi_j \delta \phi_i}  $ 
defined in Item \ref{item:pde_second},
then 
 $\cG_{\rm analyt}$    with initial condition $(t,x)$
is a CLPG 
and 
a potential function is given by
 \eqref{eq:potential_variation} 
with  
$(  \frac{\delta V_i}{\delta \phi_i})_{i\in I_N}$
defined in Item \ref{item:pde_first}.
Moreover, if \eqref{eq:pde_symmetry} holds for all $(t,x)\in [0,T]\times\sR^{n_x}$,
then  $\cG_{\rm analyt}$ is an MPG. 

\end{enumerate}
   
 \end{Theorem}

{ 
Compared with   the probabilistic characterizations in Theorem \ref{thm:differential_MPG_proba},
 \eqref{eq:pde_value_1st} and  \eqref{eq:pde_value_2nd}  do not need the spatial derivatives of   policies, allowing for  characterizing the linear derivatives of value functions with a more general class of nondifferentiable policies. 
 }


\subsection{Continuous-time Linear Quadratic   Potential Game}
 Theorems \ref{thm:differential_MPG_proba} and 
  \ref{thm:differential_MPG_pde}
   characterize   continuous-time MPGs differently 
depending on the regularity   of   model coefficients and players' admissible policies.
This section shows that
  these two characterizations  coincide  for a class of linear-quadratic (LQ) games.
  Moreover, by leveraging the LQ structure of the problem, 
 simpler characterizations of MPGs can be obtained through a system of ODEs.
In this section, the time variable of all coefficients is dropped when there is no risk of  confusion.

  Consider  the     game 
 $\cG_{\rm LQ}=(I_N, \cS, (A_i)_{i\in I_N}, \pi^{(N)}, (V_i)_{i\in I_N})$,
where $I_N=\{1,\ldots, N\}$, 
$\cS= [0,T]\times \sR^{n_x}$,   $A_i=\sR^{k_i}$, and 
   $\pi^{(N)}=\prod_{i\in I_N}\pi_i$ is the set of   policy profiles,  
  where  player $i$'s policy class 
  $\pi_i$ contains   linear policies   defined by
\bb
\label{eq:policy_lq}
\pi_i=\{\phi:[0,T]\times\sR^{n_x}\to \sR^{k_i}\mid \phi(t,x)=K_i(t)x, \; K_i \in C([0,T];  \sR^{k_i\times n_x})\}.
\ee
With an abuse of notation,
we identify 
 a policy $\phi_i\in \pi_i$  with the feedback map $K_i $, 
and 
identify 
a policy profile 
  $\phi\in \pi^{(N)}$ with  feedback maps $K=(K_i)_{i\in I_N}$.
For each $i\in I_N$, define  player $i$'s   value function 
$V_i:[0,T]\times \sR^{n_x} \times \pi^{(N)}\to \sR$
by
 \begin{align}
 \label{eq:value_lq} 
 \begin{split}
  V^\phi_i(t,x)
  & =
   \frac{1}{2}    \sE  \bigg[\int_t^T
\left(\big(X^{t,x,\phi}_s\big)^\top  Q_i(t) X^{t,x,\phi}_s
 +
\big(K(s)X^{t,x,\phi}_s\big)^\top R_i(t) K(s) X^{t,x,\phi}_s \right)
\d s
\\
&\qq 
+ 
(X^{t,x,\phi}_T)^\top  G_i X^{t,x,\phi}_T
\bigg], 
\end{split}
\end{align}
  where  $X^{t,x,\phi}$ satisfies  the following  dynamics:
\bb
\label{eq:state_lq}
\d  X_s =\left(A(s)X_s+B(s)K(s) X_s    \right)\d s +\sigma \d W_s, \quad s\in [t,T]; \quad X_t=x,
\ee
and  $W$ is a given $n_w$-dimensional standard Brownian motion. 
We assume the following standard conditions   for the coefficients of \eqref{eq:value_lq} and \eqref{eq:state_lq}: 
$A\in C([0,T];\sR^{n_x\times n_x})$, 
$B\in C([0,T];\sR^{n_x\times n_a})$,
$\sigma\in \sR^{n_x\times n_w}$, and for all $i\in I_N$, 
$Q_i\in C([0,T];\sS^{n_x})$,
$R_i\in C([0,T];\sS^{n_a})$
and $G_i\in  \sS^{n_x}$. 
 Here for ease of exposition and simple characterization, we focus on LQ games with (possibly degenerate)  uncontrolled additive noises. As will be clear from the analysis that 
similar characterizations can be established for LQ games with controlled multiplicative noises.

 \paragraph{Probabilistic  characterization in  LQ  games.}
  
  To characterize the      LQ games,  
 write  $B$ and $R_i$ in the following block form:
\bb\label{eq:BR_block}
B= (B_1, \ldots,   B_N),
\quad 
 R_i=((R_i)_1, \ldots,   (R_i)_N)
 =((R_i)_{h\ell})_{h,\ell\in I_N}
\ee
with   $B_\ell \in C([0,T]; \sR^{n_x\times k_\ell})$,  
 $(R_i)_{ \ell} \in C([0,T];  \sR^{n_a\times k_\ell})$
 and  $(R_i)_{ h\ell} \in C([0,T];  \sR^{k_h\times k_\ell})$
 for all $h,\ell\in I_N$. 
Fix   $i,h, \ell\in I_N$ and $(t,x)\in [0,T]\times \sR^{n_x}$.  
The probabilistic  characterization \eqref{eq:value_derivative_1st}
of  
  $ \frac{\delta V^{t,x}_i}{\delta \phi_h} (\phi; \phi'_h) $   simplifies into 
\begin{align}
  \label{eq:value_derivative_1st_lq}
  \begin{split}
 \frac{\delta V^{t,x}_i}{\delta \phi_h} (\phi;\phi'_h)
&    = \sE \bigg[\int_t^T 
 \bigg( X_s^\top Q_i  Y^{h}_s
 +(K X_s)^\top  
 \Big(R_i  K Y^{h}_s 
 +(R_i)_h   K'_h X_s \Big)
 \bigg)
  \d s  + X_T^\top G_i  Y^{h}_T\bigg],
\end{split}
\end{align}
where 
$(R_i)_{ h}$ is defined in \eqref{eq:BR_block},
  $X=X^{t,x,\phi}$ is the   state process satisfying      \eqref{eq:state_lq}, 
$Y^{h}$ is  the sensitivity process of $X$ with respect $K'_h$ satisfying 
(cf.~\eqref{eq:X_sensitivity_first}): 
 \begin{align} 
\label{eq:X_sensitivity_first_lq}
\begin{split}
\d Y^{h}_s & = \big( (A +B K )  Y^{h}_s   
+B_h K'_h X_s   \big)  \d s 
\quad \fa  s\in [t,T]; \q 
Y^{h}_t=0.
 \end{split}
 \end{align}
Moreover, the  probabilistic  characterization
  \eqref{eq:value_derivative_2nd}
of    
  $ \frac{\delta^2 V^{t,x}_i}{\delta \phi_h\delta \phi_\ell}(\phi; \phi'_h, \phi''_\ell)   $
simplifies into  
\begin{align}
  \label{eq:value_derivative_2nd_lq}
 \begin{split} 
 \frac{\delta^2 V^{t,x}_{i}}{\delta \phi_h \delta \phi_\ell} (\phi;\phi'_h, \phi''_\ell)
&    = \sE \left[\int_t^T
\left( Y^\ell_s  Q_i  Y^h_s
 +\big(K  Y^\ell_s +E_\ell K''_\ell X_s\big)^\top R_i 
 \big(K Y^h_s +E_h K'_h X_s\big)
\right)  \d s\right]
 \\
&\q   +  \sE \bigg[\int_t^T
\left(
X^\top_s Q_i Z_s
+(K  X_s)^\top R_i 
\big( K Z_s +E_\ell K''_\ell  Y^h_s+E_hK'_h  Y^\ell_s\big)
\right)
 \d s  \bigg]
 \\
 &\q 
+\sE\left[ (Y^\ell_T  )^\top G_i Y^h_T +X^\top_T G_i Z_T\right],
 \end{split}
\end{align}  
where     $E_j\in \sR^{n_a\times k_j}$, $j\in \{h,\ell\}$,  is a block row matrix defined as in \eqref{eq:control_2nd},
 $Z$ is  the sensitivity process of $X$ with respect $K'_h$ and $K''_\ell$ satisfying 
(cf.~\eqref{eq:X_sensitivity_second}): 
\begin{align} 
\label{eq:X_sensitivity_second_lq}
\begin{split}
\d Z_s & = \left( (A+BK) Z_s   +  B_\ell K''_\ell Y^h_s
 +B_h K'_h Y^\ell_s 
 \right)\d s
\q \fa s\in [t,T];
\q Z_t=0.
 \end{split}
 \end{align}
 
  \paragraph{PDE  characterization in  LQ  games.}
For  the PDE characterization,
define for all 
  $\phi\in \pi^{(N)}$, $\phi'_h=K'_h\in \pi_h$ and $\phi''_\ell=K''_\ell\in \pi_\ell$, 
\begin{subequations}
\label{eq:quadratic_ansatz}
\begin{align}
{{V}^\phi_i}(t,x)&\coloneqq \frac{1}{2}x^\top \Psi_i(t)x+\psi_i(t) \q \fa (t,x)\in [0,T]\times \sR^{n_x}, 
\label{eq:quadratic_value}
\\
{\frac{\delta \ol{ V}^\phi_i}{\delta \phi_h}}(t,x; \phi'_h)&\coloneqq \frac{1}{2}x^\top \Theta^{h}_i(t)x+\theta^{h}_i(t) \q \fa (t,x)\in [0,T]\times \sR^{n_x}, 
\label{eq:quadratic_1st}
\\
{\frac{\delta^2 \ol{V}^\phi_i}{\delta \phi_h\delta \phi_\ell}}(t,x; \phi'_h, \phi''_\ell)&\coloneqq \frac{1}{2}x^\top \Lambda^{h,\ell}_i(t)x+\lambda^{h, \ell}_i(t)\q \fa (t,x)\in [0,T]\times \sR^{n_x},
\label{eq:quadratic_2nd}
\end{align}
\end{subequations}
%
where $ \Psi_i,  \Theta^h_i ,  \Lambda^{h,\ell}_i  \in C([0,T];\sS^{n_x})$ satisfies  
\begin{subequations}
\label{eq:ode_lq}
 \begin{empheq}[left=\empheqlbrace]{align}
  \dot \Psi_i&+ \left(A+BK\right)^\top \Psi_i+ \Psi_i \left(A+BK\right) 
+Q_i+K^\top R_i K=0 \q \fa t\in [0,T];
\q \Psi_i(T)=G_i,
\label{eq:ode_psi}
\\
\dot \Theta^{h}_i &+ \left(A+BK \right)^\top \Theta^{h}_i+ \Theta^{h}_i \left(A+BK \right) 
 + \Psi_i B_h K'_h+(B_h K'_h)^\top  \Psi_i
 \nb
\\
&   +  K^\top (R_i)_{ h}K'_h
+ (K'_h)^\top \big((R_i)_{ h}\big)^\top K
=0  \q \fa t\in [0,T]; \q \Theta^{h}_i(T)=0, 
\label{eq:ode_theta}
\\ 
\dot \Lambda^{h,\ell}_i &+ \left(A+BK \right)^\top \Lambda^{h,\ell}_i+ \Lambda^{h,\ell}_i \left(A+BK \right) 
\nb \\
&  + \left( \Theta^{h}_iB_\ell K''_\ell +(B_\ell K''_\ell)^\top  \Theta^{h}_i  \right)
+\left(\Theta^{\ell}_i B_h K'_h
 +(B_h K'_h
)^\top  \Theta^{\ell}_i \right) 
\nb \\
& 
 +(K''_\ell)^\top (R_i)_{\ell h} K'_h
 +(K'_h)^\top \big((R_i)_{\ell h}\big)^\top  K''_\ell
=0 \q \fa t\in [0,T]; \q \Lambda^{h,\ell}_i(T)=0, 
\label{eq:ode_lambda}
  \end{empheq}
\end{subequations}
and 
$ \psi_i, \theta^{h}_i, \lambda^{h,\ell}_i\in C([0,T];\sR)$ satisfies 
\bb
\label{eq:zero_order_term}
\left\{
\begin{aligned}
&    \dot \psi_i+ \frac{1}{2}\tr\left(\sigma\sigma^\top \Psi_i \right)=0\q \fa t\in [0,T];
\quad  \psi_i(T)=0, 
\\
&\dot \theta^{h}_i + \frac{1}{2}\tr\left(\sigma\sigma^\top \Theta^{h}_i  \right)=0\q \fa t\in [0,T];
\quad \theta^{h}_i(T)=0,
\\
&   \dot \lambda^{h,\ell}_i + \frac{1}{2}\tr\left(\sigma\sigma^\top \Lambda^{h,\ell}_i  \right)=0\q \fa t\in [0,T];
\quad \lambda^{h,\ell}_i(T)=0.
\end{aligned}
\right.
\ee
In \eqref{eq:ode_lq} and \eqref{eq:zero_order_term},  
we use a dot to represent the derivative with respect to time,
and use $\tr(\cdot)$ to denote the trace of a matrix.
  Additionally, for the sake of notational simplicity,
we omit 
the dependence  of  $\Theta^{h}_i, \theta^{h}_i$  
on $K'_h$, 
and the dependence of $\Lambda^{h,\ell}_i, \lambda^{h,\ell}_i$   
on $K'_h$ and $K''_\ell$. 
%
%
%
One can easily verify that \eqref{eq:quadratic_value}, \eqref{eq:quadratic_1st} and 
\eqref{eq:quadratic_2nd}
 are the unique solutions  to the PDEs \eqref{eq:pde_value}, 
\eqref{eq:pde_value_1st} and \eqref{eq:pde_value_2nd}, respectively.

The following theorem proves the equivalence of 
Theorems 
  \ref{thm:differential_MPG_proba} and 
   \ref{thm:differential_MPG_pde}  for LQ games. 
  {The proof follows from a direct application of It\^{o}'s  formula,
  whose detail is given in Appendix \ref{sec:proof_lq_equivlance}}.

\begin{Theorem}
\label{thm:lq_characterisation}
For all  $i,h,\ell \in I_N$, 
$\phi\in \pi^{(N)}$, $\phi'_h\in \pi_h$, $\phi''_\ell\in \pi_\ell$
and $(t,x)\in [0,T]\times\sR^{n_x}$,
\bb
\label{eq:equivalence}
 \frac{\delta V^{t,x}_i}{\delta \phi_h} (\phi;\phi'_h)
={\frac{\delta \ol{V}^\phi_i}{\delta \phi_h}}(t,x; \phi'_h),
\quad 
 \frac{\delta^2 V^{t,x}_i}{\delta \phi_h \delta \phi_\ell} (\phi;\phi'_h, \phi''_\ell)
={\frac{\delta^2 \ol{V}^\phi_i}{\delta \phi_h\delta \phi_\ell}}(t,x; \phi'_h, \phi''_\ell).
\ee
That is, Condition \eqref{eq:prob_symmetry} in Theorem
  \ref{thm:differential_MPG_proba} is equivalent to  
  Condition \eqref{eq:pde_symmetry} in Theorem 
  \ref{thm:differential_MPG_pde}. 

Consequently,  if 
   there exists $(t,x)\in [0,T]\times \sR^{n_x}$ such that 
   for all  $i,j \in I_N$ with $i\not =j$, 
$\phi\in \pi^{(N)}$, $\phi'_i\in \pi_i$ and $\phi''_j\in \pi_j$,
\bb
\label{eq:lq_symmetry}
 \frac{1}{2}x^\top \Lambda^{i,j}_i(t)x+ \lambda^{i, j}_i(t) =
 \frac{1}{2}x^\top \Lambda^{j,i}_j(t)x +\lambda^{j, i}_j(t),
\ee
then 
     $\cG_{\rm LQ}$ with initial condition $(t,x)$ is a CLPG.
     Moreover,  $\cG_{\rm LQ}$ is an MPG if and only if  for all  $i,j \in I_N$ with $i\not =j$, 
$\phi\in \pi^{(N)}$, $\phi'_i\in \pi_i$ and $\phi''_j\in \pi_j$,
\bb
\label{eq:lq_symmetry_mpg}
  \Lambda^{i,j}_i(t)   =
  \Lambda^{j,i}_j(t), 
  \q \fa t\in [0,T].
\ee
\end{Theorem}

{ 
Theorem \ref{thm:lq_characterisation}  characterizes LQ  MPGs using   a system of  ODEs  
 parametrized by   all players'   policies.  
In particular, an LQ MPG 
with the policy class   \eqref{eq:policy_lq}
can be directly characterized by the first-order derivatives $\Theta^i_j$ of value functions.

\begin{Corollary}
\label{cor:lq_mpg}
 $\cG_{\rm LQ}$ is an MPG if and only if    for all $i,j \in I_N$, 
  $ (R_i)_{ij }=(R_j)_{ ij} $,
  and 
  for all $\phi\in \pi^{(N)}$ and $\phi'_i\in \pi_i$,
  $
 (  \Theta^i_i-\Theta^i_j )  B_j =0$,
 where 
 $ \Theta^i_i$ and $\Theta^i_j $ are defined by \eqref{eq:ode_theta}
 and $B_j$ is  given in \eqref{eq:BR_block}.
\end{Corollary}

By Corollary \ref{cor:lq_mpg}, one can establish, under suitable controllability assumptions of the system, the equivalence of
the  LQ MPG game and the team  Markov game (introduced in Example \ref{prop:team_MG}).

\begin{Proposition}
\label{prop:lq_nondegenreate}
Suppose  $B= (B_1, \ldots,   B_N)$ satisfies for all $i\in I_N$ and $t\in [0,T]$, 
$B_i(t) $ has full row rank.
Then $\cG_{\rm LQ}$ is an MPG if and only if 
for all $i,j\in I_N$, $Q_i=Q_j$, $R_i=R_j$ and $G_i=G_j$. 
\end{Proposition}

\begin{Remark}
\label{rmk:lq_rank_condition}
  Proposition \ref{prop:lq_nondegenreate}  holds for  multi-dimensional  LQ   games with a general state coefficient $A$,
a  general noise coefficient $\sigma$, and   symmetric cost coefficients $(Q_i, R_i, G_i)_{i \in I_N}$. The proof relies on  
the full rank condition of  $(B_i)_{i \in I_N}$, which by the Hautus lemma 
ensures each player's controllability of the system.

This  controllability allows players to drive the system to any desired state regardless of the policies chosen by other players. It implies  a strong coupling between the players' actions  and hence requires 
the homogeneity of cost functions for a potential game. It suggests that  the strength of interactions among players through dynamics and policies is crucial in determining whether a dynamic game is a potential game. These are intriguing insights about the structure of dynamic potential games. 
Quantifying the intricate interplay between the potential structure and player interactions for general dynamic games will be an intriguing research topic.   
\end{Remark}

}

 \section{Proofs of Main Results}
\label{sec:main_proof}
 
 \subsection{Proofs of Theorems   \ref{thm:separation_value},
 \ref{thm:distributed_game} 
  and \ref{thm:distributed_game_differentiable}}
 \label{sec:proof_separation}
 
 \begin{proof}[Proof of Theorem  \ref{thm:separation_value}]
We only prove the characterization of CLPGs with a fixed initial state $s_0\in \cG$. The characterization of MPGs holds by repeating the argument for all initial states. 

 It is clear that if 
 for all $i\in I_N$, 
 the value function $V^{s_0}_i$ admits the decomposition \eqref{eq:separation_value},
 then   $\cG$ with initial state $s_0$ is a CLPG 
with   potential $\Phi^{s_0}$. 
Hence it remains to prove that \eqref{eq:separation_value} is a necessary condition for 
 $\cG$  being a CLPG with potential $\Phi^{s_0}$.
 Let $i\in I_N$, $\phi_{-i}\in  \pi^{(N)}_{-i}$ 
 and $\phi_i, \phi_i',\phi_i''\in \pi_i$ be arbitrary policies. 
  Then by Definition \ref{def:MPG},  
   \begin{align*}
 V^{s_0}_i((\phi_i, \phi_{-i})) &=   \Phi^{s_0}((\phi_i, \phi_{-i})) -\Phi^{s_0}((\phi_i',\phi_{-i}))+V^{s_0}_i((\phi_i',\phi_{-i})), 
\\
  V^{s_0}_i((\phi_i, \phi_{-i})) &=   \Phi^{s_0}((\phi_i, \phi_{-i})) -\Phi^{s_0}((\phi_i'',\phi_{-i}))+V^{s_0}_i((\phi_i'',\phi_{-i})).
\end{align*}
   This shows that  for any 
   pair of policies $(\phi'_i,\phi''_i)\in \pi_i$,
   $$-\Phi^{s_0}((\phi_i',\phi_{-i}))+V^{s_0}_i((\phi_i',\phi_{-i}))= -\Phi^{s_0}((\phi_i'',\phi_{-i}))+V^{s_0}_i((\phi_i'',\phi_{-i})).
   $$
   Consequently,  
$U^{s_0}_i(\phi_{-i})\coloneqq  -\Phi^{s_0}((\phi_i',\phi_{-i}))+V^{s_0}_i((\phi_i',\phi_{-i}))$ is well-defined,
as it   only depends  on   $\phi_{-i}$ but  is   independent of $\phi'_i\in \pi_i$. 
  \end{proof}

  \begin{proof}[Proof of Theorem  \ref{thm:distributed_game}]
 Observe that 
 for all $i\in I_N$ and $(t,x)\in [0,T]\times\sR^{n_x}$, 
 by   \eqref{eq:cost_dist_i} and \eqref{eq:distributed_cost_decomp},
$
 V^{t,x}_i(\phi) =  \Phi^{t,x}(\phi)+U^{t,x}_i(\phi_{-i}),
$
where 
 \begin{align*}
U^{t,x}_i(\phi_{-i})
&=  \sE^{\mathbb{P}} \left[\int_t^T U_{f_i}(s,X^{t,x,\phi,-i}_s,\alpha^{t,x,\phi,-i}_s)\d s + U_{g_i}(X^{t,x,\phi,-i}_T)\right],
 \end{align*}
 with $X^{t,x,\phi,-i}= (X^{t,x_j,\phi_j})_{j\not =i}$
 and $\alpha^{t,x,\phi,-i}= (\phi_j(\cdot,X^{t,x_j,\phi_j}))_{j\not =i}$ being the state and control processes 
 of other players. 
By \eqref{eq:state_policy_distributed} and the definition of distributed policies,
 the states $(X^{t,x,\phi,\ell})_{\ell\in I_N}$
 are decoupled  and hence the term $U^{t,x}_i(\phi_{-i})$ depends only on $\phi_{-i}$ and is independent of $\phi_i$. 
 This along with  Theorem \ref{thm:separation_value}
 implies  that 
 $\Phi$  is a potential function for   $\cG_{\rm dist}$.
 \end{proof}
 
{ 
 \begin{proof}[Proof of Theorem  \ref{thm:distributed_game_differentiable}]
 Recall that   by \cite[Theorem 1.3.1]{facchinei2003finite}, 
 for a given continuously differentiable function $H:\sR^n\to \sR^n$,
 there exists a  function $\Phi:\sR^n\to \sR$ such that $H(z) =\nabla  \Phi (z) $ for all $z\in \sR^n$ if and only if the Jacobian matrix of $H$   is symmetric for all $z\in \sR^n$. In this case, $\Phi:\sR^n\to \sR$ can be chosen as  
 $$
 \Phi(z)\coloneqq \int_0^1 H(z^0+t(z-z^0))^\top (z-z^0)\d t,
 $$  
 for any given $z^0\in \sR^n$. The desired conclusion follows by applying this theorem with different choices of $H$. 
 
First,   
recall that  $n_x+n_a= \sum_{i\in I_N} (n_i+k_i)$ 
and 
for each $t\in [0,T]$,
define   $H_{f,t}: \sR^{n_x+n_a}  \to \sR^{n_x+n_a} $ such that  
$
(H_{f,t}(x,a) )_i\coloneqq (\partial_{(x_i,a_i)} f_i)(t, x,a)
$ 
for all $i\in I_N$.
The symmetry condition \eqref{eq:symmetric_hessian_distributed}
implies that the Jacobian matrix of 
$H_{f,t}$ is symmetric for all $(x,a)\in \sR^{n_x+n_a} $. 
 Hence applying \cite[Theorem 1.3.1]{facchinei2003finite} with $H=H_{f,t}$ implies that 
 $H_{f,t} = \nabla F(t,\cdot)$ with $\nabla F$ defined in \eqref{eq:F_G_distributed}. 
In particular, for all $i\in I_N$, $(\partial_{(x_i,a_i)} f_i)(t, x,a)=(\partial_{(x_i,a_i)} F)(t,x,a) $ 
for all $(t,x,a)\in [0,T]\times\sR^{n_x}\times \sR^{n_a}$. 
Thus  the function 
$U_{f_i}\coloneqq f_i- F$ satisfies 
$\partial_{(x_i,a_i)} U_{f_i}\equiv 0$, which along the the mean-value theorem implies that 
  $U_{f_i} $ is independent of $(x_i,a_i)$.
 Hence the function $F$ in \eqref{eq:F_G_distributed} satisfies the condition  \eqref{eq:distributed_cost_decomp}.
  The claim that $G$ in \eqref{eq:F_G_distributed} satisfies the condition  \eqref{eq:distributed_cost_decomp}
  follows   from similar arguments.
  \end{proof}
}

  \subsection{Proof   of Theorem \ref{thm:symmetry_value_sufficient}}
\label{sec:proof_abstract_MPG}

\begin{proof}[Proof of Theorem \ref{thm:symmetry_value_sufficient}]

We only prove the characterization of CLPGs, as the characterization of MPGs follows by repeating the argument for all initial states. 
Throughout this proof,  we denote by $s$ the fixed initial state $s_0$
for simplicity.

For each 
$i\in I_N$ and $\phi'_i\in \pi_i$, 
using Condition \ref{item:joint_continuity}
and applying   Lemma \ref{lemma:multi-dimension_derivative}
to $  \frac{\delta V_i^{s}}{\delta\phi_i}(\cdot; {\phi}'_i)$ 
 yield that  
   for all 
$\phi ,z \in \pi^{(N)}$,   
$[0,1]\ni r\mapsto  \frac{\delta V_i^{s}}{\delta\phi_i}(z+r( \phi -z); {\phi}'_i)\in \sR$ 
is differentiable  
and  
\bb
 \label{eq:chainrule_p-Vi}
 \frac{\d}{\d r}   \frac{\delta V_i^{s}}{\delta\phi_i}(z+r( \phi -z);  {\phi}'_i) 
 =\sum_{j=1}^N \frac{\delta^2 V_i^{s}}{\delta\phi_i\delta\phi_j}
 (z+r( \phi-z);  {\phi}'_i, \phi_j-z_j).
 \ee
Moreover, as   the second-order derivatives are bilinear  with respect to the last two arguments,    
one can assume without loss of generality   that 
Conditions \ref{item:integrable_p-Vi} and 
\ref{item:symmetry-Vi}  
and \eqref{eq:chainrule_p-Vi}
 hold    for all $  {\phi}'_i\in \spn( \pi_i)$
and $  {\phi}''_j\in \spn( \pi_j)$. 

Note that  $\Phi^s:    \pi^{(N)} \to \sR$ in   \eqref{eq:potential_variation}
  is well-defined, as 
  $r\mapsto    \frac{\delta V_j^{s}}{\delta\phi_j} (z+r( \phi-z) ; \phi_j-z_j)$ is continuous on $[0,1]$ for all $j\in I_N$. 
  We first prove    that for all   $\phi=(\phi_i)_{i\in I_N}\in \pi^{(N)}$, $i\in I_N$ and $\phi'_i\in \pi_i$, 
$$
\frac{\d }{\d \eps }\Phi^{s}(\phi^\eps)\bigg\vert_{\eps=0}
 =\frac{\delta V^{s}}{\delta \phi_i} (\phi ; \phi'_i-\phi_i)
 \q \textnormal{with $\phi^\eps \coloneqq (\phi_i+\eps(\phi'_i-\phi_i),\phi_{-i})$.}
$$
Observe that by \eqref{eq:potential_variation}, for all      $\eps\in (0,1]$,
\begin{align*}
  \Phi^{s} (\phi^{\eps} ) 
-  \Phi^{s}(\phi) 
&= 
\int_0^1\sum_{j=1}^N  \frac{\delta V_j^{s}}{\delta\phi_j} (z+r( \phi^{\eps}-z); \phi_j+\eps \delta_{ji}(\phi'_i-\phi_i)-z_j) \d r
\\
&\quad -
 \int_0^1\sum_{j=1}^N  \frac{\delta V_j^{s}}{\delta\phi_j} (z+r(\phi-z);\phi_j-z_j)  \d r,
 \end{align*}
 where $\delta_{ji}=0$  if $j\not =i$ and 
 $\delta_{ii}=1$.
Then by $z+r( \phi^{\eps}-z)\in \pi^{(N)}$,
for all   $ \eps\in (0,1]$, 
\begin{align}
\label{eq:dphi_i}
\begin{split}
 \frac{ \Phi^{s} (\phi^{ \eps} ) 
-  \Phi^{s}(\phi)  }{\eps} 
&
= 
\frac{1}{\eps} \int_0^1\sum_{j=1}^N  \left(  \frac{\delta V_j^{s}}{\delta\phi_j}(z+r( \phi^{ \eps}-z);\phi_j-z_j)
-  \frac{\delta V_j^{s}}{\delta\phi_j}(z+r( \phi-z);\phi_j-z_j)
\right)  \d r \\
 &\q  +    \frac{1}{\eps} \int_0^1\sum_{j=1}^N \eps \delta_{ji} \frac{\delta V_j^{s}}{\delta\phi_j}(z+r( \phi^{\eps}-z) ; \phi'_i-\phi_i)  \d r 
\\
& =    \int_0^1  \sum_{j=1}^N \frac{1}{ \eps}
\left(  \frac{\delta V_j^{s}}{\delta\phi_j} (z+r( \phi^{ \eps}-z);\phi_j-z_j) 
- \frac{\delta V_j^{s}}{\delta\phi_j} (z+r( \phi-z);\phi_j-z_j)\right)    \d r  
 \\
 &\q  +    \int_0^1  \frac{\delta V_i^{s}}{\delta\phi_i}(z+r( \phi^{\eps}-z) ; \phi'_i-\phi_i)   \d r. 
 \end{split}
\end{align}

To  send $\eps\to 0$ in \eqref{eq:dphi_i}, note that 
for all $\eps\in [0,1]$, 
$(z+r( \phi^{\eps}-z))_{-i}=z_{-i}+r(\phi_{-i}-z_{-i})$ and 
\begin{align}
\label{eq:convex_combination_r_eps}
\begin{split}
(z+r( \phi^{\eps}-z))_i
&=z_i + r (\phi_i+\eps(\phi'_i-\phi_i)-z_i)
\\
&=z_i+r(\phi_i-z_i)+\eps \big((z_i+r(\phi'_i-z_i))-(z_i+r(\phi_i-z_i))\big)
\end{split}
\end{align}
with  $z_i+r(\phi_i-z_i),z_i+r(\phi'_i-z_i)\in \pi_i$. 
Thus for all $j\in I_N$,
as $\phi_j-z_j  \in \spn(\pi_j)$,
  the twice differentiability of  $V^{s}_j$ and Lemma  \ref{lemma:derivative_line}  imply that 
$\eps\mapsto   \frac{\delta V_j^{s}}{\delta\phi_j} (z+r( \phi^{\eps}-z); \phi_j-z_j ) $ 
 is differentiable  on   $  [0,1]$ and 
\begin{align*}
\frac{\d }{\d \eps}\frac{\delta V_j^{s}}{\delta\phi_j}  (z+r( \phi^{\eps}-z); \phi_j-z_j )  
&= \frac{\delta^2 V_j^{s}}{\delta\phi_j\delta\phi_i}
(z+r( \phi^{\eps }-z); \phi_j-z_j , r(\phi'_i-\phi_i))
\\
&= \frac{\delta^2 V_j^{s}}{\delta\phi_j\delta\phi_i}
(z+r( \phi^{\eps }-z); \phi_j-z_j  , \phi'_i-\phi_i)r,  
\end{align*}
where the last identity used 
the linearity of $ \frac{\delta^2 V_j^{s}}{\delta\phi_j\delta\phi_i}$ in its last component.
Hence,  by 
the mean value theorem,
for all    $\ \eps\in (0,1]$,
\begin{align}
\label{eq:MVT_eps}
\begin{split}
&\left|\frac{1}{ \eps}
\left(  \frac{\delta V_j^{s}}{\delta\phi_j} (z+r( \phi^{\eps}-z);\phi_j-z_j)- \frac{\delta V_j^{s}}{\delta\phi_j} (z+r( \phi-z);\phi_j-z_j)\right)
\right|
\\
&\q \le
\sup_{ r,\eps\in [0,1]}\left| \frac{\delta^2 V_j^{s}}{\delta\phi_j\delta\phi_i}(z+r( \phi^{ \eps}-z);\phi_j-z_j, \phi'_i-\phi_i)r \right|
\\
& \q = 
\sup_{ r, \eps\in [0,1]}\left| \frac{\delta^2 V_i^{s}}{\delta\phi_i\delta\phi_j}(z+r( \phi^{ \eps}-z);\phi'_i-\phi_i,\phi_j-z_j)r \right|
<\infty,
\end{split}
\end{align}
where the last inequality follows from  Conditions   \ref{item:integrable_p-Vi} and \ref{item:symmetry-Vi}. 
Similarly,  as $\phi'_i-\phi_i\in \spn(\pi_i)$,
by the twice differentiability of  $V^{s}_i$ (see \eqref{eq:convex_combination_r_eps}),
for all $r\in (0,1)$,
\begin{align*}  
 \lim_{\eps\downarrow 0}  \frac{\delta V_i^{s}}{\delta\phi_i}(z+r( \phi^{\eps}-z)  ; \phi'_i-\phi_i)
  =   \frac{\delta V_i^{s}}{\delta\phi_i}(z+r( \phi-z) ; \phi'_i-\phi_i),
 \end{align*}
 and for all $r, \eps\in [0,1]$,  by  the mean value theorem,
 \begin{align*}
&  \left|  \frac{\delta V_i^{s}}{\delta\phi_i}(z+r( \phi^{ \eps}-z)  ; \phi'_i-\phi_i)\right|
  \\
  &\q \le \left|  \frac{\delta V_i^{s}}{\delta\phi_i}(z+r( \phi-z)  ; \phi'_i-\phi_i)\right|
  +\left|  \frac{\delta^2 V_i^{s}}{\delta\phi_i\delta\phi_i}(z+r( \phi^\eps-z)  ; \phi'_i-\phi_i,\phi'_i-\phi_i)r\right|,
 \end{align*}
 which is uniformly bounded with respect to $(r,\eps)\in [0,1]^2$  
 due to  \eqref{eq:chainrule_p-Vi} and   Condition     \ref{item:integrable_p-Vi}.  
  Hence, letting     $\eps\to 0$ in  \eqref{eq:dphi_i} and 
 using  Lebesgue's dominated convergence theorem give  
  \begin{align*}
\begin{split}
 \frac{\d }{\d \eps}  \Phi^{s} (\phi^{\eps  } ) \bigg\vert_{\eps=0} 
&  =    \int_0^1  \sum_{j=1}^N  \frac{\delta^2 V_j^{s}}{\delta\phi_j\delta\phi_i}(z+r( \phi-z);\phi_j-z_j, \phi'_i-\phi_i) r  \d r  
+ \int_0^1      \frac{\delta V_i^{s}}{\delta\phi_i}(z+r( \phi
-z) ; \phi'_i-\phi_i)   \d r
 \\
 &  =    \int_0^1  \sum_{j=1}^N   \frac{\delta^2 V_i^{s}}{\delta\phi_i\delta\phi_j}
 (z+r( \phi-z); \phi'_i-\phi_i, \phi_j-z_j) r  \d r  
+ \int_0^1    \frac{\delta V_i^{s}}{\delta\phi_i}(z+r( \phi-z) ; \phi'_i-\phi_i)   \d r
\\
 &  =    \int_0^1   r\frac{\d}{\d r} \left( \frac{\delta V_i^{s}}{\delta\phi_i}(z+r( \phi-z); \phi'_i-\phi_i) \right)  \d r  
+ \int_0^1    \frac{\delta V_i^{s}}{\delta\phi_i}(z+r( \phi-z) ; \phi'_i-\phi_i)   \d r,
 \end{split}
\end{align*}
where 
the second to last identity used  
Condition   \ref{item:symmetry-Vi}   
  and the last identity used  \eqref{eq:chainrule_p-Vi}.  
The integration by parts  formula then yields  
  \begin{align}
  \label{eq:dPhi=dV}
\begin{split}
&  \frac{\d }{\d \eps}  \Phi^{s} (\phi^{\eps  } ) \bigg\vert_{\eps=0} 
 =\frac{\delta V_i^{s}}{\delta\phi_i}(\phi ; \phi'_{i}-\phi_i).
 \end{split}
\end{align}

Now let 
   $i\in I_N$,  $ \phi'_i\in \pi_i$
  and $\phi\in \pi^{(N)}$. 
 For each $\eps\in [0,1]$,  
  let $\phi^\eps = (\phi_i+\eps(\phi'_i-\phi_i),\phi_{-i})\in \pi^{(N)}$.
  By the  differentiability of    $V_i^{s}$ and  Lemma  \ref{lemma:derivative_line},
  $\frac{\d }{\d \eps } V_i^{s}(\phi^\eps)=\frac{\delta V_i^{s}}{\delta\phi_i}(\phi^\eps; \phi'_i-\phi_i)  $ for all  $ \eps\in  [0,1]$,
  and  
  $\eps\mapsto  \frac{\delta V_i^{s}}{\delta\phi_i}(\phi^\eps; \phi'_i-\phi_i)$ 
  is differentiable on $[0,1]$.
   This implies that  $\eps\mapsto V_i^{s}(\phi^\eps) $ is continuously differentiable   on $[0,1]$. 
Then 
by Lemma  \ref{lemma:derivative_line} and \eqref{eq:dPhi=dV},
$[0,1]\ni \eps\mapsto \Phi^{s}(\phi^\eps)\in \sR$ is also 
continuously differentiable 
with $\frac{\d}{\d \eps}\Phi^{s}(\phi^\eps) =\frac{\d }{\d \eps } V_i^{s}(\phi^\eps)$ for all $\eps\in [0,1]$.
  Hence by the fundamental theorem of calculus, 
 \begin{align*}
  V_i^{s}((\phi'_i,\phi_{-i}))- V_i^{s}((\phi_i,\phi_{-i}))
 &  
 =\int_0^1 \frac{\delta V_i^{s}}{\delta\phi_i}(\phi^\eps; \phi'_i-\phi_i)   \d \eps 
 =\int_0^1 \frac{\delta \Phi^{s}}{\delta\phi_i}(\phi^\eps; \phi'_i-\phi_i)   \d \eps
 \\
&  =\Phi^{s}((\phi'_i,\phi_{-i}))- \Phi^{s}((\phi_i,\phi_{-i})).
 \end{align*}
  This proves that $\Phi^s$ defined  in \eqref{eq:potential_variation} 
  is a potential function of the game $\cG$ with initial state $s$. 
\end{proof}

 \subsection{Proof of Theorem \ref{thm:differential_MPG_pde}} 
\label{sec:proof_sde_analytic}


 

\begin{proof}[Proof of Theorem \ref{thm:differential_MPG_pde} Item \ref{item:pde_first}]
 Fix  $i,k\in I_N$,  $\phi\in \pi^{(N)}$ and $\phi'_k\in \spn(\pi_k)$.
Let  $\eta=\beta\gamma$.
By (H.\ref{assum:Holder_regularity}),  the H\"{o}lder continuity of $\phi, \phi'_k$
and $V^\phi_i$, it is clear that all coefficients of \eqref{eq:pde_value_1st} are in $C^{ {\eta}/{2}, \eta}([0,T]\times \sR^{n_x};\sR)$. 
Thus by    \cite[Theorem 5.1, p.~320]{ladyzhenskaia1988linear},
\eqref{eq:pde_value_1st} admits a unique classical solution   
 in $C^{1+ {\eta}/{2},2+\eta}([0,T]\times \sR^{n_x}; \sR)$.
This shows that the map 
$\frac{\delta V_i}{\delta \phi_k}$ in  \eqref{eq:value_1st_pde}
is well-defined. 
Moreover,
as  \eqref{eq:pde_value_1st} is linear in terms of  $\phi'_k$, 
$ \spn(\pi_k)\ni\phi'_k\mapsto \frac{\delta V^\phi_i}{\delta \phi_k}(t,x; \phi'_k)\in \sR$
is linear.  

Now fix $\phi'_k\in \pi_k$. For    each $\eps\in [0,1]$, let 
$\phi^\eps=(\phi_k+\eps(\phi'_k-\phi_k),\phi_{-k})$. 
We claim that
\bb
\label{eq:derivative_1st_holder}
\lim_{\eps\downarrow 0 }\bigg\|\frac{V^{\phi^\eps}_i(\cdot) -V^{\phi}_i(\cdot)}{\eps}-\frac{\delta V^\phi_i}{\delta \phi_k}(\cdot; \phi'_k-\phi_k)\bigg\|_{1+\beta\gamma/2,2+\beta\gamma}=0.
\ee
Recall that     $V^{\phi^\eps}_i$, $\eps\in (0,1]$, satisfies 
the following PDE:
 for all $(t,x)\in [0,T]\times \sR^{n_x}$, 
\begin{align}
\label{eq:pde_value_eps}
\begin{split}
&\p_t V^{\phi^\eps}_i (t,x) + \cL^{\phi^\eps} V^{\phi^\eps}_i (t,x)+f_i(t,x,\phi^\eps(t,x))=0;   
\q V^{\phi^\eps}_i (T,x)= g_i(x).
\end{split}
\end{align}
Then by \eqref{eq:pde_value_1st}, 
 $U^\eps(\cdot)=\frac{V^{\phi^\eps}_i(\cdot) -V^{\phi}_i(\cdot)}{\eps}-\frac{\delta V^\phi_i}{\delta \phi_k}(\cdot; \phi'_k-\phi_k)$  satisfies  for all $(t,x)\in [0,T]\times \sR^{n_x}$, 
\begin{align}
\label{eq:pde_U_eps}
\begin{split}
&\p_t U^{\eps} (t,x) +\cL^{\phi^\eps} U^{\eps} (t,x)  + F^\eps(t,x )=0;
\quad U^{\eps} (T,x)= 0,
 \end{split}
\end{align}
where 
\begin{align}
\label{eq:F_eps}
\begin{split}
 F^\eps(t,x ) 
 &=
  \frac{(\cL^{\phi^\eps}-\cL^{\phi})V(t,x)+f_i(t,x,\phi^\eps(t,x))-f_i(t,x,\phi(t,x))}{\eps}
 \\
 &\quad - (\p_{a_k} H_i)\big(t,x,(\p_x V^\phi_i)(t,x),(\p_{xx} V^\phi_i)(t,x),\phi(t,x)\big)^\top (\phi'_k-\phi_k)(t,x)
 \\
 & \q  +(\cL^{\phi^\eps}-\cL^{\phi})\frac{\delta V^\phi_i}{\delta \phi_k}(t,x; \phi'_k-\phi_k).
 \end{split}
\end{align}
By Taylor's theorem,  
  for all $(t,x)\in [0,T]\times \sR^{n_x}$, 
  \begin{align*}
&  \frac{f_i(t,x,\phi^\eps(t,x))-f_i(t,x,\phi(t,x))}{\eps}-(\p_{a_k} f_i)\big(t,x, \phi(t,x)\big)^\top (\phi'_k-\phi_k)(t,x)
\\
&= 
\eps  \int_0^1  (\phi'_k-\phi_k)(t,x)^\top(\p_{a_ka_k} f_i)(t,x,\phi^{r\eps}(t,x))  (\phi'_k-\phi_k)(t,x) \d r,
\end{align*}
  which along with the H\"{o}lder continuity of $\p_{a_ka_k} f_i$, $\phi_k$ and $\phi'_k$
  shows that 
  there exists  $C\ge 0$ such that 
  $\|\frac{f_i(\cdot,\phi^\eps(\cdot))-f_i(\cdot,\phi(\cdot))}{\eps}-(\p_{a_k} f_i)\big(\cdot, \phi(\cdot)\big)^\top (\phi'_k-\phi_k)(\cdot)\|_{\eta/2,\eta}\le C\eps$ for all $\eps\in (0,1]$. 
  Similarly, by using $V, \frac{\delta V^\phi_i}{\delta \phi_k}(\cdot; \phi'_k-\phi_k)\in C^{1+ {\eta}/{2},2+\eta}([0,T]\times \sR^{n_x};\sR)$, 
one can prove the
$\|\cdot\|_{\eta/2,\eta}$-norm of the 
 remaining terms in  \eqref{eq:F_eps} converges to zero as $\eps\to 0$.
As 
$\phi^\eps $ is a convex combination of $\phi$ and $(\phi'_k,\phi_{-k})$ (see \eqref{eq:policy_pde}),
  $\sup_{\eps\in [0,1]}\|\phi^\eps\|_{\gamma/2,\gamma}<\infty$.
Hence applying 
 the Schauder  a-priori estimate  \cite[Theorem 5.1, p.~320]{ladyzhenskaia1988linear}
 to \eqref{eq:pde_U_eps}
 implies that 
there exists $C\ge 0$ such that
for all $\eps\in (0,1]$,
$\|U^\eps\|_{1+\eta/2,2+\eta}\le C\|F^\eps\|_{\eta/2,\eta}\le C\eps$.
This  proves \eqref{eq:derivative_1st_holder},
which implies that  
$\frac{\d}{\d \eps}  {  V^{\phi^\eps}_i} (t,x ) \big\vert_{\eps=0} =
  \frac{\delta V^{\phi}_i}{\delta \phi_k }(t,x; \phi'_k-\phi_k )$.
  \end{proof}

\begin{proof}[Proof of Theorem \ref{thm:differential_MPG_pde} Item \ref{item:pde_second}]
 Fix  $i,k,\ell \in I_N$,  $\phi\in \pi^{(N)}$,  $\phi'_k\in \spn(\pi_k)$
 and   $\phi''_\ell\in \spn(\pi_\ell)$.
 Let $\eta=\beta\gamma$. 
 Recall that    $V^{\phi}_i(\cdot)$, 
 $ \frac{\delta V^\phi_i}{\delta \phi_k}(\cdot; \phi'_k)$ and   $ \frac{\delta V^\phi_i}{\delta \phi_\ell}(\cdot; \phi''_\ell)$
 are 
 in   $C^{1+ {\eta}/{2},2+\eta}([0,T]\times \sR^{n_x}; \sR)$,
 as shown in       
   the proof of Theorem \ref{thm:differential_MPG_pde} Item \ref{item:pde_first}.
 Hence 
the well-posedness of  \eqref{eq:pde_value_2nd} in $C^{1+ {\eta}/{2},2+\eta}([0,T]\times \sR^{n_x}; \sR)$      
follows from  a similar argument as that for \eqref{eq:pde_value_1st},
which subsequently implies that  
  the map 
$\frac{\delta^2 V_i}{\delta \phi_k\delta \phi_\ell}$ in  \eqref{eq:value_2nd_pde}
is well-defined. 
Moreover, by 
 \eqref{eq:pde_value_1st},
 for all $k\in I_N$, 
$\phi'_k\mapsto \big(\p_x \frac{\delta V^\phi_i}{\delta \phi_k}\big)(t,x; \phi'_k)$
and 
$\phi'_k\mapsto  \big(\p_{xx} \frac{\delta V^\phi_i}{\delta \phi_k}\big)(t,x; \phi'_k)$
 are  linear. 
This shows that  \eqref{eq:pde_value_2nd} is bilinear in terms of $\phi'_k$ and $\phi''_\ell$,
and hence the bilinearity of 
$   (\phi'_k, \phi''_\ell)\mapsto   \frac{\delta^2 V^\phi_i}{\delta \phi_k\phi_\ell}(t,x; \phi'_k,\phi''_\ell) $.

Now 
fix  $\phi'_k\in \spn(\pi_k)$ and $\phi''_\ell\in \pi_\ell$.
For all   $\eps\in [0,1]$,  let 
$\phi^\eps=(\phi_\ell+\eps(\phi''_\ell-\phi_\ell),\phi_{-\ell})$. 
We aim to  prove   
\bb
\label{eq:derivative_2nd_sup}
\lim_{\eps\downarrow 0 }\bigg\|
\frac{ 1}{\eps}\bigg(\frac{\delta V^{\phi^\eps}_i}{\delta \phi_k}(\cdot; \phi'_k ) - \frac{\delta V^{\phi}_i}{\delta \phi_k}(\cdot; \phi'_k )\bigg)-  \frac{\delta^2 V^{\phi}_i}{\delta \phi_k\delta \phi_\ell}(\cdot; \phi'_k,\phi''_\ell-\phi_\ell )
\bigg\|_{0}=0.
\ee
By   \eqref{eq:pde_value_1st} and  \eqref{eq:pde_value_2nd}, 
 for all $\eps\in (0,1]$,
$$W^\eps(\cdot)=\frac{ 1}{\eps}\left(\frac{\delta V^{\phi^\eps}_i}{\delta \phi_k}(\cdot; \phi'_k ) - \frac{\delta V^{\phi}_i}{\delta \phi_k}(\cdot; \phi'_k )\right)-  \frac{\delta^2 V^{\phi}_i}{\delta \phi_k\delta \phi_\ell}(\cdot; \phi'_k,\phi''_\ell-\phi_\ell )$$  
satisfies for all  $(t,x)\in [0,T]\times \sR^{n_x}$,
\begin{equation}
    \label{eq:PDE_W_eps}
\p_t W^{\eps} (t,x) +\cL^{\phi^\eps} W^{\eps} (t,x)  + G^\eps(t,x )=0;
\ \ \ \  W^{\eps} (T,x)= 0,
\end{equation}
where $G^\eps(t,x )=  G^\eps_1(t,x )+G^\eps_2(t,x )+G^\eps_3(t,x )$
satisfies for all $\bar{x}= (t,x)\in [0,T]\times\sR^{n_x}$,  
 \begin{align}
G^\eps_{1}(\bar{x}) 
 &=
  \frac{1}{\eps}
(\cL^{\phi^\eps}-\cL^{\phi}) \frac{\delta V^{\phi}_i}{\delta \phi_k}(\bar{x}; \phi'_k ) 
\label{eq:G1_eps}
\\
 &\quad -
 (\p_{a_\ell} L)\left(\bar{x}, \left(\p_x \frac{\delta V^\phi_i}{\delta \phi_k}\right)(\bar{x}; \phi'_k) ,
  \left(\p_{xx} \frac{\delta V^\phi_i}{\delta \phi_k}\right)(\bar{x}; \phi'_k) ,\phi(\bar{x})\right) ^\top (\phi''_\ell-\phi_\ell)(\bar{x}),
\nb
\end{align}
\begin{align}
G^\eps_2(\bar{x} ) 
 &=
  \frac{1}{\eps}
  \bigg[
(\p_{a_k} H_i)\big(\bar{x},(\p_x V^{\phi^\eps}_i)(\bar{x}),(\p_{xx} V^{\phi^\eps}_i)(\bar{x}),{\phi^\eps}(\bar{x})\big)^\top \phi'_k(\bar{x}) 
\label{eq:G2_eps}
  \\
&\q 
-(\p_{a_k} H_i)\big(\bar{x},(\p_x V^\phi_i)(\bar{x}),(\p_{xx} V^\phi_i)(\bar{x}),\phi(\bar{x})\big)^\top \phi'_k(\bar{x})  
  \bigg]
   \nb
 \\
 &\quad -
 \bigg[
    (\p_{a_k} L)\left(\bar{x}, \left(\p_x \frac{\delta V^\phi_i}{\delta \phi_\ell}\right)(\bar{x}; \phi''_\ell-\phi_\ell) ,
\left(\p_{xx} \frac{\delta V^\phi_i}{\delta \phi_\ell}\right)(\bar{x}; \phi''_\ell-\phi_\ell),
\phi(\bar{x})\right) ^\top
 \phi'_k(\bar{x})
  \nb
 \\
 & \q
 +   (\phi''_\ell-\phi_\ell)(\bar{x})^\top (\p_{a_ka_\ell} H_i)\big(t,x,(\p_x V^{\phi}_i)(\bar{x}),(\p_{xx} V^{\phi}_i)(\bar{x}),\phi(\bar{x})\big)  \phi'_k(\bar{x})
 \bigg],
 \nb
 \\
 G^\eps_{3}(\bar{x} ) 
 &=
(\cL^{\phi^\eps}-\cL^{\phi})\frac{\delta^2 V^{\phi}_i}{\delta \phi_k\delta \phi_\ell}(\bar{x}; \phi'_k,\phi''_\ell-\phi_\ell ).
\label{eq:G_3_eps}
\end{align}
We claim that   $\lim_{\eps\downarrow 0}\|G^\eps\|_0=0$. The convergences of 
 $G^\eps_1$ and  $G^\eps_3$ follow directly from the H\"{o}lder continuity of $B$ and $\Sigma$.
For  $G^\eps_2$, using 
 $H_i(t,x,y,z,a)-H_i(t,x,y',z',a)=L(t,x,y,z,a)-L(t,x,y',z',a)$
 (see \eqref{eq:Hamiltonian_H} and \eqref{eq:L}), 
\begin{align*}
\|G^\eps_2\|_0
 &\le
\bigg\|   \frac{1}{\eps}
  \bigg[
(\p_{a_k} L)\big(\cdot, \p_x V^{\phi^\eps}_i, \p_{xx} V^{\phi^\eps}_i, {\phi^\eps} \big)^\top \phi'_k  
-(\p_{a_k} L)\big(\cdot, \p_x V^{\phi}_i, \p_{xx} V^{\phi}_i, {\phi^\eps} \big)^\top \phi'_k  
  \bigg]
  \\
 &\quad -
    (\p_{a_k} L)\left(\cdot,  \p_x \frac{\delta V^\phi_i}{\delta \phi_\ell}  ,
 \p_{xx} \frac{\delta V^\phi_i}{\delta \phi_\ell} ,
\phi^\eps \right) ^\top
 \phi'_k 
 \bigg\|_0
  \\
&\q 
+ \bigg\| \frac{1}{\eps}
  \left( 
 (\p_{a_k} H_i)\big(\cdot, \p_x V^{\phi}_i, \p_{xx} V^{\phi}_i, {\phi^\eps} \big) 
-(\p_{a_k} H_i)\big(\cdot ,(\p_x V^\phi_i),(\p_{xx} V^\phi_i),\phi\big) 
\right)^\top \phi'_k 
   \\
 & \q
  -   (\phi''_\ell-\phi_\ell)^\top (\p_{a_ka_\ell} H_i)\big(\cdot,(\p_x V^{\phi}_i),(\p_{xx} V^{\phi}_i),\phi\big)  \phi'_k
  \bigg\|_0
   \\
 &\quad 
+\left\| (\p_{a_k} L)\left(\cdot,  \p_x \frac{\delta V^\phi_i}{\delta \phi_\ell}  ,
 \p_{xx} \frac{\delta V^\phi_i}{\delta \phi_\ell} ,
\phi^\eps \right)  
-
    (\p_{a_k} L)\left(\cdot, \p_x \frac{\delta V^\phi_i}{\delta \phi_\ell}  ,
 \p_{xx} \frac{\delta V^\phi_i}{\delta \phi_\ell},
\phi\right) \right\|_0
\| \phi'_k\|_0
\\
&\coloneqq  G^\eps_{2,1} +G^\eps_{2,2} + G^\eps_{2,3},
\end{align*}
where we wrote 
$  \frac{\delta V^\phi_i}{\delta \phi_\ell}$ for $   \frac{\delta V^\phi_i}{\delta \phi_\ell}(\cdot; \phi''_\ell-\phi_\ell)$
to simplify the notation. 
By \eqref{eq:derivative_1st_holder},
$\lim_{\eps\downarrow 0 }\Big\|\frac{V^{\phi^\eps}_i(\cdot) -V^{\phi}_i(\cdot)}{\eps}-\frac{\delta V^\phi_i}{\delta \phi_\ell}(\cdot; \phi''_\ell-\phi_\ell)\Big\|_{1+\eta/2,2+\eta}=0$,
which along with  the linearity of  
 $(y,z)\mapsto L(t,x,y,z,a)$ and 
(H.\ref{assum:Holder_regularity})
shows   
$\lim_{\eps\downarrow 0 } \|G^\eps_{2,1}\|_0=0$.
For the term $G^\eps_{2,2}$, 
for all $\bar{x} \in [0,T]\times\sR^{n_x}$, 
   writing $\p_x V^{\phi}_i=(\p_x V^{\phi}_i)(\bar{x})$
and $\p_{xx} V^{\phi}_i=(\p_{xx} V^{\phi}_i)(\bar{x})$
and applying the fundamental theorem of calculus
yield
\begin{align*}
& \frac{1}{\eps}
  \left( 
 (\p_{a_k} H_i)\big(\bar{x},  \p_x V^{\phi}_i ,  \p_{xx} V^{\phi}_i, {\phi^\eps}(\bar{x}) \big) 
-(\p_{a_k} H_i)\big(\bar{x} , \p_x V^\phi_i, \p_{xx} V^\phi_i,\phi(\bar{x})\big) 
\right)  
   \\
 & \q
  -   (\p_{a_ka_\ell} H_i)\big(\bar{x}, \p_x V^{\phi}_i , \p_{xx} V^{\phi}_i,\phi(\bar{x})\big)^\top  (\phi''_\ell-\phi_\ell)(\bar{x}) 
  \\
  &=  \int_0^1 
 \left( (\p_{a_ka_\ell} H_i)\big(\bar{x},\p_x V^{\phi}_i,\p_{xx} V^{\phi}_i,\phi^{r\eps}(\bar{x})\big) -
  (\p_{a_ka_\ell} H_i)\big(\bar{x},\p_x V^{\phi}_i ,\p_{xx} V^{\phi}_i,\phi(\bar{x})\big)\right)^\top  (\phi''_\ell-\phi_\ell)(\bar{x}) \d r
  \\
&  \le C\eps^{\beta  } \|\phi''_\ell-\phi_\ell\|_0^2,
\end{align*}
where the last inequality  used 
  the $\beta $-H\"{o}lder continuity 
of $a\mapsto (\p_{a_ka_\ell} H_i)\big(\bar{x},\p_x V^{\phi}_i,\p_{xx} V^{\phi}_i,a\big)$.
This proves $\lim_{\eps\downarrow 0 } \|G^\eps_{2,2}\|_0=0$. 
The H\"{o}lder continuity of coefficients   gives   $\lim_{\eps\downarrow 0 } \|G^\eps_{2,3}\|_0=0$
and $\lim_{\eps\downarrow 0 } \|G^\eps_{3}\|_0=0$.
Applying the maximum principle    \cite[Theorem 2.5, p.~18]{ladyzhenskaia1988linear}
to \eqref{eq:PDE_W_eps}
yields that there exists $C\ge 0$ such that 
$\|W^\eps\|_0\le C\|G^\eps\|_0$ for all $\eps\in (0,1]$.
This  along with the convergence of $G^\eps$ implies 
$\lim_{\eps\downarrow 0 } \|W^\eps\|_0=0$
and proves \eqref{eq:derivative_2nd_sup}.
Consequently, for all $(t,x)\in [0,T]\times \sR^{n_x}$,
  $\frac{\d}{\d \eps} \frac{\delta V^{\phi^\eps}_i}{\delta \phi_k}(t,x; \phi'_k ) \big\vert_{\eps=0} =
  \frac{\delta^2 V^{\phi}_i}{\delta \phi_k\delta \phi_\ell}(t,x; \phi'_k,\phi''_\ell-\phi_\ell )$.
Interchanging the roles of  $k$ and $\ell $ in the above analysis  proves the desired statement.
\end{proof}

\begin{proof}[Proof of Theorem \ref{thm:differential_MPG_pde} Item \ref{item:pde_potential}]
By   Theorem \ref{thm:differential_MPG_pde} Items  \ref{item:pde_first} and \ref{item:pde_second},
it remains to verify    Theorem \ref{thm:symmetry_value_sufficient} 
 Conditions  \ref{item:integrable_p-Vi}
and  \ref{item:joint_continuity}
for fixed
 $i,j\in I_N$, 
 $ z, 
 \phi \in \pi^{(N)}$, 
  $\phi'_i,\tilde{\phi}'_i\in \pi_i$  
and $ {\phi}''_j \in \pi_j$.

For Condition  \ref{item:integrable_p-Vi}, 
for all $r, \eps\in [0,1]$,
let $\phi^{r,\eps}=
z+r( (\phi_i+\eps(\tilde{\phi}'_{i}-\phi_i),\phi_{-i})-z) $.
As $\sup_{r, \eps\in [0,1]}\|\phi^{r,\eps}\|_{\gamma/2,\gamma}<\infty$,
by applying  
 the Schauder  a-priori estimate  \cite[Theorem 5.1, p.~320]{ladyzhenskaia1988linear}
 to \eqref{eq:pde_value}, \eqref{eq:pde_value_1st}
 and \eqref{eq:pde_value_2nd},
 there exists $\eta\in (0,1]$ such that 
 $\sup_{r, \eps\in [0,1]}\Big\| \frac{\delta^2 V^{\phi^{r,\eps}}_i}{\delta \phi_i\delta \phi_j} (\cdot; \phi'_i,\phi''_j) \Big \|_{\gamma/2,\gamma}<\infty$.

 For Condition  \ref{item:joint_continuity}, 
 for all $\eps=(\eps_i)_{i=1}^N\in [0,1]^N$,
let  $\phi^\eps=  (z_i+\eps_i({\phi}_{i}-z_i))_{i\in I_N}$
and observe that 
$U^\eps= V^{\phi^{\eps}}-V^{\phi^{0}}$ satisfies  for all $(t,x)\in [0,T]\times \sR^{n_x}$, 
\bb
\left\{
 \begin{aligned}
&\p_t U^\eps (t,x) + \cL^{\phi^\eps} U^\eps (t,x)+( \cL^{\phi^\eps}- \cL^{z}) V^{\phi^{0}}(t,x) + f_i(t,x,\phi^\eps(t,x))
- f_i(t,x,{\phi^0}(t,x))=0,
\\
& U^\eps(T,x)= 0.
 \end{aligned} 
 \right.
 \ee
As  $\sup_{ \eps\in [0,1]^N}\|\phi^{\eps}\|_{\gamma/2,\gamma}<\infty$,
by (H.\ref{assum:Holder_regularity}),     the Schauder  a-priori estimate
and the fundamental theorem of calculus,
  there exists   $C\ge 0$ such that  for all $ \eps\in [0,1]^N$, 
  it holds with $\eta=\beta \gamma$ that
 \begin{align*}
  \|U^\eps\|_{1+\eta/2,2+ \eta}
  &\le C\|( \cL^{\phi^\eps}- \cL^{\phi^{0}}) V^{\phi^{0}}+ f_i(\cdot,\phi^\eps(\cdot))
- f_i(\cdot,{\phi^0}(\cdot))\|_{ \eta/2, \eta}
\\
&\le C\left\|\int_0^1
\sum_{j=1}^N 
(\p_{a_j} H_i)\big(\cdot, (\p_x V^{\phi^{0}}_i)(\cdot), (\p_{xx} V^{\phi^{0}}_i)(\cdot),\phi^{r\eps}(\cdot)\big)^\top \eps_j
(\phi_j-z_j)(\cdot)\d r \right\|_{ \eta/2, \eta},
\end{align*}
where  $\phi^{r\eps}=  (z_i+r\eps_i({\phi}_{i}-z_i))_{i\in I_N}$.
This   along with the H\"{o}lder continuity of $\p_{a} H_i$  implies $ \lim_{\eps\downarrow 0}\|V^{\phi^{\eps}}-V^{\phi^{0}}\|_{1+ \beta \gamma/2,2+ \beta \gamma}=0$. 
Applying similar arguments to  \eqref{eq:pde_value_1st}  yields   
$\lim_{\eps\downarrow 0} \frac{\delta V^{\phi^\eps}_i}{\delta \phi_i}(\cdot; \phi'_i)
= \frac{\delta V^{\phi^{0}}_i}{\delta \phi_i}(\cdot; \phi'_i)$
and 
$\lim_{\eps\downarrow 0}\frac{\delta V^{\phi^\eps}_i}{\delta \phi_j}(\cdot; \phi''_j) 
=  \frac{\delta V^{\phi^{0}}_i}{\delta \phi_j}(\cdot; \phi''_j)$
with respect to   the 
$  \|\cdot   \|_{1+ \beta \gamma/2,2+ \beta \gamma}$-norm. 
Finally,  by
 \eqref{eq:pde_value_2nd} and  
 the maximum principle, 
$\lim_{\eps\downarrow 0} \big\|\frac{\delta^2 V^{\phi^\eps}_i}{\delta \phi_i\delta \phi_j}(\cdot; \phi'_i,\phi''_j)- \frac{\delta^2 V^{\phi^0}_i}{\delta \phi_i\delta \phi_j}(\cdot; \phi'_i,\phi''_j)\big \|_0=0$.
This verifies Condition  \ref{item:joint_continuity}.
\end{proof}

\subsection{Proofs of  
Corollary \ref{cor:lq_mpg} 
and Proposition \ref{prop:lq_nondegenreate}}
\label{sec:proof_example}

  {

\begin{proof}[Proof of Corollary \ref{cor:lq_mpg}]
    For each  $\pi\in \pi^{(N)}$ with   feedback maps $K$,   define $F\coloneqq A+BK$. 
For each $i,j\in I_N$ with $i\not =j$,  $\phi'_i\in \pi_i$ and $\phi''_j\in \pi_j$,
define  $\Delta \Theta^{i} \coloneqq   \Theta^{i}_i-\Theta^{i}_j$
and   $\Delta \Theta^{j} \coloneqq   \Theta^{j}_i-\Theta^{j}_j$.
For any matrix $M\in \sR^{n_x\times n_x}$, 
define $M^S \coloneqq M+M^\top$. 

Fix $\pi\in \pi^{(N)}$.
Observe from  \eqref{eq:ode_lambda} that for all $\phi'_i\in \pi_i$ and $\phi''_j\in \pi_j$,
$M\coloneqq \Lambda^{i,j}_i-
  \Lambda^{j,i}_j$ satisfies the ODE: 
\begin{align*}
  \dot M &+ F^\top M+ M F 
  + \left(\Delta \Theta^{i}  B_j K''_j   \right)^S
+\left(\Delta \Theta^{j}  B_i K'_i \right)^S 
\nb \\
& 
 +\left((K''_j)^\top  (R_i-R_j)_{j i}  K'_i\right)^S
=0, \q t\in [0,T]; \q   M(T)=0, 
\end{align*}
where we used $\big((R_i)_{\ell h}\big)^\top = (R_i)_{h\ell  } $ for all $i, h,\ell\in I_N$, as  $R_i$ takes values in $\sS^{n_a}$. By Condition \eqref{eq:lq_symmetry_mpg}
and the continuity of coefficients, 
$\cG_{\rm LQ}$ is an MPG
 if and only if for all $i,j\in I_N$ with $i\not =j$,
$K_i \in C([0,T];  \sR^{k_i\times n_x})$ and 
$K_j \in C([0,T];  \sR^{k_j\times n_x})$, 
 \begin{equation}
 \label{eq:Lambda_test}
 \left(\Delta \Theta^{i}  B_j K''_j  \right)^S
+\left(\Delta \Theta^{j}  B_i K'_i
   \right)^S 
 +\left((K''_j)^\top  (R_i-R_j)_{j i}  K'_i\right)^S =0,
\end{equation}
Note that  by     \eqref{eq:ode_lambda},
   $\Delta \Theta^{i} (T)= \Delta \Theta^{j}(T)=0$, and 
 $\Delta \Theta^{i}(t) $ (resp.~$\Delta \Theta^{j} (t)$) 
 depends linearly   on the integral  of  $K_i'$ (resp.~$K''_j$)   for all $t\in [0,T]$. 
Hence for each $t\in [0,T]$,
choosing a sequence of   $K_i'$ (resp.~$K''_j$)
that converges to $0$ except at $t$  shows that \eqref{eq:Lambda_test} is equivalent to 
for all 
$K_i \in C([0,T];  \sR^{k_i\times n_x})$, 
$K_j \in C([0,T];  \sR^{k_j\times n_x})$, 
$M_i\in \sR^{k_i\times n_x}$, and $M_j\in \sR^{k_i\times n_x}$,
$ \left(\Delta \Theta^{i}  B_j M_j  \right)^S=0$, 
$\left(\Delta \Theta^{j}  B_i M_i
   \right)^S =0$,
and 
$ \left(M_j^\top  (R_i-R_j)_{j i}  M_i\right)^S =0$.
Since
 $\tr(M^S)=2\tr(M )$ for any   $M\in \sR^{n_x\times n_x}$,
and  $(M_1,M_2)\mapsto \tr(M^\top_1 M_2)$
is  the Frobenius inner product,
we see $\cG_{\rm LQ}$ is an MPG
 if and only if  
  for all $i,j\in I_N$ with $i\not =j$,
 $ (R_i)_{ji}=(R_j)_{j i}$,
 and 
 for all 
$K_i \in C([0,T];  \sR^{k_i\times n_x})$ and 
$K_j \in C([0,T];  \sR^{k_j\times n_x})$,  
 $
 \Delta \Theta^{i}  B_j =0$ and 
 $  \Delta \Theta^{j}  B_i =0 
 $. This finishes the proof.
\end{proof}

 \begin{proof}[Proof of Proposition \ref{prop:lq_nondegenreate}]

 Since for all $t\in [0,T]$, $B_i(t)$ has full row rank and hence has a right inverse, by Corollary \ref{cor:lq_mpg},
$\cG_{\rm LQ}$ is an MPG if and only if 
for all $i,j\in I_N$ with $i\not =j$, all 
$K\in C([0,T];  \sR^{n_a\times n_x})$,
and all $  
K_i \in C([0,T];  \sR^{k_i\times n_x})$,
\begin{equation}
\label{eq:lq_full_rank_theta}
 (R_i)_{ij}=(R_j)_{ij} 
 \quad \textnormal{and} \quad
   \Theta^{i}_i-
\Theta^{i}_j=0.
\end{equation}
 By \eqref{eq:ode_theta},
  $M\coloneqq   \Theta^{i}_i-
\Theta^{i}_j$ satisfies the ODE: for all $t\in [0,T]$,  
\begin{align*}
    \dot{M}  &+ F^\top   M+M F 
 +\left( (\Psi_i-\Psi_j) B_i K'_i\right)^S 
    +  \left(K^\top (R_i-R_j)_{ i}K'_i\right)^S 
=0; \q  M(T)=0,
\end{align*} 
where $A^S \coloneqq A+A^\top$
for any   $A\in \sR^{n_x\times n_x}$.
Since $\Psi_i-\Psi_j$ is independent of $K'_i$,
 it holds that \eqref{eq:lq_full_rank_theta} 
is equivalent to 
for all $i,j\in I_N$ with $i\not =j$, 
$(R_i)_{ij}=(R_j)_{ij}$ and for all 
$K\in C([0,T];  \sR^{n_a\times n_x})$,
$(\Psi_i-\Psi_j)B_i+ K^\top (R_i-R_j)_{ i}=0$.
Observe from   \eqref{eq:ode_psi} that
modifying $K$ at any time point will not change the value of $\Psi_i-\Psi_j$.
This, along with a limiting argument 
and the full row rank of $B_i$,
shows  that $\cG_{\rm LQ}$ is an MPG    if and only if $R_i =R_j$ and 
$\Psi_i=\Psi_j $ for all  $K\in C([0,T];  \sR^{n_a\times n_x})$,
which is equivalent to that for all $i,j\in I_N$, $R_i=R_j$, $Q_i=Q_j$ and $G_i=G_j$ by \eqref{eq:ode_psi}.
This completes the proof.
\end{proof}
}

\bibliographystyle{abbrv}
\bibliography{potential_game.bib}

\newpage
\appendix 

\section{Proofs of Theorems \ref{thm:differential_MPG_proba}
and \ref{thm:lq_characterisation}}
 \label{sec:techinical_lemma}
 
 \subsection{Proof of Theorem \ref{thm:differential_MPG_proba}}
\label{sec:proof_sde_probabilistic}

We start by presenting  Lemmas    
 \ref{lemma:1st_derivative_state},
 \ref{lemma:2nd_derivative_state}
 and \ref{lemma:control_derivative},
 which characterize 
 the derivatives of  $X^{t,x,\phi}$ 
 and $\alpha^{t,x,\phi}$  
 with respect to policies.
The proofs  
 are given in 
   Appendix 
  \ref{sec:differentiability_lemma}.

The following notation will be used in the subsequent analysis:
for each $t\in [0,T]$, $n\in \sR$
 and  $(X^n)_{n=0 }^\infty\in \cS^\infty([t,T];\sR^n)$,  
 we write  $\cS^\infty\text-\lim_{n\to\infty}X^n=X^0$
if   for all $q\ge 1$,   $\sE[\sup_{s\in [t,T]} |X^n_s-X^0_s|^q]<\infty$.
Let    $ X^\eps\in \cS^\infty([t,T];\sR^n)$  for all $ \eps\in [0,1)$, 
we write $Y=\cS^\infty \text- \p_\eps X^\eps\vert_{\eps=0}$ if  $\cS^\infty \text-\lim_{r\downarrow 0}\frac{1}{r}(X^{r}-X^{0})=Y$.  
 Higher order  derivatives of a process   can be defined similarly. 
 The 
continuity and differentiability for processes 
in $\cH^\infty([t,T];\sR^n)$ are defined similarly.

 \begin{Lemma}
 \label{lemma:1st_derivative_state}
  Assume (H.\ref{assum:regularity_N}). 
Let  
   $(t,x)\in [0,T]\times \sR^{n_x}$, $\phi\in \pi^{(N)}$ and $i\in I_N$.
   For all 
  $\phi'_i\in \spn(\pi_i)$,
  \eqref{eq:X_sensitivity_first} admits a unique solution $\frac{\delta X^{t,x}}{\delta \phi_i}(\phi;\phi'_i)\in \cS^\infty([t,T];\sR^{n_x})$
  satisfying  the following properties:
\begin{enumerate}[(1)]
\item
\label{item:linear_dX_dphi}
 The map $\frac{\delta X^{t,x}}{\delta \phi_i}(\phi;\cdot):\spn(\pi_i)\to   \cS^\infty([t,T];\sR^{n_x})$ is linear; 
\item
\label{item:Y=dX_dphi}
  For all $\phi'_i\in \pi_i$, 
$ \frac{\delta X^{t,x}}{\delta \phi_i}(\phi;\phi'_i-\phi_i)
=\cS^\infty \text-\p_\eps X^{t,x, (\phi_i+\eps (\phi_i'-\phi_i),\phi_{-i})}\big\vert_{\eps=0}$.
\end{enumerate}

 \end{Lemma}


 \begin{Lemma}
 \label{lemma:2nd_derivative_state}
  Suppose (H.\ref{assum:regularity_N}) holds. 
Let  
   $(t,x)\in [0,T]\times \sR^{n_x}$, $\phi\in \pi^{(N)}$
   and
   $i,j\in I_N$.
      For all  $\phi'_i\in \spn(\pi_i)$
   and $\phi''_j\in \spn(\pi_j)$,
\eqref{eq:X_sensitivity_second}  admits a unique solution $\frac{\delta^2 X^{t,x}}{\delta \phi_i\delta \phi_j}(\phi;\phi'_i,\phi''_j)\in \cS^\infty([t,T];\sR^{n_x})$
satisfying    the following properties:
\begin{enumerate}[(1)]
\item
\label{item:bilinear_d2X_d2phi}
 The map $\frac{\delta^2 X^{t,x}}{\delta \phi_i\delta \phi_j}(\phi;\cdot,\cdot):\spn(\pi_i)\times \spn(\pi_j)\to   \cS^\infty([t,T];\sR^{n_x})$ is bilinear;
 \item 
 \label{item:symmetry_d2X_d2phi}
For all $\phi'_i\in \pi_i$ and $\phi''_j\in \pi_j$, 
$ \frac{\delta^2 X^{t,x}}{\delta \phi_i\delta \phi_j}(\phi;\phi'_i,\phi''_j)  
=\frac{\delta^2 X^{t,x}}{\delta \phi_j\delta \phi_i}(\phi;\phi''_j,\phi'_i)  $;
\item 
\label{item:Z=d^2X_d^2phi}
For all $\phi'_i\in \spn(\pi_i)$ and $\phi''_j\in \pi_j$, 
$ \frac{\delta^2 X^{t,x}}{\delta \phi_i\delta \phi_j}(\phi;\phi'_i,\phi''_j-\phi_j)
=\cS^\infty \text-\p_\eps 
\frac{\delta X^{t,x}}{\delta \phi_i}((\phi_j+\eps (\phi''_j-\phi_j),\phi_{-j});\phi'_i)
\big\vert_{\eps=0}.
$
\end{enumerate}

 \end{Lemma}


\begin{Lemma}
\label{lemma:control_derivative}
 Suppose (H.\ref{assum:regularity_N}) holds. 
Let   $(t,x)\in [0,T]\times\sR^{n_x}$ and  $\phi\in \pi^{(N)}$.
For each 
$s\in [t,T]$, let   $\alpha^{t,x,\phi}_s=\phi(s,X^{t,x,\phi}_s)$.
Then for all $i,j\in I_N$, 
\begin{enumerate}[(1)]
\item 
\label{item:control_1st}

$
 \frac{\delta \alpha^{t,x}}{\delta \phi_i}(\phi;\cdot):\spn(\pi_i)\to \cH^\infty([t,T];\sR^{n_a})$ is linear,
 and 
$ \frac{\delta \alpha^{t,x}}{\delta \phi_i}(\phi;\phi'_i-\phi_i)
=\cH^\infty \text-\p_\eps \alpha^{t,x, (\phi_i+\eps (\phi_i'-\phi_i),\phi_{-i})}\big\vert_{\eps=0}$
for all $\phi'_i\in \pi_i$;

\item 
\label{item:control_2nd}
$
 \frac{\delta^2 \alpha^{t,x}}{\delta \phi_i\delta\phi_j}(\phi;\cdot,\cdot):\spn(\pi_i)\times\spn(\pi_j) \to \cH^\infty([t,T];\sR^{n_a})
 $ is bilinear, 
  for all $\phi'_i\in \pi_i$ and $\phi''_j\in \pi_j$, 
$ \frac{\delta^2 \alpha^{t,x}}{\delta \phi_i\delta \phi_j}(\phi;\phi'_i,\phi''_j)  
=\frac{\delta^2 \alpha^{t,x}}{\delta \phi_j\delta \phi_i}(\phi;\phi''_j,\phi'_i)  $
and 
for all $\phi'_i\in \spn(\pi_i)$ and $\phi''_j\in \pi_j$,
$ \frac{\delta^2 X^{t,x}}{\delta \phi_i\delta \phi_j}(\phi;\phi'_i,\phi''_j-\phi_j)
=\cH^\infty \text-\p_\eps 
\frac{\delta \alpha^{t,x}}{\delta \phi_i}((\phi_j+\eps (\phi''_j-\phi_j),\phi_{-j});\phi'_i)
\big\vert_{\eps=0} 
$.

\end{enumerate}

\end{Lemma}
  

\begin{proof}[Proofs of Theorem   \ref{thm:differential_MPG_proba}
Items \ref{item:prob_first} and \ref{item:prob_second}]
Fix $\phi\in \pi^{(N)}$. 
The linearity of $\frac{\delta V^{t,x}_i}{\delta \phi_j}(\phi;\cdot):   \spn(\pi_j) \to \sR$
follows from
the linearity of 
$ \frac{\delta X^{t,x}}{\delta \phi_j} (\phi;\cdot)$
and 
$ \frac{\delta \alpha^{t,x}}{\delta \phi_j} (\phi;\cdot)$
  shown in 
 Lemmas     \ref{lemma:1st_derivative_state} and \ref{lemma:control_derivative}. 
Let $\phi'_j\in \pi_j$ and for each $\eps\in [0,1)$, 
 let $\phi^\eps=(\phi_j+\eps(\phi'_j-\phi_j), \phi_{-j})$.
 Recall that 
 \begin{align*}
  V^{t,x}_i(\phi^\eps) = \sE \left[\int_t^T f_i(s,X^{t,x,\phi^\eps}_s, \alpha^{t,x,\phi^\eps}_s)\d s + g_i(X^{t,x,\phi^\eps}_T)\right].
 \end{align*}
 Then   \cite[Theorem 9, p.~97]{krylov2008controlled},  the $\cS^\infty$-differentiability of $X^{t,x,\phi^\eps}$ and the $\cH^\infty$-differentiability of $\alpha^{t,x,\phi^\eps}$
 (see  Lemmas     \ref{lemma:1st_derivative_state},   \ref{lemma:2nd_derivative_state} and \ref{lemma:control_derivative})
 imply that 
 $\frac{\d }{\d \eps}V^{t,x}_i(\phi^\eps)\vert_{\eps=0}=  \frac{\delta V^{t,x}_i}{\delta \phi_j} (\phi;\phi'_j-\phi_j)
$,
hence  Theorem   \ref{thm:differential_MPG_proba}
 Item \ref{item:prob_first}.

To prove Theorem   \ref{thm:differential_MPG_proba}
 Item \ref{item:prob_second}, 
for all $(k,\ell)\in \{(i,j),(j,i)\}$,
observe that 
the bilinearity of $\frac{\delta^2 V^{t,x}_i}{\delta \phi_k\delta \phi_\ell}:\pi^{(N)} \times  \spn(\pi_k)\times \spn(\pi_\ell) \to \sR$ follows from \eqref{eq:value_derivative_2nd}, the linearity of 
$ \frac{\delta X^{t,x}}{\delta \phi_k} (\phi;\cdot)$
and 
$ \frac{\delta \alpha^{t,x}}{\delta \phi_\ell} (\phi;\cdot)$,
and the bilinearity of 
$   \frac{\delta^2 X^{t,x}}{\delta \phi_k\phi_\ell} (\phi;\cdot,\cdot)$
and 
$ \frac{\delta^2 \alpha^{t,x}}{\delta \phi_k\phi_\ell} (\phi;\cdot,\cdot)$. 
The facts that $\frac{\delta^2 V^{t,x}_i}{\delta \phi_i\delta \phi_j}$ is a linear derivative of $\frac{\delta V^{t,x}_i}{\delta \phi_i}$ in $\pi_j$ 
and that 
$\frac{\delta^2 V^{t,x}_i}{\delta \phi_j\delta \phi_i}$ is a linear derivative of $\frac{\delta V^{t,x}_i}{\delta \phi_j}$ in $\pi_i$ 
can be proved by similar arguments as those employed for the first-order derivatives. 
This finishes the proof. 
\end{proof}

\begin{proof}[Proof of Theorem   \ref{thm:differential_MPG_proba}
Item \ref{item:prob_potential}]

We only prove   the characterization of CLPGs,
which   follows by applying   Theorem \ref{thm:symmetry_value_sufficient}
to the   game $\cG_{\rm prob}$ with fixed $(t,x)\in [0,T]\times \sR^{n_x}$. 
The convexity of $\pi_i$ follows from the convexity of $A_i$ in  (H.\ref{assum:regularity_N}),
the linear differentiability of $(V_i)_{i\in I_N}$ has been proved in
Theorem   \ref{thm:differential_MPG_proba}
Items \ref{item:prob_first} and \ref{item:prob_second}, 
 and Condition  \ref{item:symmetry-Vi}  in   Theorem \ref{thm:symmetry_value_sufficient}
 has been imposed in Condition \eqref{eq:prob_symmetry}. 
 Hence it remains to 
verify  Conditions \ref{item:integrable_p-Vi} and \ref{item:joint_continuity}  in   Theorem \ref{thm:symmetry_value_sufficient}
for given  
 $i,j\in I_N$, 
 $ z, 
 \phi \in \pi^{(N)}$, 
  $\phi'_i,\tilde{\phi}'_i\in \pi_i$  
and $ {\phi}''_j \in \pi_j$.

To verify Condition \ref{item:integrable_p-Vi},
for all $r, \eps\in [0,1]$,
let $\phi^{r,\eps}=
z+r( (\phi_i+\eps(\tilde{\phi}'_{i}-\phi_i),\phi_{-i})-z) $.
To prove $ \sup_{r,\eps\in [0,1]}\left|  \frac{\delta^2 V_i^{t,x}}{\delta\phi_i\delta\phi_j}\big(\phi^{r,\eps}; {\phi}'_i,{\phi}''_j\big)  \right|<\infty$, 
by \eqref{eq:value_derivative_2nd} and 
 the polynomial growth of   $f_i$ and $g_i$ and their derivatives (see (H.\ref{assum:regularity_N})),  it suffices to 
show for all $q\ge 1$, there exists $C_q$ such that for all $r,\eps\in [0,1]$, 
\bb
\label{eq:uniform_moment_bound_X}
 \|X^{t,x,\phi^{r,\eps}}\|_{\cS^q},
  \left\| \frac{\delta X^{t,x}}{\delta \phi_i} (\phi^{r,\eps};\phi'_i)\right\|_{\cS^q},
 \left\| \frac{\delta X^{t,x}}{\delta \phi_j} (\phi^{r,\eps};\phi''_j)\right\|_{\cS^q},
  \left\|   \frac{\delta^2 X^{t,x}}{\delta \phi_i\phi_j} (\phi^{r,\eps};\phi'_i,\phi''_j)\right\|_{\cS^q}
\le C_q,
\ee
and 
\bb
\label{eq:uniform_moment_bound_alpha}
 \|\alpha^{t,x,\phi^{r,\eps}}\|_{\cH^q},
  \left\| \frac{\delta \alpha^{t,x}}{\delta \phi_i} (\phi^{r,\eps};\phi'_i)\right\|_{\cH^q},
 \left\| \frac{\delta \alpha^{t,x}}{\delta \phi_j} (\phi^{r,\eps};\phi''_j)\right\|_{\cH^q},
  \left\|   \frac{\delta^2 \alpha^{t,x}}{\delta \phi_i\phi_j} (\phi^{r,\eps};\phi'_i,\phi''_j)\right\|_{\cH^q}
\le C_q.
\ee
Since $\phi^{r,\eps}$ is a convex combination of 
policies $\phi$, $z$ and $(\tilde{\phi}'_i,\phi_{-i})$,
there exists $L\ge 0$ such that for all $r,\eps\in [0,1]$, 
$$
\sup_{(t,x)\in [0,T]\times \sR^{n_x}}({|\phi^{r,\eps}(t,0)|}+|(\p_x \phi^{r,\eps})(t,x)|+|(\p_{xx} \phi^{r,\eps})(t,x)|)\le L.
$$ 
Hence,
the uniform moment estimate \eqref{eq:uniform_moment_bound_X}
follows by applying standard moment estimates of SDEs (see \cite[Theorem 3.4.3]{zhang2017backward})
to \eqref{eq:state_policy}, 
\eqref{eq:X_sensitivity_first} 
and \eqref{eq:X_sensitivity_second}.
The estimate 
\eqref{eq:uniform_moment_bound_alpha}
then follows from  
\eqref{eq:control_1st}, \eqref{eq:control_2nd}
and the Cauchy-Schwarz inequality.
This verifies Condition \ref{item:integrable_p-Vi}.

To verify Condition \ref{item:joint_continuity},
for all $\eps=(\eps_i)_{i=1}^N\in [0,1]^N$,
let  $\phi^\eps=  (z_i+\eps_i({\phi}_{i}-z_i))_{i\in I_N}$.
Observe that  $\phi^\eps_i$ is a convex combination of $\phi_i$ and $z_i$, 
there exists $L\ge 0$ such that for all $\eps\in [0,1]^N$,
$$
\sup_{(t,x)\in [0,T]\times \sR^{n_x}}({|\phi^{\eps}(t,0)|}+|(\p_x \phi^{\eps})(t,x)|+|(\p_{xx} \phi^{\eps})(t,x)|)\le L,
$$ 
from which one can deduce  uniform moment estimates 
of the state and control processes 
as in \eqref{eq:uniform_moment_bound_X}
and \eqref{eq:uniform_moment_bound_alpha}.
As $\lim_{\eps_i\downarrow 0}\phi^\eps_i =\phi^0_i $ a.e., 
by   stability results of SDEs (e.g., \cite[Theorem 3.4.2]{zhang2017backward}),
\begin{align*}
&\lim_{\eps\to 0} \left(X^{t,x,\phi^{\eps}} ,  \frac{\delta X^{t,x}}{\delta \phi_i} (\phi^{\eps};\phi'_i),  \frac{\delta X^{t,x}}{\delta \phi_j} (\phi^{\eps};\phi''_j), \frac{\delta^2 X^{t,x}}{\delta \phi_i\phi_j} (\phi^{\eps};\phi'_i,\phi''_j)\right)
 \\
&\q = \left( X^{t,x,\phi^{0}} ,  \frac{\delta X^{t,x}}{\delta \phi_i} (\phi^{0};\phi'_i),  \frac{\delta X^{t,x}}{\delta \phi_j} (\phi^{0};\phi''_j), \frac{\delta^2 X^{t,x}}{\delta \phi_i\phi_j} (\phi^{0};\phi'_i,\phi''_j)
\right) 
\quad \textnormal{a.e.~$[t,T]\times \Omega$.}
\end{align*}
By \eqref{eq:control_1st}, \eqref{eq:control_2nd} and the convergence of $\phi^\eps$,
similar convergence results also hold for the control processes. 
Hence by     Lebesgue's dominated convergence theorem,  
$\lim_{\eps\to 0} \frac{\delta^2 V_i^{t,x }}{\delta \phi_i\phi_j}(\phi^\eps ;\phi'_i, \phi''_j) 
= \frac{\delta^2 V_i^{t,x }}{\delta \phi_i\phi_j}(\phi^0 ;\phi'_i, \phi''_j)$. This verifies Condition \ref{item:joint_continuity}.
\end{proof}  

\subsection{Proof of Theorem \ref{thm:lq_characterisation} }
\label{sec:proof_lq_equivlance}

\begin{proof}[Proof of Theorem \ref{thm:lq_characterisation}]
Fix   $i,h,\ell \in I_N$, 
$\phi\in \pi^{(N)}$, $\phi'_h\in \pi_h$, $\phi''_\ell\in \pi_\ell$
and $(t,x)\in [0,T]\times\sR^{n_x}$.
Denote by $X$ the state process $X^{t,x,\phi}$. 
We first prove the equivalence of   first-order linear derivatives. 
Applying It\^{o}'s formula to 
$s\mapsto \frac{1}{2} X_s ^\top \Theta^h_i(s)  X_s+ \theta^h_i(s)$ gives that 
\begin{align*}
&\left(\frac{1}{2} X_T ^\top \Theta^h_i(T) X_T+ \theta^h_i(T)\right) -\left(\frac{1}{2}x^\top \Theta^h_i(t)  x+ \theta^h_i(t)\right)
\\
&\q
=  \sE\left[\int_t^T\left( \frac{1}{2} X_s^\top \dot\Theta^h_i   X_s + \dot\theta^h
+ \frac{1}{2} \la (\Theta^h_i+ (\Theta^h_i)^\top) X_s, (A+BK)X_s\ra + \frac{1}{2}\tr\left(\sigma\sigma^\top \Theta^{h}_i  \right) \right)\d s
 \right],
\end{align*}
which along with   \eqref{eq:quadratic_1st} and   \eqref{eq:ode_theta}  shows that  
\begin{align*}
{\frac{\delta \ol{ V}^\phi_i}{\delta \phi_h}}(t,x; \phi'_h)
 &=
  \frac{1}{2}  \sE\left[\int_t^T X_s ^\top \left( 
 \Psi_i B_h K'_h+(B_h K'_h)^\top  \Psi_i
 +K^\top (R_i)_{ h}K'_h
+ (K'_h)^\top \big((R_i)_{ h}\big)^\top K
\right)X_s\d s
 \right]
 \\
 & =
  \sE\left[\int_t^T  \left( 
 X_s^\top \Psi_i B_h K'_hX_s  
 +(KX_s)^\top (R_i)_{ h}K'_h X_s\right)\d s
 \right],
\end{align*}
where we used the symmetry of $ \Psi_i(\cdot)$. 
Then by \eqref{eq:value_derivative_1st_lq}, 
\begin{align}
\label{eq:difference_1st_derivative_term1}
\begin{split}
& \frac{\delta V^{t,x}_i}{\delta \phi_h} (\phi;\phi'_h)
 - {\frac{\delta \ol{ V}^\phi_i}{\delta \phi_h}}(t,x; \phi'_h)
\\
&\q =
  \sE \bigg[\int_t^T 
 \bigg( X_s^\top Q_i  Y^{h}_s
 +(K X_s)^\top  
R_i  K Y^{h}_s 
 - X_s^\top \Psi_i B_h K'_hX_s 
 \bigg)
  \d s  + X_T^\top G_i  Y^{h}_T\bigg].
  \end{split}
\end{align}
Applying It\^{o}'s formula to 
$s\mapsto    X_s^\top  \Psi_i (s)  Y^h_s$ and using \eqref{eq:X_sensitivity_first_lq}  and  \eqref{eq:ode_psi} imply that 
\begin{align*}
\begin{split}
\sE[X_T^\top G_i  Y^{h}_T]
& =    \sE \bigg[\int_t^T 
 \bigg( X_s^\top \dot \Psi_i  Y^{h}_s
 +\la   \Psi_i X_s, (A+BK)  Y^h_s+B_hK'_h X_s \ra 
 + \la   \Psi_i Y^h_s, (A+BK)  X_s \ra 
 \bigg)
  \d s\bigg]  
  \\
& =  -  \sE \bigg[\int_t^T 
 \bigg( X_s^\top  (Q_i+K^\top R_i K)  Y^{h}_s
- X_s ^\top \Psi_i B_hK'_h X_s
 \bigg)
  \d s\bigg].    
  \end{split}
\end{align*}
This along with  \eqref{eq:difference_1st_derivative_term1} proves 
$\frac{\delta V^{t,x}_i}{\delta \phi_h} (\phi;\phi'_h)
 - {\frac{\delta \ol{ V}^\phi_i}{\delta \phi_h}}(t,x; \phi'_h)=0$.

Next we  prove the equivalence of   second-order linear derivatives. 
By   \eqref{eq:value_derivative_2nd_lq}
and $R_iE_h=(R_i)_h$,
\begin{align}
 \label{eq:prob_term1} 
 \begin{split} 
  \frac{\delta^2 V^{t,x}_{i}}{\delta \phi_h \delta \phi_\ell} (\phi;\phi'_h, \phi''_\ell)
&   = \sE \bigg[\int_t^T
\bigg( Y^\ell_s  Q_i  Y^h_s
+ (K  Y^\ell_s)^\top R_i K Y^h_s
\\
&\quad   
  +(K Y^h_s)^\top (R_i)_\ell K''_\ell X_s
  +(K  Y^\ell_s)^\top (R_i)_h K'_h X_s
  +(K''_\ell  X_s)^\top (R_i)_{\ell h} K'_h X_s
 \\
&\q   +  
X^\top_s Q_i Z_s
+(K  X_s)^\top 
\Big( R_i  K Z_s +(R_i )_\ell K''_\ell  Y^h_s+(R_i )_hK'_h  Y^\ell_s\Big)
\bigg)  \d s\bigg]
 \\
 &\q 
+\sE\left[ (Y^\ell_T  )^\top G_i Y^h_T +X^\top_T G_i Z_T\right].
 \end{split}
\end{align}  
Applying It\^{o}'s formula to 
$s\mapsto (Y^\ell_s  )^\top \Psi_i(s)  Y^h_s +X^\top_s \Psi_i(s) Z_s$ 
and using \eqref{eq:X_sensitivity_first_lq}, \eqref{eq:X_sensitivity_second_lq} and  \eqref{eq:ode_psi},
\begin{align*}
&\sE\left[ (Y^\ell_T  )^\top G_i Y^h_T +X^\top_T G_i Z_T\right]
\\
&\q = \sE\left[ \int_t^T \left( \la  \Psi_i Y^\ell_s ,  \d Y^h_s\ra + \la \d Y^\ell_s, \Psi_i Y^h_s\ra 
  +\la  \Psi_i  X_s , \d Z_s\ra+ \la  \d X_s ,  \Psi_i  Z_s\ra \right)\d s \right]
  \\
&\qq    +  \sE\left[ \int_t^T \left((Y^\ell_s)^\top \dot  \Psi_i   Y^h_s +X_s^\top  \dot \Psi_i Z_s\right) \d s\right]
\\
&\q = \sE\left[ \int_t^T \left((Y^\ell_s)^\top    \Psi_i  B_h K'_h X_s  
+(Y^h_s)^\top    \Psi_i  B_\ell K''_\ell X_s  
 +X_s^\top    \Psi_i (B_\ell K''_\ell Y^h_s
 +B_h K'_h Y^\ell_s) \right) \d s\right]
 \\
 &\qq    -  \sE\left[ \int_t^T \left((Y^\ell_s)^\top (Q_i+K^\top R_i K)  Y^h_s +X_s^\top  (Q_i+K^\top R_i K) Z_s\right) \d s\right].
\end{align*}
Substituting this into \eqref{eq:prob_term1}  implies that 
   \begin{align}
 \label{eq:prob_term2} 
 \begin{split} 
  &\frac{\delta^2 V^{t,x}_{i}}{\delta \phi_h \delta \phi_\ell} (\phi;\phi'_h, \phi''_\ell)
  \\
&  \q  = \sE \bigg[\int_t^T
\bigg( 
(Y^\ell_s)^\top  \left(  \Psi_i  B_h K'_h 
+(B_h K'_h )^\top \Psi_i
+K^\top (R_i)_h K'_h
+ (K'_h)^\top((R_i )_h)^\top K 
\right)
X_s  
\\
&\qq 
+(Y^h_s)^\top 
\left(   \Psi_i  B_\ell K''_\ell 
+ (B_\ell K''_\ell )^\top \Psi_i
+K^\top (R_i)_\ell K''_\ell
+ (K''_\ell)^\top((R_i )_\ell)^\top K
\right) X_s  
\\
&\qq 
  +(K''_\ell  X_s)^\top (R_i)_{\ell h} K'_h X_s
\bigg)  \d s\bigg].
 \end{split}
\end{align}   
Meanwhile,
applying It\^{o}'s formula to 
$s\mapsto \frac{1}{2} X_s ^\top \Lambda^{h,\ell}_i(s)  X_s+ \lambda^{h,\ell}_i (s)$ 
and using  \eqref{eq:quadratic_2nd}  and  \eqref{eq:ode_lambda} 
give
\begin{align}
\label{eq:analytic_term1}
\begin{split}
{\frac{\delta^2 \ol{V}^\phi_i}{\delta \phi_h\delta \phi_\ell}}(t,x; \phi'_h, \phi''_\ell)
 & =
  \sE\left[\int_t^T  \left( 
 X_s^\top  \Theta^{h}_iB_\ell K''_\ell X_s + X_s^\top  \Theta^{\ell}_i B_h K'_h  X_s  
 +(K''_\ell X_s)^\top (R_i)_{ \ell h}K'_h X_s\right)\d s
 \right].
 \end{split}
 \end{align}
 Finally, applying   It\^{o}'s formula to 
$s\mapsto   (Y^h_s  )^\top \Theta^\ell_i(s)  X_s + (Y^\ell_s  )^\top \Theta^h_i(s)  X_s $ 
and using  \eqref{eq:X_sensitivity_first_lq} 
yield
 \begin{align*}
0&=   \sE\bigg[ \int_t^T \bigg( 
 (Y^h_s  )^\top \dot \Theta^\ell_i  X_s + (Y^\ell_s  )^\top \dot \Theta^h_i  X_s
 +
\la   \Theta^\ell_i   X_s ,   (A +B K )  Y^{h}_s   
+B_h K'_h X_s  \ra
  \\
&\q  + \la   \Theta^\ell_i   Y^h_s ,   (A +B K )  X_s   
 \ra + \la   \Theta^h_i   X_s ,   (A +B K )  Y^{\ell}_s   
+B_\ell K''_\ell X_s  \ra
+ \la   \Theta^h_i Y^\ell_s ,   (A +B K )  X_s   \ra
   \bigg) \d s \bigg]
\\
&= \sE\bigg[ \int_t^T \bigg( 
 (Y^h_s  )^\top \left( \dot \Theta^\ell_i +(A +B K )^\top  \Theta^\ell_i + \Theta^\ell_i (A+BK)  \right) X_s 
 \\
 &\q  
  + (Y^\ell_s  )^\top \left( \dot \Theta^h_i +(A +B K )^\top  \Theta^h_i + \Theta^h_i  (A +B K )  \right) X_s 
 +
   X^\top _s     \Theta^\ell_i   B_h K'_h X_s   
+   X^\top_s    \Theta^h_i    B_\ell K''_\ell X_s   
  \bigg)\d s \bigg], 
\end{align*}
 which along with \eqref{eq:ode_theta}, \eqref{eq:prob_term2}  and 
 \eqref{eq:analytic_term1} implies that 
 $\frac{\delta^2 V^{t,x}_{i}}{\delta \phi_h \delta \phi_\ell} (\phi;\phi'_h, \phi''_\ell) = {\frac{\delta^2 \ol{V}^\phi_i}{\delta \phi_h\delta \phi_\ell}}(t,x; \phi'_h, \phi''_\ell)$. 
 
We then prove the equivalence of  Conditions \eqref{eq:prob_symmetry} and
  \eqref{eq:lq_symmetry}.
   By \eqref{eq:equivalence}, it is clear that     \eqref{eq:lq_symmetry} implies \eqref{eq:prob_symmetry}.
   On the other hand, if \eqref{eq:prob_symmetry} holds,    by   \eqref{eq:zero_order_term} and \eqref{eq:equivalence},
$$
x^\top(     \Lambda^{i,j}_i  (t) -  \Lambda^{j,i}_j(t))x+\int_t^T \tr\left(\sigma\sigma^\top (\Lambda^{i,j}_i(s)-\Lambda^{j,i}_j (s) )\right)\d s =0
\q    \fa (t,x)\in [0,T]\times \sR^{n_x}, 
$$
which  along with
 $\spn\{xx^\top |x\in\sR^{n_x}\}=\sS^{n_x}$ (which can be shown by  the eigenvalue decomposition)  implies that
 $\Lambda^{i,j}_i  (T)-  \Lambda ^{j,i}_j(T)=0$
 and 
    $\dot{ \Lambda}^{i,j}_i   -\dot{ \Lambda}^{j,i}_j+\Lambda^{i,j}_i-\Lambda^{j,i}_j =0 $.   This shows that 
     \eqref{eq:lq_symmetry} holds. 
     
     Finally, by \eqref{eq:zero_order_term}, it is clear that \eqref{eq:lq_symmetry_mpg} implies   \eqref{eq:lq_symmetry}
     for all $(t,x)\in [0,T]\times \sR^{n_x}$, which shows that $\cG_{\rm LQ}$ is an MPG. Conversely, 
     suppose that $\cG_{\rm LQ}$ is an MPG. Using    the fact the policy class $\pi_i$ in  \eqref{eq:policy_lq}  is a vector space and the mean-value theorem,
     one can show that    
  \eqref{eq:lq_symmetry} must hold  for all $(t,x)\in [0,T]\times \sR^{n_x}$,
  which implies that $ x^\top( \Lambda^{i,j}_i(t)   -
  \Lambda^{j,i}_j(t))x=0  $  for all $(t,x)\in [0,T]\times \sR^{n_x}$.
  This along with $  \Lambda^{i,j}_i(t)   -
  \Lambda^{j,i}_j(t)\in \sS^{n_x} $  implies that $ \Lambda^{i,j}_i(t)   =
  \Lambda^{j,i}_j(t)$  for all $t\in [0,T]$.
\end{proof}

 \section{Proofs of Technical Lemmas}
 \label{sec:techinical_lemma}

 \subsection{Proofs of Lemmas   \ref{lemma:derivative_line} and   \ref{lemma:multi-dimension_derivative}}

 \begin{proof}[Proof of Lemma \ref{lemma:derivative_line}]
 By Definition \ref{eq:first_der_def} and 
 the linear differentiability of $f$, 
$[0,1]\ni \eps\mapsto f(\phi^\eps)\in \sR$ admits a right-hand derivative 
$ \frac{\delta f}{\delta \phi_i} (\phi ; \phi'_i-\phi_i)$
at $\eps=0$. 
We now  prove the differentiability of $\eps\mapsto f(\phi^\eps)$ at $\eps_0\in (0,1)$.
For all $\eps>0$ such that $\eps_0+\eps<1$, 
\begin{align*}
\phi^{\eps_0+\eps} = ( \phi^{\eps_0}_i+\eps(\phi'_i-\phi_i),\phi_{-i})=\left( \phi^{\eps_0}_i+\frac{\eps}{1-\eps_0}(\phi'_i-\phi^{\eps_0}_i),\phi_{-i}\right).
\end{align*}
Thus  by the linear differentiability of $f$, 
\begin{align*}
\lim_{\eps\downarrow   0 }\frac{  f(\phi^{\eps_0+\eps} ) -f (\phi^{\eps_0} ) }{  \eps}
&= \frac{1}{ 1-\eps_0}
\lim_{\eps\downarrow   0 }\frac{  f\left(\left( \phi^{\eps_0}_i+\frac{\eps}{1-\eps_0}(\phi'_i-\phi^{\eps_0}_i),\phi_{-i}\right) \right) -f (\phi^{\eps_0} ) }{  \eps/(1-\eps_0)}
\\
&=  \frac{1}{ 1-\eps_0} \frac{\delta f}{\delta \phi_i} (\phi^{\eps_0} ; \phi'_i-\phi^{\eps_0}_i)=  \frac{\delta f}{\delta \phi_i} (\phi^{\eps_0} ; \phi'_i-\phi_i),
\end{align*}
where the last identity used the linearity of $\frac{\delta f}{\delta \phi_i} (\phi^{\eps_0} ; \cdot):  \spn(\pi_i) \to \sR$.
On the other hand, 
for all $\eps>0$ such that $\eps_0-\eps>0$, 
$
\phi^{\eps_0-\eps} = ( \phi^{\eps_0}_i-\eps(\phi'_i-\phi_i),\phi_{-i})=\left( \phi^{\eps_0}_i+\frac{\eps}{\eps_0}(\phi_i-\phi^{\eps_0}_i),\phi_{-i}\right).
$
Hence  by the linear differentiability of $f$, 
\begin{align*}
\lim_{\eps\downarrow   0 }\frac{  f(\phi^{\eps_0-\eps} ) -f (\phi^{\eps_0} ) }{ - \eps}
&= -\frac{1}{\eps_0} 
\lim_{\eps\downarrow   0 }\frac{  f\left(\left( \phi^{\eps_0}_i+\frac{\eps}{\eps_0}(\phi_i-\phi^{\eps_0}_i),\phi_{-i}\right) \right) -f (\phi^{\eps_0} ) }{  \eps/\eps_0 }
\\
&=- \frac{1}{\eps_0} \frac{\delta f}{\delta \phi_i} (\phi^{\eps_0} ; \phi_i-\phi^{\eps_0}_i)=  \frac{\delta f}{\delta \phi_i} (\phi^{\eps_0} ; \phi'_i-\phi_i).
\end{align*}
This proves  the differentiability of $\eps\mapsto f(\phi^\eps)$ at $\eps_0\in (0,1)$. 
Similarly, $  \eps\mapsto f(\phi^\eps)$ admits a left-hand derivative at $1$.
 \end{proof}

\begin{proof}[Proof of Lemma \ref{lemma:multi-dimension_derivative}]
 Let   $z,\phi\in \pi^{(N)}$  be fixed. 
 For simplicity, 
we assume $N=2$ and $I_N= \{i,j\}$,   as the same argument can be easily extended to $N\ge 3$. 

We  first  prove the differentiability of 
 $ r \mapsto  f(z+r( \phi -z)) $  admits a right-hand derivative 
at $0$.  
For each $r\in [0,1]$, let $\phi^r= z+r( \phi -z)$.
Recall that by Lemma \ref{lemma:derivative_line}, for all $r\in [0,1]$, 
$\eps\mapsto f ((\phi^{r}_{-j}, z_j+ \eps(\phi_j-z_j)))$
is differentiable on $[0,1]$ 
and 
$$
\frac{\d}{\d \eps}f((\phi^{r}_{-j}, z_j+ \eps(\phi_j-z_j)))
= \frac{\delta  f}{ \delta\phi_j}((\phi^{r}_{-j}, z_j+ \eps(\phi_j-z_j)) ;   \phi_j-z_j)
\quad \fa \eps\in (0,1).
$$
Then for all $r  \in (0,1)$,   using the fact that $\phi^r_{-i}=\phi^r_j$ (note that $I_N= \{i,j\}$), 
\begin{align*}
 f(\phi^{r}  )-f(z ) 
  &   = f(\phi^{r} )
- f(( \phi^r_i, z_{-i}))
+ f(( \phi^r_i, z_{-i})  )
-f(z )  
\\
&  = \frac{\delta  f}{ \delta\phi_j}( (z_i+r(\phi_i-z_i),z_j+r\eps_r (\phi_j-z_j)   ) ;  \phi_j-z_j)r
+ f(( \phi^r_i, z_{-i})  )
-f(z ),  
\end{align*}
for some $\eps_r \in (0,1)$, 
where the last identity used the mean value theorem.
Dividing both sides of the above identity by $r$ and letting $r\to 0$ yield
\begin{align*}
\lim_{r\downarrow 0 }\frac{f(\phi^{r}  )-f(z  )}{r}
&  =\lim_{r\downarrow 0}  \frac{\delta  f}{ \delta\phi_j}( (z_i+r(\phi_i-z_i),z_j+r\eps_r (\phi_j-z_j)   ) ;  \phi_j-z_j) 
+ \frac{\delta  f}{\delta\phi_i }(z ;   \phi_i-z_i)
\\
&  =  \frac{\delta  f}{ \delta\phi_j}( z ;  \phi_j-z_j) 
+ \frac{\delta  f}{\delta\phi_i }(z ;  \phi_i-z_i),
\end{align*}
where the last identity used the continuity of 
$
 {\eps}\mapsto  \frac{\delta  f}{\delta \phi_j}(z+\eps\cdot (\phi-z) ;  \phi_j-z_j)
$
 at $0$,
 which holds due to the continuity  assumption of $\frac{\delta f}{\delta \phi_j}$ 
  and the linearity of  $ \frac{\delta f}{\delta \phi_j} $  with respect to the last   argument.      
This proves the desired differentiability of $r\mapsto f (\phi^{r} )$ at $0$.
Similarly, 
$r\mapsto f (\phi^{r} )$ is differentiable at   $r\in (0,1]$.
\end{proof}

\subsection{Proofs of Lemmas   \ref{lemma:1st_derivative_state},  \ref{lemma:2nd_derivative_state}
 and \ref{lemma:control_derivative}}
\label{sec:differentiability_lemma}
 \begin{proof}[Proof of Lemma   \ref{lemma:1st_derivative_state}]
 Since $B$, $\Sigma$ and $\phi$ have bounded first-order derivatives and $\phi'_i\in \spn(\pi_i)$ is of linear growth in $x$, 
 standard a-priori estimates of   SDEs (see e.g., \cite[Theorem 3.4.3]{zhang2017backward})
 show that \eqref{eq:X_sensitivity_first} admits a unique solution $\frac{\delta X^{t,x}}{\delta \phi_i}(\phi;\phi'_i)\in \cS^\infty([t,T];\sR^{n_x})$. 
As the terms $(\p_{a_i} B^\phi)[\phi'_i](s, X^{t,x, \phi}_s  )$ 
and $(\p_{a_i} \Sigma^\phi)[\phi'_i](s, X^{t,x, \phi}_s  )  $ are linear with respect to $\phi'_i$, 
one can easily verify by the uniqueness of solutions  that for all $\alpha, \beta\in \sR$ and $\phi'_i,\phi_i''\in \spn(\pi_i)$, 
$
\frac{\delta X^{t,x}}{\delta \phi_i}(\phi;\alpha{\phi_i'}+\beta \phi_i'' )
=\alpha\frac{\delta X^{t,x}}{\delta \phi_i}(\phi;{\phi_i'} )
++\beta\frac{\delta X^{t,x}}{\delta \phi_i}(\phi;  \phi_i'' ).
$
This proves the linearity of    $\spn(\pi_i)\ni \phi'_i\mapsto \frac{\delta X^{t,x}}{\delta \phi_i}(\phi;\phi'_i)\in \cS^\infty([t,T];\sR^{n_x}) $ in Item \ref{item:linear_dX_dphi}. 
Finally,   observe that 
for all $\eps\in [0,1)$, $ X^{t,x, (\phi_i+\eps (\phi_i'-\phi_i),\phi_{-i})}$  satisfies 
\begin{align*} 
\d X_s & = B(s, X_s, (\phi_i+\eps (\phi'_i-\phi_i),\phi_{-i})(s,X_s) )\d s 
\\
&\quad 
+ \Sigma(s,   X_s , (\phi_i+\eps (\phi'_i-\phi_i),\phi_{-i})(s,X_s))\d W_s, 
\q s\in [t,T];
\quad 
 X_t   =x. 
 \end{align*}
 As 
$B,\Sigma$, $\phi$ and $\phi'_i$
are continuous differentiable in $(x,a)$,
$\cS^\infty \text-\p_\eps  X^{t,x, (\phi_i+\eps (\phi_i'-\phi_i),\phi_{-i})}\big\vert_{\eps=0}$ exists  
due to    \cite[Theorem 4, p.~105]{krylov2008controlled}
and satisfies 
\eqref{eq:X_sensitivity_first}
due to    \cite[Remark 5, p.~108]{krylov2008controlled}. 
The uniqueness of solutions to \eqref{eq:X_sensitivity_first}
then yields Item \ref{item:Y=dX_dphi}.
 \end{proof}

 \begin{proof}[Proof of Lemma   \ref{lemma:2nd_derivative_state}]
 Standard well-posedness results   of   SDEs
 and the regularity conditions of $B$, $\Sigma$, $\phi$, $\phi'_i$ and $\phi''_j$
 show that 
 \eqref{eq:X_sensitivity_second}  admits a unique solution
 $\frac{\delta^2 X^{t,x}}{\delta \phi_i\delta \phi_j}(\phi;\phi'_i,\phi''_j)$.
 
For Item \ref{item:bilinear_d2X_d2phi}, 
 observe that by Lemma  \ref{lemma:1st_derivative_state} and \eqref{eq:F_B},
 for any given $\phi''_j\in \spn(\pi_j)$, the map 
 $\phi'_i\mapsto  \big(F_B, F_\Sigma\big)\left(s,X^{t,x, \phi}_s, \frac{\delta X^{t,x}}{\delta \phi_i}(\phi;\phi'_i), \frac{\delta X^{t,x}}{\delta \phi_j}(\phi;\phi''_j), \phi'_i, \phi''_j\right)$
  is linear. 
  This along with the linearity of the SDE 
\eqref{eq:X_sensitivity_second}
shows that 
  $\phi'_i\mapsto \frac{\delta^2 X^{t,x}}{\delta \phi_i\delta \phi_j}(\phi;\phi'_i,\phi''_j)$
  is linear. 
  Similar arguments yield 
the linearity of   $\phi''_j\mapsto \frac{\delta^2 X^{t,x}}{\delta \phi_i\delta \phi_j}(\phi;\phi'_i,\phi''_j)$. 
 
For Item \ref{item:symmetry_d2X_d2phi},  
  observe that 
  $
  F_B(t,x, y_1, y_2,  \phi'_i, \phi''_j )
  =  F_B (t,x, y_2, y_1,  \phi''_j,\phi'_i )
  $,
  as $B$ is twice continuously differentiable with respect to $(x,a)$ and hence has a symmetric Jacobian. 
Similarly, we have  
    $F_\Sigma \left(t,x, y_1, y_2,  \phi'_i, \phi''_j\right)
  =  F_\Sigma\left(t,x, y_2, y_1,  \phi''_j,\phi'_i\right)$.
  This proves that  
both  $ \frac{\delta^2 X^{t,x}}{\delta \phi_i\delta \phi_j}(\phi;\phi'_i,\phi''_j)  $ and 
$\frac{\delta^2 X^{t,x}}{\delta \phi_j\delta \phi_i}(\phi;\phi''_j,\phi'_i)  $
satisfy \eqref{eq:X_sensitivity_second},
which along with  
  the uniqueness of solutions to   \eqref{eq:X_sensitivity_second}
  yields Item \ref{item:symmetry_d2X_d2phi}.
  
  Finally,  
for all $\eps\in [0,1)$, 
let $\phi^\eps =(\phi_j+\eps (\phi''_j-\phi_j),\phi_{-j})$.
By Lemma   \ref{lemma:1st_derivative_state},
$ \frac{\delta X^{t,x}}{\delta \phi_i}(\phi^\eps ;\phi'_i)$  satisfies 
 \begin{align*} 
\begin{split}
\d Y_s & = (\p_x B^{\phi^\eps}) (s, X^{t,x, \phi^\eps}_s  )  Y_s   \d s 
 +  \big((\p_x \Sigma^{\phi^\eps}) (s, X^{t,x, {\phi^\eps}}_s  )  Y_s  \big)
\d W_s
\\
&\quad 
+ (\p_{a_i} B^{\phi^\eps})[\phi'_i](s, X^{t,x, {\phi^\eps}}_s  )  \d s
+ (\p_{a_i} \Sigma^{\phi^\eps})[\phi'_i](s, X^{t,x, {\phi^\eps}}_s  )  \d W_s
\quad \fa  s\in [t,T].
 \end{split}
 \end{align*}
By     \cite[Theorem 4, p.~105]{krylov2008controlled}, 
$\cS^\infty \text-\p_\eps 
\frac{\delta X^{t,x}}{\delta \phi_i}(\phi^\eps ;\phi'_i)
\big\vert_{\eps=0}$ exists and satisfies 
\eqref{eq:X_sensitivity_second},
which along with  
  the uniqueness of solutions to   \eqref{eq:X_sensitivity_second}
  yields Item  \ref{item:Z=d^2X_d^2phi}.  
This finishes the proof. 
 \end{proof}

\begin{proof}[Proof of Lemma   \ref{lemma:control_derivative}]
For Item \ref{item:control_1st}, the linearity of $ \frac{\delta \alpha^{t,x}}{\delta \phi_i}(\phi;\cdot)$ follows from
\eqref{eq:control_1st} and the linearity of  $ \frac{\delta X^{t,x}}{\delta \phi_i}(\phi;\cdot)$ (see Lemma \ref{lemma:1st_derivative_state}).
For each $\eps\in [0,1)$, let $\phi^\eps=(\phi_{i}+\eps(\phi'_i-\phi_i),\phi_{-i})$ and observe that 
$$
\alpha^{t,x,\phi^\eps}_s=\phi^\eps (s, X^{t,x,\phi^\eps}_s)=  (\phi_{i}+\eps(\phi'_i-\phi_i),\phi_{-i})(s, X^{t,x,\phi^\eps}_s).
$$
Hence    \cite[Theorem 9, p.~97]{krylov2008controlled}, the differentiability of $\phi$ and $\phi'_i$
and the $\cS^\infty$-differentiability of $(X^{t,x,\psi^\eps})_{\eps\in [0,1)}$ imply the existence and the desired formula of   
 $\cH^\infty \text-\p_\eps \alpha^{t,x, (\phi_i+\eps (\phi_i'-\phi_i),\phi_{-i})}\big\vert_{\eps=0}$.

 For Item \ref{item:control_2nd}, the bilinearity and symmetry of $ \frac{\delta^2 \alpha^{t,x}}{\delta \phi_i\delta\phi_j}(\phi;\cdot,\cdot)$ follow directly from
\eqref{eq:control_2nd} and     Lemmas  \ref{lemma:1st_derivative_state}  and \ref{lemma:2nd_derivative_state}. 
 For each $\eps\in [0,1)$, let $\phi^\eps=(\phi_{j}+\eps(\phi'_j-\phi_j),\phi_{-j})$ and  
  \begin{align}
  \frac{\delta \alpha^{t,x}}{\delta \phi_i}(\phi^\eps;\phi'_i)_s
  & =(\p_x \phi^\eps)(s,X^{t,x,\phi^\eps}_s)  \frac{\delta X^{t,x}}{\delta \phi_i}(\phi^\eps;\phi'_i)_s+E_i\phi'_i(s, X^{t,x,\phi^\eps}_s),
 \end{align}
 Applying     \cite[Theorem 9, p.~97]{krylov2008controlled} again yields the   expression of 
$\cH^\infty \text-\p_\eps 
\frac{\delta \alpha^{t,x}}{\delta \phi_i}((\phi_j+\eps (\phi''_j-\phi_j),\phi_{-j});\phi'_i)
\big\vert_{\eps=0} $. 
\end{proof}

 \end{document}